\documentclass{amsart}

\usepackage[T1]{fontenc}
\usepackage{lmodern}

\usepackage{amscd, amssymb, amsmath, amsthm}
\usepackage{stmaryrd} 
\usepackage{enumerate, latexsym, mathrsfs}
\usepackage{xypic}
\usepackage[pdftex]{geometry}
\usepackage[
		bookmarks=true, bookmarksopen=true,%
    bookmarksdepth=3,bookmarksopenlevel=2,%
    colorlinks=true,%
    linkcolor=blue,%
    citecolor=blue]{hyperref}
\newtheorem{theorem}{Theorem}[section]
\newtheorem{lemma}[theorem]{Lemma}
\newtheorem{proposition}[theorem]{Proposition}
\newtheorem{corollary}[theorem]{Corollary}
\newtheorem{conjecture}[theorem]{Conjecture}
\newtheorem{hypothesis}[theorem]{Hypothesis}

\theoremstyle{definition}
\newtheorem{definition}[theorem]{Definition}
\newtheorem{convention}[theorem]{Convention}

\newtheorem{example}[theorem]{Example}
\newtheorem{question}[theorem]{Question}
\newtheorem{remark}[theorem]{Remark}

\numberwithin{equation}{section}

\newcommand{\Natural}{{\mathbb N}}
\newcommand{\Real}{{\mathbb R}}
\newcommand{\Rational}{{\mathbb Q}}
\newcommand{\Complex}{{\mathbb C}}
\newcommand{\Integral}{{\mathbb Z}}

\newcommand{\Sph}{{\mathbb S}}
\newcommand{\Hyp}{{\mathbb H}}

\newcommand{\univ}{{\mathrm{univ}}}

\newcommand{\algcl}{{\mathrm{ac}}}
\newcommand{\stable}{{\mathtt{st}}}

\newcommand{\ud}{{\mathrm{d}}}

\title[{Finite quotients, arithmetic invariants, and hyperbolic volume}]{
Finite quotients, arithmetic invariants,\\ and hyperbolic volume}
     
\author[Yi Liu]{%
        Yi Liu} 
\address{%
        Beijing International Center for Mathematical Research, Peking University\\
				Beijing 100871, China P.R.} 
\email{%
    liuyi@bicmr.pku.edu.cn}

\thanks{Partially supported by NSFC Grant 11925101, 
and National Key R\&D Program of China 2020YFA0712800}

\date{%
 \today} 

\begin{document}

\begin{abstract}
	For any pair of orientable closed hyperbolic $3$--manifolds,
	this paper shows that any isomorphism between
	the profinite completions of their fundamental groups
	witnesses a bijective correspondence between 
	the Zariski dense $\mathrm{PSL}(2,\mathbb{Q}^{\mathtt{ac}})$--representations 
	of their fundamental groups, up to conjugacy;
	moreover, corresponding pairs of representations
	have identical invariant trace fields and isomorphic invariant quaternion algebras.
	(Here, $\mathbb{Q}^{\mathtt{ac}}$ denotes an algebraic closure of $\mathbb{Q}$.)
	Next,
	assuming the $p$--adic Borel regulator injectivity conjecture for number fields,
	this paper shows that
	uniform lattices in $\mathrm{PSL}(2,\mathbb{C})$
	with isomorphic profinite completions
	have identical invariant trace fields,
	isomorphic invariant quaternion algebras,
	identical covolume, and identical arithmeticity.
\end{abstract}

\maketitle

\section{Introduction}

Hyperbolic $3$--manifolds of finite volume
have abundant finite covers, and have strong rigidity. 
The latter fact opens the door for studying fundamental groups of those manifolds
by means of geometry and arithmetics.
The hyperbolic volume, the invariant trace field,
and the invariant quaternion algebra are 
among the most important invariants,
not only in theory, but also in practical computation.
Arithmetic hyperbolic $3$--manifolds and non-arithmetic ones
exhibit contrast behaviors in many group theoretic aspects.
The former fact, in addition,
makes it reasonable to ask 
if the fundamental group could possibly be inferred
from knowledge of its finite quotients.
This challenging question
has attracted growing interest over the past decade.
Its pertaining topic
is called profinite properties of $3$--manifold groups.

For any (abstract) group $G$, 
the \emph{profinite completion} of $G$ refers to 
the group $\widehat{G}=\varprojlim_{N} G/N,$ 
where $N$ ranges over the inverse system of all finite-index normal subgroups of $G$.
When $G$ is finitely generated and residually finite,
the completion homomorphism $G\to\widehat{G}$ is injective,
and the profinite topology on $\widehat{G}$
can be characterized purely in abstract group theoretic terms \cite{Nikolov--Segal};
moreover, 
the group isomorphism type of $\widehat{G}$
can be recovered purely from
the set of isomorphism classes of all the finite quotient groups of $G$
\cite[Theorem 2.2]{Reid_survey}.
These facts make finitely generated residually finite groups
arguably the most comfortable class for speaking of profinite properties.
There are some smaller classes of groups, restricted to which
one might expect richer harvest.
Finitely generated virtually nilpotent groups, 
semisimple Lie group lattices, 
and finitely generated $3$--manifold groups 
are a few candidates on top of our mind.

Some striking examples of arithmetic lattices in $\mathrm{PSL}(2,\Complex)$
and $\mathrm{O}^+(3,1)$ with absolute profinite rigidity
have been confirmed recently 
by M.~R.~Bridson, D.~B.~McReynolds, A.~W.~Reid, and R.~Spitler \cite{BMRS}.
Their examples include the Bianchi group 
$\mathrm{PSL}(2,\Integral[(-1+\sqrt{-3})/2])$,
and the fundamental group of the Fomenko--Matveev--Weeks manifold, 
which is the unique closed orientable hyperbolic $3$--manifold of the smallest volume.
Their work shows that the isomorphism types of those groups 
are uniquely determined by the isomorphism types of their profinite completions,
among finitely generated, residually finite groups.
Restricted to the class of finitely generated $3$--manifold groups,
(which can be characterized as
fundamental groups of connected compact $3$--manifolds,)
there are examples with profinite rigidity \cite{BMRS,BRW,Bridson--Reid},
and also examples without profinite rigidity \cite{Funar_torus_bundles,Hempel_quotient,Stebe_integer_matrices}.
However, all known examples satisfy a weaker condition,
which we call \emph{profinite almost rigidity}.
This condition says that,
among finitely generated $3$--manifold groups,
the isomorphism type of the group
is determined up to finitely many possibilities
 by the isomorphism of the profinite completion.
G.~Wilkes shows that graph manifold groups are profinitely almost rigid \cite{Wilkes_sf,Wilkes_graph,Wilkes_graph_II}.
In a recent work \cite{Liu_profinite_almost_rigidity},
the author shows that finite-volume hyperbolic $3$--manifold groups are profinitely almost rigid.
It seems plausible that finitely generated $3$--manifold groups are all profinitely almost rigid.
On the other hand, 
we recall the following question raised by A.~W.~Reid
\cite[Question 9]{Reid_discrete},
which remains widely open.
Its torsion-free case is equivalent to
the (unknown) profinite rigidity 
among orientable finite-volume hyperbolic $3$--manifolds.

\begin{question}\label{lattice_profinite_rigidity}
	If a pair of lattices in $\mathrm{PSL}(2,\Complex)$ have isomorphic profinite completions,
	must they be isomorphic to each other?
\end{question}

The present paper may be viewed as a continuation of \cite{Liu_profinite_almost_rigidity}
and an attempt toward Question \ref{lattice_profinite_rigidity}.
We seek for a connection between 
dynamics of pseudo-Anosov suspension flow
and arithmetic invariants of $\mathrm{SL}(2,\Complex)$--representations.
In the end,
we are able to use this connection,
and reduce the profinite invariance problem of hyperbolic volume and arithmeticity
to a long-standing conjecture of number theoretic nature,
that is, 
the $p$--adic Borel regulator injectivity conjecture for number fields.

For any number field $k$ with $r_1$ real places and $r_2$ complex places,
the zeroth K-group $K_0(O_k)$ of the ring of integers $O_k$ 
can be identified with the ideal class group of $O_k$ direct sum with $\Integral$,
and the first K-group $K_1(O_k)$ can be identified with 
the multiplicative unit group $O_k^\times$.
In particular, the real dimension of 
$K_n(O_k)\otimes_\Integral\Real$ 
equals $1$ for $n=0$, by the class number finiteness,
and equals $r_1+r_2-1$ for $n=1$, by the Dirichlet unit theorem.
A.~Borel shows that the rank of higher K-groups obeys periodicity,
and indeed, when $n\geq2$,
the real dimension of 
$K_n(k)\otimes_\Integral\Real\cong K_n(O_k)\otimes_\Integral\Real$ equals $0,r_1+r_2,0,r_2$
for $n\equiv 0,1,2,3$ modulo $4$, respectively.
For the nonvanishing dimensions, the classical Borel regulator maps
associated to $k$ can be reinterpreted as $\Real$--linear homomorphisms
$K_{4m+1}(k)\otimes_\Integral\Real\to\Real^{r_1}\oplus\Complex^{r_2}$
and $K_{4m-1}(k)\otimes_\Integral\Real\to\Complex^{r_2}$,
which are actually injective onto $\Real^{r_1}\oplus\Real^{r_2}$ and $\Real^{r_2}$,
respectively.
For any rational prime $p$,
the $p$--adic Borel regulator maps associated to $k$
can be defined analogously,
and can be thought of as
$\Rational_p$--linear homomorphisms 
$K_{2m-1}(k)\otimes_\Integral\Rational_p\to \bigoplus_{v|p} k_v$.
We state the \emph{$p$--adic Borel regulator injectivity conjecture} as follows.
The conjecture is expected to be true to the generality as stated.
(See Section \ref{Sec-prelim_I} for more detail.)

\begin{conjecture}\label{Borel_hypothesis_p_EZ}
	For any number field $k$ and any rational prime $p$,
	the $p$--adic Borel regulator map 
	is injective on 
	$K_{2m-1}(k)\otimes_\Integral \Rational_p$
	for every odd dimension $2m-1\geq3$. 
\end{conjecture}

\begin{theorem}\label{main_vol_EZ}
	If Conjecture \ref{Borel_hypothesis_p_EZ} holds true,
	then,
	among uniform lattices in $\mathrm{PSL}(2,\Complex)$,
	the following lattice invariants
	are all determined by the group isomorphism type of the profinite completion:
	the invariant trace field in $\Complex$, 
	the conjugacy class of the invariant quaternion algebra in $\mathrm{Mat}(2,\Complex)$,
	the	covolume,	and arithmeticity.
\end{theorem}

See Theorem \ref{main_profinite_invariance_vol} in Section \ref{Sec-main_theorems} for the precise statement.
The proof of Theorem \ref{main_vol_EZ} actually only needs the injectivity assumption
on $K_3(k)\otimes_\Integral\Rational_p$ and for infinitely many $p$
(Hypothesis \ref{Borel_hypothesis_p}).
It should be possible to extend Theorem \ref{main_vol_EZ} to the non-uniform case,
if one invokes works of H.~Wilton and P.~A.~Zalesskii \cite{WZ_geometry,WZ_decomposition},
and keeps track of the parabolic subgroups more closely.
However, we do not pursue the generalization in this paper, 
as the conclusion of Theorem \ref{main_vol_EZ} is already hypothetical.
Note that Theorem \ref{main_vol_EZ} can be derived immediately from its torsion-free case,
so it is essentially about orientable closed hyperbolic $3$--manifolds
with profinitely isomorphic fundamental groups.

\begin{theorem}\label{main_rep_EZ}
	If a pair of orientable closed hyperbolic $3$--manifolds
	have profinitely isomorphic fundamental groups,
	then
	the Zariski dense $\mathrm{PSL}(2,\Rational^\algcl)$--representations
	of their fundamental groups
	can be paired up bijectively up to conjugacy.
	Moreover, corresponding pairs of representations
	have identical invariant trace fields in $\Rational^\algcl$,
	and
	have conjugate invariant quaternion algebras in $\mathrm{Mat}(2,\Rational^\algcl)$.
\end{theorem}

The notation $\Rational^\algcl$ means an algebraic closure of $\Rational$.
Note that Theorem \ref{main_rep_EZ} does not rely on Conjecture \ref{Borel_hypothesis_p_EZ}.
Theorem \ref{main_rep_EZ} is a main step toward Theorem \ref{main_vol_EZ}.
See Theorem \ref{main_profinite_correspondence_PSL} in Section \ref{Sec-main_theorems}
for the precise statement;
an effective description of the correspondence
can be found in the beginning paragraph of the proof of Theorem \ref{main_profinite_correspondence_PSL}.

Theorem \ref{main_rep_EZ} is inspired by \cite[Theorem 4.8]{BMRS}.
The Zariski denseness assumption is needed for legitimating the invariant quaternion algebra.
What we are able to get rid of is the Galois rigidity assumption from \cite[Theorem 4.8]{BMRS}.
The basic reason is that we employ a bijective correspondence between periodic trajectories
for profinitely isomorphic pseudo-Anosov mapping tori,
as established in \cite{Liu_profinite_almost_rigidity}.

As one should be warned,
Theorem \ref{main_rep_EZ} does not tell whether
the discrete faithful representations (or their Galois conjugates) 
are paired up by the asserted correspondence.
Nor does \cite[Theorem 4.8]{BMRS}.
In fact, the desired injectivity in \cite{BMRS} is eventually obtained 
using the Hopfian property of topologically finitely generated profinite groups,
and moreover, the argument thereof
appeals to concrete features of particular lattices, 
such as lists of their small index subgroups.
In the proof of Theorem \ref{main_vol_EZ}, 
we are only able to infer from Conjecture \ref{Borel_hypothesis_p_EZ}
that discrete faithfulness should be preserved under the correspondence,
and this is a consequence of the volume rigidity 
for $\mathrm{SL}(2,\Complex)$--representations of 
uniform $\mathrm{PSL}(2,\Complex)$--lattices
(see Example \ref{hyp_manifold_rep_vol}).

\subsection*{Ingredients}
In the rest of the introduction,
we explain our strategy for proving Theorems \ref{main_vol_EZ} and \ref{main_rep_EZ}.
We also point out some new ideas of the present paper. 

Let $\pi_A$ and $\pi_B$ be a pair of finitely generated, residually finite groups.
Suppose that $\Psi\colon\widehat{\pi}_A\to\widehat{\pi}_B$ is an isomorphism 
between their profinite completions.
In order to set up a correspondence of representations 
$\pi_B\to\mathrm{PSL}(2,\Rational^\algcl)$ and $\pi_A\to\mathrm{PSL}(2,\Rational^\algcl)$
up to conjugacy, one would naturally look for 
certain kind of pullback operation with respect to $\Psi$.
Assuming Zariski denseness of the representations will allow us to speak of 
the invariant quaternion algebras (of the images), 
and will also make it convenient for passing to finite-index subgroups.

For simplicity, let us consider 
a Zariski dense $\mathrm{PSL}(2,\Rational^\algcl)$--representation of $\pi_B$ 
that comes from a representation $\rho_B\colon \pi_B\to\mathrm{SL}(2,K)$,
where $K$ is subfield finite over $\Rational$.
For all but finitely many non-Archimedean places $v$ of $K$,
the induced representation $\rho_{B,v}\colon \pi_B\to\mathrm{SL}(2,K_v)$
is bounded.
For any such $v$, 
we obtain a representation $\rho_{A,v}^{\Psi}\colon\pi_A\to\mathrm{SL}(2,K_v)$,
as the $\Psi$--pullback of the continuous extension 
$\rho_{B,v}\colon\widehat{\pi}_B\to\mathrm{SL}(2,K_v)$
to $\widehat{\pi}_A$ and restricted to $\pi_A$.
The same construction has already appeared in the proof of \cite[Theorem 4.8]{BMRS}.
However, we do not automatically know
whether the characters of $\rho_{A,v}^{\Psi}$ 
are valued in $\Rational^\algcl$, 
or whether the characters are the same for various $v$. 
Certainly one would hope they are,
and then, 
there will be some representation $\rho_A\colon \pi_A\to\mathrm{SL}(2,\Rational^\algcl)$
with the same character.
If so,
$\rho_A$ should be the right candidate for corresponding to $\rho_B$,
but how could we possibly read off the invariant quaternion algebra of $\rho_A$
from the construction,
and especially,
its ramification behavior over the Archimedean places
of the invariant trace field?

We are able to address the above issues in the case of closed hyperbolic $3$--manifold groups
$\pi_{A/B}=\pi_1(M_{A/B})$. (The notation $A/B$ means $A$ or $B$ respectively.)
We may assume $M_{A/B}$ fibering over a circle 
as a surface bundle with pseudo-Anosov monodromy,
because of the virtual fibering theorem \cite{Agol_VHC},
(see also \cite{AFW_book_group} for a survey 
on virtual properties of $3$--manifold groups).
It is proved in \cite{Liu_profinite_almost_rigidity} that $\Psi$
witnesses a bijective correspondence of the surface fiber subgroups;
and moreover,
fixing a corresponding pair of surface fibers $S_A$ and $S_B$,
$\Psi$ witnesses a bijective correspondence between the periodic trajectories
of pseudo-Anosov suspension flows on $M_A$ and on $M_B$.
More precisely,
there exists an invertible profinite integer 
$\mu\in\widehat{\Integral}^\times$,
which depends only on $(\pi_A,\pi_B,\Psi)$.
For any corresponding pair of periodic trajectories $\mathbf{c}_A$ and $\mathbf{c}_B$
(as conjugacy classes in $\pi_A$ and $\pi_B$),
the $\mu$--power of $\mathbf{c}_A$ projects $\mathbf{c}_B$ under $\Psi$
(as conjugacy classes in $\widehat{\pi}_A$ and $\widehat{\pi}_B$).
It is this trajectory correspondence that 
allows us to compare the characters of $\rho_A$ and $\rho_B$
on \emph{some} pairs of conjugacy classes in $\pi_A$ and $\pi_B$.
We use this partial information to infer about the characters of $\rho_{A,v}$,
and about the invariant quaternion algebra of $\rho_A$.

In order to compute the characters of $\rho_{A,v}^\Psi$
on the periodic trajectories of $M_A$ from data with $\rho_B$ and $M_B$,
we improve a main result regarding $\widehat{\Integral}^\times$--regularity 
in \cite{Liu_profinite_almost_rigidity}.
We show that $\mu$ has multiplicative order $2$ (Lemma \ref{p_regularity}).
Writing $\mu=(\mu_p)_p$ with respect to $\widehat{\Integral}^\times\cong\prod_p\Integral_p^\times$,
this means $\mu_p\in\{\pm1\}$ for each rational prime $p$.
The topological ingredient for the improvement comes from
the resolution of McMullen's conjecture on virtual homological spectral radii
of pseudo-Anosov automorphisms \cite{Liu_vhsr}.
That conjecture asserts that every pseudo-Anosov automorphism $f\colon S\to S$
admits a lift $f'\colon S'\to S'$ to some finite cover $S'\to S$,
such that $f'_*\colon H_1(S';\Complex)\to H_1(S';\Complex)$ has a complex eigenvalue
outside the unit circle.
As a consequence, we can identify certain twisted Alexander polynomials of $M_A$ and of $M_B$ 
(associated to some $\Psi$--corresponding finite quotients
$\pi_A\to\Gamma$ and $\pi_B\to \Gamma$),
such that the polynomials are monic and reciprocal over $\Integral$,
but are not product of cyclotomic factors.
With these polynomials,
we show that $\mu_p$ is rational for each $p$, 
using Mahler's $p$--adic analogue of the Gelfond--Schneider theorem \cite{Mahler_GS}.
Then we deduce $\mu_p\in\{\pm1\}$ by simple observation.
(See Sections \ref{Sec-profin_power_operation} and \ref{Sec-mu_square}.)

In order to determine the invariant quaternion algebra of $\rho_A$,
we introduce two new criteria,
in terms of the character $\chi_A$ of $\rho_A$
evaluated at the periodic trajectories of $M_A$.
The first criterion says that the trace field of $\rho_A$ 
is generated over $\Rational$ 
by the values of $\chi_A$ on the periodic trajectories (Lemma \ref{trace_field_pt}).
The second criterion says that 
the quaternion algebra of $\rho_A$ at a real place of the trace field 
is ramified or unramified,
according as the real values of $\chi_A$ on the periodic trajectories
are uniformly bounded or unbounded (Lemma \ref{quaternion_algebra_pt}).
(See Section \ref{Sec-traj_arithm}.)
We can identify the invariant trace fields of $\rho_A$ and $\rho_B$
with the first criterion (applied restricting to finite-index subgroups),
and identify the Archimedean places of ramifications with the second criterion,
and identify the non-Archimedean places of ramification 
easily from the construction. 
Since quaternion algebras over number fields are isomorphically classified
by their places of ramification,
we can identify the invariant trace field of $\rho_A$ 
isomorphically with that of $\rho_B$.
(See Section \ref{Sec-profin_corresp_rep}.)

The above discussion covers all the essential points 
for proving Theorem \ref{main_rep_EZ}.
As to Theorem \ref{main_vol_EZ}, 
essentially we need to show 
that the induced homomorphisms
$\rho_{A*}\colon H_3(\pi_{A};\Integral)\to H_3(\mathrm{SL}(K);\Integral)\otimes_\Integral\Rational$
and
$\rho_{B*}\colon H_3(\pi_{B};\Integral)\to H_3(\mathrm{SL}(K);\Integral)\otimes_\Integral\Rational$
map the fundamental classes $[M_A]$ and $[M_B]$
to the same group homology class up to sign.
(There are natural isomorphisms
$K_3(K)\otimes_\Integral\Rational\cong
H_3(\mathrm{SL}(K);\Integral)\otimes_\Integral\Rational\cong
H_3(\mathrm{SL}(N,K);\Integral)\otimes_\Integral\Rational$ for sufficiently large $N$.)
Then Borel's theorem will tell us that $M_A$ and $M_B$ have identical hyperbolic volume,
and the rest assertions of Theorem \ref{main_vol_EZ} will become easy consequences 
of Theorem \ref{main_rep_EZ}.
With the profinite isomorphism $\Psi$,
it is only obvious to identify $[M_A]$ with $\pm[M_B]$
in the continuous group homology (also known as the profinite group homology)
$H_3^{\mathtt{cont}}(\widehat{\pi}_A;\Integral_p)\cong H_3^{\mathtt{cont}}(\widehat{\pi}_B;\Integral_p)$.
This forces us, in the present paper, to work with
suitable versions of structured topological group homology or cohomology,
(for example,
the locally analytic group cohomology of $\mathrm{GL}(N,O_{\mathfrak{P}})$
where $O_{\mathfrak{P}}$ is the completion of the ring of integers in $K$
at some prime $\mathfrak{P}$).
We have to transfer information back to the abstract group homology at some point.
It turns out that Conjecture \ref{Borel_hypothesis_p_EZ} (or Hypothesis \ref{Borel_hypothesis_p})
is exactly the missing piece for completing the whole argument.
(See Section \ref{Sec-profin_invar_vol}.)

\subsection*{Organization}
The preliminary of this paper is divided into two sections.
The first part (Section \ref{Sec-prelim_I}) contains a quick review of Conjecture \ref{Borel_hypothesis_p_EZ}.
We also include necessary background regarding volume of $\mathrm{SL}(2,\Complex)$--representations,
but the materials in Section \ref{Sec-prelim_I} are not needed until Section \ref{Sec-profin_invar_vol}.
The second part (Section \ref{Sec-prelim_II}) is directly related to the next few sections.
We review general facts about quaternion algebras and pseudo-Anosov automorphism dynamics.

In Section \ref{Sec-traj_arithm}, we prove the two criteria regarding periodic trajectories in a pseudo-Anosov mapping torus
and arithmetic invariants associated to an irreducible $\mathrm{SL}(2,\Complex)$--representation.
In Section \ref{Sec-profin_corresp_traj},
we recall the profinite correspondence of periodic trajectories,
and extract a precise statement from results in \cite{Liu_profinite_almost_rigidity}.
In Section \ref{Sec-profin_power_operation},
we provide some explicit description about the profininte power operation 
on profinite Laurent polynomial rings and profinite matrix groups.
In Section \ref{Sec-mu_square}, 
we prove our improvement $\mu^2=1$ based on algebraic lemmas in Section \ref{Sec-profin_power_operation}.

In Section \ref{Sec-rep_qa},
we develop some {\it ad hoc} theory regarding group representations in quaternion algebras,
in order to facilitate subsequent disposal.
In Section \ref{Sec-profin_corresp_rep},
we prove an essential case of Theorem \ref{main_rep_EZ},
that is, the case with pseudo-Anosov mapping tori.
In Section \ref{Sec-profin_invar_vol},
we prove the same essential case of Theorem \ref{main_vol_EZ}.
In Section \ref{Sec-main_theorems},
we deduce the general cases of 
Theorem \ref{main_rep_EZ} (formally stated as Theorem \ref{main_profinite_correspondence_PSL})
and Theorem \ref{main_vol_EZ} (formally stated as Theorem \ref{main_profinite_invariance_vol})
from the above essential cases.
and we actually restate the theorems formally,
respectively.

\subsection*{Acknowledgement} 
The author thanks Ian Agol, Frank Calegari, and Shengkui Ye for valuable communications.
The author also thanks the anonymous referee for careful proofreading 
and pointing out missing arguments in a preliminary version.

\section{Preliminary I}\label{Sec-prelim_I}
In this preliminary section, 
we review volume of representations in $\mathrm{SL}(2,\Complex)$
and its relation with the classical Borel regulator. 
We review the $p$--adic Borel regulator injectivity conjecture (Conjecture \ref{Borel_hypothesis_p_EZ})
in a special case as Hypothesis \ref{Borel_hypothesis_p}.
For convenience of our application,
we reformulate the materials
on the level of suitable versions of group homology and cohomology,
and include the algebraic K-theoretic versions in 
Remarks \ref{Borel_theorem_infinity_remark} and \ref{Borel_hypothesis_p_remark}.
We could have placed this section right before Section \ref{Sec-profin_invar_vol}
without affecting the logical order,
but we decide to place it here as background of Conjecture \ref{Borel_hypothesis_p_EZ}.

\subsection{Volume of representations in $\mathrm{SL}(2,\Complex)$}
	Let $M$ be an oriented connected closed (smooth) $3$--manifold.
	Fixing a universal covering space $\widetilde{M}_\univ$ of $M$, 
	we denote by $\pi=\pi_1(M)$ the universal deck transformation group,
	which we also frequently refer to as the fundamental group of $M$.
	For any group homomorphism $\rho\colon \pi\to \mathrm{SL}(2,\Complex)$,
	there is a well-defined value
	$\mathrm{Vol}_{\mathrm{SL}(2,\Complex)}\left(M,\rho\right)\in\Real$,
	called the \emph{volume} of $M$ with respect to 
	the $\mathrm{SL}(2,\Complex)$--representation $\rho$.
	We recall a description of $\mathrm{Vol}_{\mathrm{SL}(2,\Complex)}(M,\rho)$
	as follows.
	
	Denote by
	$\Hyp^3=\{(x+\mathbf{i}y,h)\in\Complex\times\Real\colon h>0\}$
	the upper half space model of $3$--dimensional hyperbolic geometry.
	The model is furnished with the Riemannian metric $(\ud x^2+\ud y^2 +\ud h^2)^{1/2}/h$,
	which has constant sectional curvature $-1$.
	The orientation-preserving isometry group $\mathrm{Isom}_+(\Hyp^3)$
	is identified with $\mathrm{PSL}(2,\Complex)$,
	agreeing with the natural identification 
	of the sphere at infinity $\partial_{\infty}\Hyp^3$ 
	with the complex projective line
	$\Complex\mathbf{P}^1=\Complex\cup\{\infty\}$.
	If $D_\rho\colon\widetilde{M}_\univ\to \Hyp^3$
	is a $(\pi,\mathrm{PSL}(2,\Complex))$--equivariant
	(smooth) map
	with respect to the projectivized representation 
	$\mathbf{P}(\rho)\colon\pi\to \mathrm{PSL}(2,\Complex)$, 
	the volume of $(M,\rho)$ can be expressed with an integral
	\begin{equation}\label{volume_D_rho}
	\mathrm{Vol}_{\mathrm{SL}(2,\Complex)}\left(M,\rho\right)=
	\int_{\mathcal{F}} D^*_\rho\,\ud\mathrm{Vol}_{\Hyp^3}.
	\end{equation}
	Here, $\mathcal{F}\subset\widetilde{M}_\univ$ is a fundamental domain of $\pi$,
	and where $\ud\mathrm{Vol}_{\Hyp^3}=h^{-3}\ud x\wedge\ud y\wedge \ud h$ 
	denotes the volume form on $\Hyp^3$.
	
	Any map $D_\rho$ as above is called a \emph{developing map} for $\rho$.
	Since $\Hyp^3$ is contractible, developing maps always exist,
	and the resulting volume value 
	does not depend on particular choices of $\mathcal{F}$ or $D_\rho$.
	Moreover, the volume of $(M,\rho)$ is invariant under continuous deformations
	of $\rho$. In particular, it is $\mathrm{SL}(2,\Complex)$--conjugation invariant in $\rho$.
	It takes only finitely many values as $\rho$ ranges over all the representations
	$\pi\to\mathrm{SL}(2,\Complex)$.
	We refer the reader to \cite{DLSW_rep_vol} 
	for an exposition of volume of representations in a general setting,
	and for further references to many important results in the literature.
	
	\begin{example}\label{hyp_manifold_rep_vol}
	Let $M$ be an oriented closed hyperbolic $3$--manifold.
	For any representation $\rho\colon\pi_1(M)\to \mathrm{SL}(2,\Complex)$,
	the following inequality holds
	$$\left|\mathrm{Vol}_{\mathrm{SL}(2,\Complex)}(M,\rho)\right|\leq \mathrm{Vol}_{\Hyp^3}(M).$$
	Moreover, the equality is achieved if and only if
	$\rho$ is a discrete faithful representation.
	See	\cite[Theorem 1.1]{Kim--Kim_volume_maximal} or \cite[Theorem 1.4]{BCG_Milnor_Wood}.
	\end{example}
	
	\begin{remark}\label{PSL_rep_vol_remark}	
	It is also possible to define the volume of any $\mathrm{PSL}(2,\Complex)$--representation
	of $\pi_1(M)$.
	For any representation $\rho\colon \pi_1(M)\to \mathrm{SL}(2,\Complex)$,
	the volume $\mathrm{Vol}_{\mathrm{PSL}(2,\Complex)}(M,\mathbf{P}(\rho))$
	of the projectivized representation
	$\mathbf{P}(\rho)\colon \pi_1(M)\to \mathrm{PSL}(2,\Complex)$
	is equal to $\mathrm{Vol}_{\mathrm{SL}(2,\Complex)}(M,\rho)$,
	by definition (see \cite[Section 2]{DLSW_rep_vol}).
	When $M$ is hyperbolic,	$\rho$ is discrete and faithful	if and only if
	$\mathbf{P}(\rho)$ is discrete and faithful.
	In this case, $\mathbf{P}(\rho)$ is unique up 
	to a real Lie group automorphism of $\mathrm{PSL}(2,\Complex)$.
	It follows that $\rho$ is unique up to an inner automorphism,
	and a complex conjugation, and a scalar representation $\pi_1(M)\to\{\pm1\}$.
	It is also known that any discrete faithful representation 
	of $\pi_1(M)$ in $\mathrm{PSL}(2,\Complex)$
	lifts to a representation in $\mathrm{SL}(2,\Complex)$ \cite{Culler_lift}.
	In general, a representation $\rho^\flat\colon\pi_1(M)\to\mathrm{PSL}(2,\Complex)$
	is not necessarily liftable to $\mathrm{SL}(2,\Complex)$.
	The obstruction can be characterized as some second Stiefel--Whitney class
	$w_2(\rho^{\flat})\in H^2(\pi_1(M);\Integral/2\Integral)$;
	in the liftable case, the $\mathrm{SL}(2,\Complex)$--representation
	lifts of $\rho^\flat$ differ from each other
	by a scalar representation factor $\varepsilon\colon \pi_1(M)\to \{\pm1\}$,
	so they form an affine $H^1(\pi_1(M);\Integral/2\Integral)$.
	(See \cite[Remark 4.1]{Heusener--Porti_PSL}; compare \cite[Proposition 5.1]{DLSW_rep_vol}.)
	\end{remark}
	
	
	\subsection{Borel regulators: from classical to $p$--adic}
	Volume of $\mathrm{SL}(2,\Complex)$--representations can be reinterpreted
	in terms of the (second) complex Borel regulator.
	We briefly recall some relevant facts, 
	and explain a precise formula (\ref{volume_Borel_complex}).
	We recall analogous Borel regulators in the $p$--adic setting,
	following Huber and Kings \cite{Huber--Kings_regulator}.	
				
	For every dimensions $n$,
	the $n$--th differentiable group cohomology 
	$H^n_{\mathtt{diff}}(\mathrm{GL}(N,\Complex);\Complex)$
	of the complex Lie group $\mathrm{GL}(N,\Complex)$
	stabilizes for all sufficiently large $N$,
	with respect to the natural inclusions
	$\mathrm{GL}(1,\Complex)\to \mathrm{GL}(2,\Complex)\to \cdots
	\to \mathrm{GL}(N,\Complex)\to \cdots$,
	(inserting extra diagonal entries $1$ at the lower right corner
	of square matrices).
	In fact,
	via the van Est isomorphism \cite[Chapter X, Corollary 5.6]{Borel--Wallach_book},
	$H^n_{\mathtt{diff}}(\mathrm{GL}(N,\Complex);\Complex)$
	identifies
	with the complex Lie algebra cohomology
	$H^n(\mathfrak{gl}(N,\Complex),\mathrm{U}(N);\Complex)
	\cong H^n(\mathfrak{gl}(N,\Complex);\Complex)$ for all $N$.
	The cohomology ring 
	$H^*(\mathfrak{gl}(N,\Complex);\Complex)$ is
	an exterior $\Complex$--algebra on 
	canonical primitive elements 
	$p_{\Complex,m}\in H^{2m-1}(\mathfrak{gl}(N,\Complex);\Complex)$,
	where $m$ ranges in $\{1,2,\cdots,N\}$,
	(see Remark \ref{normalization_remark} below 
	about the normalization convention).
	
	\begin{remark}\label{normalization_remark}
		The canonical primitive elements $p_{\Complex,m}$ 
		may be chosen differently by different authors.
		Although not quite important for our discussion,
		we fix our normalization in this paper
		following Huber and Kings \cite[Definition 0.4.5]{Huber--Kings_regulator}.
		For any field $F$ of characteristic $0$
		and any index $m\geq1$,
		we fix the canonical class	$p_{F,m}\in H^{2m-1}(\mathfrak{gl}(N,F);F)$
		as represented by the alternating $F$--multilinear form
		$$p_{F,m}(X_1,X_2\cdots, X_{2m-1})=\frac{((m-1)!)^2}{(2m-1)!}\cdot
		\sum_{\sigma}\mathrm{sgn}(\sigma)\cdot\mathrm{tr}\left(X_{\sigma(1)}\cdot X_{\sigma(2)}\cdot\cdots\cdot X_{\sigma(2m-1)}\right),$$
		for all $X_1,X_2,\cdots,X_{2m-1}$ in $\mathfrak{gl}(N,F)=\mathrm{Mat}(N,F)$,
		where $\sigma$ ranges over all permutations on the indices $\{1,2,\cdots,2m-1\}$.
		Note that these forms all come from canonical forms over $\Rational$ on by scalar extension.
	\end{remark}
			
	In this paper, we declare
	the $m$--th \emph{complex Borel regulator} 
	as the canonical differentiable group cohomology class
	$p_{F,m}\in H^{2m-1}_{\mathtt{diff}}(\mathrm{GL}(N,\Complex);\Complex)$,
	for any $N$ in the stable range.
	One may equavalently declare 
	the $m$--th complex Borel regulator
	as a continuous group cohomology class,
	because the natural homomorphism
	$H^*_{\mathtt{diff}}(\mathrm{GL}(N,\Complex);\Complex)\to
	H^*_{\mathtt{cont}}(\mathrm{GL}(N,\Complex);\Complex)$
	is isomorphic \cite[Chapter IX, Section 5]{Borel--Wallach_book}.
		
	\begin{example}\label{complex_Borel_regulator_example}\
	\begin{enumerate}
	\item
	The first complex Borel regulator can be characterized as 
	an abelian group homomorphism $H_1(\mathrm{GL}(N,\Complex);\Integral)\to\Complex$.
	One may identify the (abstract) group homology
	$H_1(\mathrm{GL}(N,\Complex);\Integral)$
	with the abelian group
	$\mathrm{GL}(N,\Complex)/[\mathrm{GL}(N,\Complex),\mathrm{GL}(N,\Complex)]
	\cong \mathrm{GL}(N,\Complex)/\mathrm{SL}(N,\Complex)
	\cong	\Complex^\times$.
	By simple computation, the first complex Borel regulator 
	is explicitly twice the logarithmic modulus function
	on $\Complex^\times$, namely, $z\mapsto 2\cdot\log|z|$.	
	\item
	The second complex Borel regulator
	lives in
	$H^3_{\mathtt{diff}}(\mathrm{GL}(N,\Complex);\Complex)
	\cong H^3_{\mathtt{diff}}(\mathrm{SL}(2,\Complex);\Complex)$.
	One may identify
	$H^3_{\mathtt{diff}}(\mathrm{SL}(2,\Complex);\Complex)$
	with $H^3(\mathfrak{sl}(2,\Complex),\mathrm{SU}(2);\Complex)$,
	or equivalently, with the cohomology of 
	$\mathrm{Isom}_+(\Hyp^3)$--invariant complex valued differential forms on $\Hyp^3$
	\cite[Chapter I, Sections 1.6 and 5.1]{Borel--Wallach_book}.
	Then,
	the second complex Borel regulator 
	is explicitly represented by twice the volume form,
	namely, $2h^{-3} \ud x\wedge \ud y \wedge \ud h$.
	(The coeffiicent $2$ can be easily checked 
	at $(x,y,h)=(0,0,1)$ in $\Hyp^3$, according to Remark \ref{normalization_remark}).
	\end{enumerate}
	\end{example}
		
	In this paper, we are the most interested in the second complex Borel regulator.
	We rewrite the notation $p_{\Complex,2}$ particularly as
	\begin{equation}\label{b_complex}
	b_\Complex\in H^3_{\mathtt{diff}}(\mathrm{GL}(N,\Complex);\Complex),
	\end{equation}
	with any $N$ in the stable range,
	(such as $N\geq2$; see Example \ref{complex_Borel_regulator_example}).
	We simply refer to $b_\Complex$ as \emph{the} complex Borel regulator.
	
	With the above description,
	we obtain a formula for volume of $\mathrm{SL}(2,\Complex)$--representations
	as follows.
	For any oriented connected closed $3$--manifold $M$
	and any representation $\rho\colon\pi_1(M)\to \mathrm{SL}(2,\Complex)$,
	\begin{equation}\label{volume_Borel_complex}
	2\cdot \mathrm{Vol}_{\mathrm{SL}(2,\Complex)}(M,\rho)=
	\left\langle \rho^*_{\stable}(b_{\Complex}),\,[M]\right\rangle.
	\end{equation}
	Here, $\langle \_,\_\rangle$
	denotes the natural pairing 
	$H^3(M;\Complex)\times H_3(M;\Integral)\to \Complex$,
	and $\rho^*_{\mathtt{st}}$ denotes the homomorphism
	$H^3_{\mathtt{diff}}(\mathrm{GL}(N,\Complex);\Complex)
	\to H^3(\pi_1(M);\Complex)\to H^3(M;\Complex)$
	induced by the composite $\mathrm{GL}(N,\Complex)$--representation 
	$\rho_{\mathtt{st}}\colon\pi_1(M)\to\mathrm{SL}(2,\Complex)\to
	\mathrm{GL}(2,\Complex)\to \mathrm{GL}(N,\Complex)$.
		
	Let $p$ be a prime in $\Rational$.
	Let $L$ be a field extension over the $p$--adic field $\Rational_p$
	of finite algebraic degree,
	with the valuation ring denoted as $R$.
	Just as complex Borel regulators arise 
	in $H^*(\mathfrak{gl}(N,\Complex);\Complex)$,
	$L$--Borel regulators arise in the cohomology ring 
	$H^*(\mathfrak{gl}(N,L);L)$
	of the Lie algebra $\mathfrak{gl}(N,L)$ over $L$.
	The Lazard isomorphism plays the role 
	instead of the van Est isomorphism,
	identifying $H^n(\mathfrak{gl}(N,L);L)$
	with the locally analytic group cohomology	
	$H^n_{\mathtt{la}}(\mathrm{GL}(N,R);L)$
	of the compact $L$--Lie group $\mathrm{GL}(N,R)$.
	The cohomology ring 
	$H^*(\mathfrak{gl}(N,L);L)$
	of the Lie algebra $\mathfrak{gl}(N,L)$ over $L$
	is an exterior $L$--algebra 
	on primitive elements of odd dimensions up to $2N-1$,
	just as in the complex case.
	(See \cite[Section 1]{Huber--Kings_regulator}.)
	
	We declare the $m$--th \emph{$L$--Borel regulator}
	as the canonical locally analytic group cohomology class
	in $p_{L,m}\in H^{2m-1}_{\mathtt{la}}(\mathrm{GL}(N,R);L)$,
	for any $N$ in the stable range.	
	By relaxing the regularity condition on the space of functions,	
	one obtains natural homomorphisms of $L$--modules
	$$H^{2m-1}_{\mathtt{la}}(\mathrm{GL}(N,R);L)
	\to H^{2m-1}_{\mathtt{cont}}(\mathrm{GL}(N,R);L)
	\to H^{2m-1}(\mathrm{GL}(N,R);L).$$
	Therefore, it makes sense to declare
	$L$--Borel regulators less restrictively
	as continuous or abstract group cohomology classes,
	by natural transformations of cohomology theories.
	
	In this paper, we only need
	in the second $L$--Borel regulator, 
	so we simply call it \emph{the} $L$--Borel regulator.
	Therefore, we denote the (second) $L$--Borel regulator as 
	\begin{equation}\label{b_L}
	b_L\in H^3_{\mathtt{la}}(\mathrm{GL}(N,R);L),
	\end{equation}
	for any $N$ in the stable range.
	(For example, one may take any $N\geq3$,
	because of the canonical isomorphism
	$H^3_{\mathtt{la}}(\mathrm{GL}(N,R);L)\cong H^3_{\mathtt{la}}(\mathrm{GL}(3,R);L)$,
	by the Lazard isomorphism; see \cite[Theorem 1.2.1]{Huber--Kings_regulator}).

	\begin{remark}\label{b_real_remark}
	As $\Rational_\infty=\Real$ 
	has exactly two finite algebraic field extensions $\Real$ and $\Complex$,
	we briefly describe real Borel regulators for completeness of 
	the classical versus $p$--adic analogy.
	For every dimension $n$,
	the $n$--th differentiable group cohomology 
	$H^n_{\mathtt{diff}}(\mathrm{GL}(N,\Real);\Real)$
	of the real Lie group $\mathrm{GL}(N,\Real)$
	stabilizes for all sufficiently large $N$,
	and there are van Est isomorphisms
	$H^n_{\mathtt{diff}}(\mathrm{GL}(N,\Real);\Real)\cong
	H^n(\mathfrak{gl}(N,\Real),\mathrm{O}(N);\Real)$
	\cite[Chapter X, Section 5]{Borel--Wallach_book}.
	The stable cohomology is an exterior $\Real$--algebra
	on primitive elements $q_{\Real,m}$ of each degree $4m-3$,
	(see \cite[Section 10.6]{Borel_stable}).
	Upon suitable normalization, one may declare 
	$q_{\Real,m}\in H^{4m-3}_{\mathtt{diff}}(\mathrm{GL}(N,\Real);\Real)$
	as the $m$--th real Borel generator, with $N$ in the stable range.
	\end{remark}

	\subsection{Borel regulator maps associated to number fields}
	Suppose that $k/\Rational$ is finite algebraic field extension.
	As $\Rational$ has a unique Archimedean completion $\Rational_\infty=\Real$,
	there is a canonical ring isomorphism
	$k_\infty=k\otimes_\Rational\Rational_\infty\cong\prod_{v|\infty} k_v$.
	The ring $k_\infty$ is isomorphic to $\prod^{r_1}\Real\times\prod^{r_2}\Complex$,
	where $r_1+2r_2=[k:\Rational]$.
	We introduce an element
	\begin{equation}\label{Borel_infinity_k}
	b_\infty\in \bigoplus_{v|\infty} H^3_{\mathtt{diff}}(\mathrm{GL}(N,k_v);k_v)
	\end{equation}
	as the sum of the Borel regulators $b_v=b_\Complex$ 
	in the direct summands with $k_v\cong\Complex$,
	for $N$ in the stable range.
	Note that $H^3_{\mathtt{diff}}(\mathrm{GL}(N,k_v);k_v)$ vanishes 
	when $k_v\cong\Real$	(Remark \ref{b_real_remark}).
	The completion embeddings $k\to k_v$
	induce subgroup embeddings $\mathrm{SL}(2,k)\to \mathrm{SL}(2,k_v)$.
	Summing up the $k_v$--valued natural pairings with the pull-backs of $b_v=b_\Complex$
	associated to the complex places $v$,
	we obtain a real linear homomorphism, also denoted as
	\begin{equation}\label{Borel_infinity_k_map}
	b_\infty\colon H_3\left(\mathrm{SL}(N,k);\Integral\right)\otimes_\Integral\Real\to k_\infty,
	\end{equation}
	with $N$ sufficiently large.	
	In this paper,
	we refer to $b_\infty$ in (\ref{Borel_infinity_k_map})
	as the \emph{classical Borel regulator map}
	associated to the number field $k$.
	In (\ref{Borel_infinity_k_map}), one may actually take any $N\geq7$,
	and the following theorem is due to A.~Borel \cite{Borel_stable},
	(see Remark \ref{Borel_theorem_infinity_remark}).
	
	\begin{theorem}\label{Borel_theorem_infinity}	
	Let $k$ be a finite algebraic field extension over $\Rational$.
	Then, the classical Borel regulator map $b_\infty$,
	as defined in (\ref{Borel_infinity_k_map}) is injective.
	Moreover, $b_\infty$ has image of real dimension $r_2$,
	where $r_2$ denotes the number of complex places of $k$.
	\end{theorem}
	
	\begin{remark}\label{Borel_theorem_infinity_remark}
	For any (unital, associative) algebra $R$ and for $n\geq1$, 
	recall that the $n$-th K-group of $R$	is defined as 
	$K_n(R)=\pi_n(\mathrm{BGL}(R)^+)$,
	via Quillen's plus construction \cite[Chapter 5, Definition 5.2.6]{Rosenberg_book}.
	If $k/\Rational$ is a finite algebraic field extension with $r_1$ real places
	and $r_2$ complex places,
	Borel shows that $K_n(O)\otimes_\Integral\Real$ has real dimension 
	$0,r_1+r_2,0,r_2$ for $n\equiv 0,1,2,3$ modulo $4$ and $n\geq2$,
	where $O$ denotes the ring of integers in $K$ \cite[Proposition 12.2]{Borel_stable}.
	More precisely, Borel considers the composite homomorphism
	\begin{equation*}\label{Borel_K_map}
	\xymatrix{
	K_n(O) \ar[r] & 
	\bigoplus_{v|\infty} K_n(O_v) \ar[r]^-{h_*} &
	\bigoplus_{v|\infty} H_n(\mathrm{BGL}(k_v)^+;\Integral) \ar[r]^-{\cong} &
	\bigoplus_{v|\infty} H_n(\mathrm{BGL}(k_v);\Integral),
	}
	\end{equation*}
	where $h_*=\bigoplus_{v|\infty}h_{v*}$ denotes the Hurewicz homomorphism.
	By pairing with the complex or real Borel regulators of the correct dimension,
	one obtains canonical abelian group homomorphisms of $K_n(O)$ to 
	$\{0\},\Real^{r_1}\oplus\Complex^{r_2},\{0\},\Complex^{r_2}$,
	for $n\equiv 0,1,2,3$ modulo $4$, respectively.
	Borel shows that these homomorphisms induce isomorphisms of 
	$K_n(O)\otimes_\Integral\Real$
	onto $\{0\},\Real^{r_1}\oplus\Real^{r_2},\{0\},\Real^{r_2}$, respectively.
	The ring inclusion $O\to k$ induces an isomorphism
	$K_n(O)\otimes_\Integral\Real\cong K_n(k)\otimes_\Integral\Real$,
	by the localization theorem 
	\cite[Chapter 5, Corollary 5.3.28; see also Theorem 5.3.2]{Rosenberg_book}).
	When $n=3$ and $N\geq7$,
	there are natural isomorphisms
	$K_3(k)\otimes_\Integral\Real\cong
	H_3(\mathrm{EL}(k);\Integral)\otimes_\Integral\Real\cong
	H_3(\mathrm{EL}(N,k);\Integral)\otimes_\Integral\Real\cong
	H_3(\mathrm{SL}(N,k);\Integral)\otimes_\Integral\Real$;
	see \cite[Proposition 2.5 (a)]{Sah_III}, \cite[Theorem 4.11]{Kallen_homological_stability}
	and \cite[{\S 16, Corollary 16.3 and Remark}]{Milnor_book_K} for the steps, respectively.
	\end{remark}
	
	For any prime $p$ in $\Rational$,
	there is a canonical ring isomorphism
	$k_p=k\otimes_\Rational\Rational_p\cong\prod_{v|p} k_v$.
	Denote by $O$ the ring of integers in $k$.
	We denote $O_p=O\otimes_\Integral \Integral_p\cong\prod_{v|p} O_v$,
	and $O_{\{p\}}=k\cap O_p$ as viewed in $k_p$.
	Analogous to (\ref{Borel_infinity_k}) and (\ref{Borel_infinity_k_map}), 
	we introduce
	\begin{equation}\label{Borel_p_k}
	b_p\in \bigoplus_{v|p} H^3_{\mathtt{la}}(\mathrm{GL}(N,O_v);k_v),
	\end{equation}
	and obtain a $\Rational_p$--linear homomorphism
	\begin{equation}\label{Borel_p_k_map}
	b_p\colon H_3\left(\mathrm{SL}(N,O_{\{p\}});\Integral\right)\otimes_\Integral\Rational_p\to k_p.
	\end{equation}
	In this paper, we refer to $b_p$ in (\ref{Borel_p_k_map})
	as the \emph{$p$--adic Borel regulator map}	associated to the number field $k$.
	In (\ref{Borel_p_k_map}), one may actually take any $N\geq7$,
	and the following hypothesis is 
	a special case of the $p$--adic Borel regulator injectivity conjecture
	(Conjecture \ref{Borel_hypothesis_p_EZ}),
	parallel to Theorem \ref{Borel_theorem_infinity},
	(see Remark \ref{Borel_hypothesis_p_remark}).
		
	\begin{hypothesis}\label{Borel_hypothesis_p}
	For every finite algebraic field extension $k$ over $\Rational$
	with the ring of integers $O$,
	there exists infinitely many primes $p$ in $\Rational$,
	such that	the following condition on $k$ and $p$ holds true:
	\begin{itemize}
	\item 
	The $p$--adic Borel regulator map $b_p$ associated to $k$,
	as defined in (\ref{Borel_p_k_map}), 
	is injective.
	\end{itemize}
	\end{hypothesis}
	
	\begin{remark}\label{Borel_hypothesis_p_remark}
	By the same procedure as in Remark \ref{Borel_theorem_infinity_remark},
	there is a canonical abelian group homomorphism $K_{2m-1}(O)\to \bigoplus_{v|p} k_v$,
	obtained using the Hurewicz homomorphism and the $m$--th $k_v$--Borel regulators.
	Replacing with the natural isomorphisms
	$K_n(O)\otimes_\Integral\Rational_p\cong 
	K_n(k)\otimes_\Integral\Rational_p$ with $n\geq2$
	\cite[Chapter 5, Corollary 5.3.28 and Theorem 5.3.2]{Rosenberg_book},
	we obtain the $\Rational_p$--linear homomorphism
	$K_{2m-1}(k)\otimes_\Integral\Rational_p\to \bigoplus_{v|p} k_v$.
	This is what we call the $m$--th \emph{$p$--adic Borel regulator map}
	associated to $k$, which makes precise the statement of Conjecture \ref{Borel_hypothesis_p_EZ}.
	This way our Conjecture \ref{Borel_hypothesis_p_EZ} 
	paraphrases \cite[Conjecture 2.9]{Calegari_stable}.
	While the conjecture remains widely open in general,
	some hypothetical criterion or partial evidence 
	has been pointed out by F.~Calegari, 
	see \cite[Lemma 2.11 and Proposition 2.14]{Calegari_stable}.
	When $n=3$ and $N\geq7$, 
	there are natural isomorphisms
	$K_3(O)\otimes_\Integral\Rational_p\cong
	K_3(O_{\{p\}})\otimes_\Integral\Rational_p\cong
	H_3(\mathrm{EL}(O_{\{p\}});\Integral)\otimes_\Integral\Rational_p\cong
	H_3(\mathrm{EL}(N,O_{\{p\}});\Integral)\otimes_\Integral\Rational_p
	\cong H_3(\mathrm{SL}(N,O_{\{p\}});\Integral)\otimes_\Integral\Rational_p$,
	which justify our Hypothesis \ref{Borel_hypothesis_p}
	as an implication of Conjecture \ref{Borel_hypothesis_p_EZ};
	see 
	\cite[Chapter 5, Corollary 5.3.28 and Theorem 5.3.2]{Rosenberg_book},
	\cite[Proposition 2.5 (a)]{Sah_III}, \cite[Theorem 4.11]{Kallen_homological_stability},
	and \cite[{\S 16, Corollary 16.3 and Remark}]{Milnor_book_K} for the steps, respectively.
	\end{remark}

\section{Preliminary II}\label{Sec-prelim_II}
In this preliminary section, we review quaternion algebras, mostly following 
the textbook of Maclachlan and Reid \cite{Maclachlan--Reid_book}.
Next, we review some dynamical aspects of 
pseudo-Anosov automorphisms and their suspension flows,
following the perspective of \cite[Section 5]{Liu_vhsr}.

\subsection{Quaternion algebras}
	Let $F$ be a field of characteristic $0$. 
	A \emph{quaternion algebra} $\mathscr{A}$	over $F$
	is an (associative unital) algebra 
	which is isomorphic to a $4$--dimensional $F$--vector space
	$F\mathbf{1}\oplus F\mathbf{i}\oplus F\mathbf{j}\oplus F\mathbf{k}$
	enriched with the multiplication rules $\mathbf{i}^2=a\mathbf{1}$, $\mathbf{j}^2=b\mathbf{1}$, and 
	$\mathbf{i}\mathbf{j}=-\mathbf{j}\mathbf{i}=\mathbf{k}$,
	where $a,b\in F^\times$ are nonzero elements of $F$.
	Adopting the Hilbert symbol,
	this is written as
	\begin{equation}\label{Hilbert_symbol_def}
	\mathscr{A}\cong \left(\frac{a,b}{F}\right).
	\end{equation}
	Hence $\mathbf{k}^2=-ab\mathbf{1}$, $\mathbf{j}\mathbf{k}=-\mathbf{k}\mathbf{j}=-b\mathbf{i}$, 
	and $\mathbf{k}\mathbf{i}=-\mathbf{i}\mathbf{k}=-a\mathbf{j}$.
	Note that different Hilbert symbols may present isomorphic quaternion algebras,
	for example, $(a,b/F)\cong(b,a/F)\cong(a,-ab/F)\cong(b,-ab/F)\cong(ax^2,by^2/F)$, 
	where $a,b,x,y\in F^\times$.

	\begin{example}\label{qa_example}\
	\begin{enumerate}
	\item 
	Any quaternion algebra $\mathscr{A}$ over 
	a field $F$ of characteristic $0$ (or more generally, $\neq2$)
	is either a division algebra or
	isomorphic to the algebra $\mathrm{Mat}(2,F)$ of $2\times2$--matrices with entries in $F$
	\cite[Chapter 2, Theorem 2.1.7]{Maclachlan--Reid_book}.
	In particular, $\mathscr{A}\cong\mathrm{Mat}(2,F)\cong(1,1/F)$ 
	if $F$ is algebraically closed.
	\item 
	There are exactly two quaternion algebras over $\Real$
	up to isomorphism:
	the matrix algebra $\mathrm{Mat}(2,\Real)\cong(1,1/\Real)$ 
	and the division algebra of Hamilton's quaternions
	$\mathcal{H}\cong(-1,-1/\Real)$. 
	The similar dichotomy occurs for
	quaternion algebras over $\mathfrak{p}$--adic number fields $k_{\mathfrak{p}}$,
	where $k$ is an algebraic field extension over $\Rational$ of finite degree,
	and where $\mathfrak{p}$ is a prime, (that is, a nonzero prime ideal of 
	the subring of integers $O_k$).
	Besides the matrix algebra $\mathrm{Mat}(2,k_{\mathfrak{p}})\cong(1,1/k_{\mathfrak{p}})$,
	there is a unique division algebra up to isomorphism.
	It has a Hilbert symbol $(\pi,u/k_{\mathfrak{p}})$,
	such that $\pi$ is a uniformizer of the $\mathfrak{p}$--adic valuation ring
	$O_{k,\mathfrak{p}}$, and $k(\sqrt{u})$ is the unique unramified quadratic extension of $k$,
	up to isomorphism.
	In all the above cases,
	the division algebra is said to be \emph{ramified}, over $\Real$ or $k_{\mathfrak{p}}$,
	while the matrix algebra is \emph{unramified} (or \emph{split}).
	(See \cite[Chapter 2, Sections 2.5 and 2.6]{Maclachlan--Reid_book}.)
	\item
	Let $k$ is a finite algebraic field extension over $\Rational$,
	and $\mathscr{B}$ be a quaternion algebra over $k$.
	Then the Archimedean places correspond naturally with the real embeddings $\sigma\colon k\to\Real$
	and the complex-conjugate pairs of the (imaginary) complex embeddings 
	$\{\sigma,\bar{\sigma}\colon k\to \Complex\}$;
	the non-Archimedean places correspond naturally with the primes $\mathfrak{p}$ of $k$.
	The complex quaternion algebras
	$\mathscr{B}\otimes_\sigma\Complex\cong\mathscr{B}\otimes_{\bar\sigma}\Complex$
	at the complex places are always unramified,
	but the real quaternion algebras $\mathscr{B}\otimes_\sigma\Real$
	at the real places may ramify or split.
	Ramification or splitting may also occur with 
	the $\mathfrak{p}$--adic quaternion algebra $\mathscr{B}\otimes_kk_{\mathfrak{p}}$
	at a $\mathfrak{p}$--adic place.
	The classification of quaternion algebras over $k$
	demonstrates that the isomorphism type of $\mathscr{B}$
	is uniquely determined by its locus of ramification,
	that is, the set of places where the completion quaternion algebra ramifies.
	Moreover, the locus of ramification is always a finite set,
	consisting of an even number of real or non-Archmedean places,
	and any such configuration can be realized with a quaternion algebra over $k$.
	(See \cite[Chapter 2, Section 2.7, and Chapter 7, Theorem 7.3.6]{Maclachlan--Reid_book}.)
	\end{enumerate}
	\end{example}
	
	More intrinsically,
	a quaternion algebra $\mathscr{A}$ can be characterized
	as a simple, central, $4$--dimensional algebra over $F$.
	The $F$--vector subspaces $F\mathbf{1}$ and $F\mathbf{i}\oplus F\mathbf{j}\oplus F\mathbf{k}$
	can be characterized respectively as 
	the center of $\mathscr{A}$	and 
	the pure quaternion complement of the center,
	whose nonzero elements are noncentral and have central squares.
	The \emph{quaternion conjugation} on $\mathscr{A}$
	is thereby intrinsically defined,
	as 
	$$\overline{t\mathbf{1}+x\mathbf{i}+y\mathbf{j}+z\mathbf{k}}=t\mathbf{1}-x\mathbf{i}-y\mathbf{j}-z\mathbf{k},$$
	where $t,x,y,z\in F$.
	It allows one to introduce
	the (reduced) \emph{trace} and the (reduced) \emph{norm} of any quaternion $q\in\mathscr{A}$:
	\begin{equation}\label{Tr_Nr_def}
	\mathrm{Tr}_{\mathscr{A}/F}(q)=q+\bar{q},\mbox{ and }\mathrm{Nr}_{\mathscr{A}/F}(q)=q\bar{q},
	\end{equation}
	respectively,
	both valued in $F$.
	The functions $\mathrm{Tr}_{\mathscr{A}/F}\colon \mathscr{A}\to F$ and 
	$\mathrm{Nr}_{\mathscr{A}/F}\colon \mathscr{A}^\times\to F^\times$ 
	are homomorphisms	of the $F$--vector space $\mathscr{A}$ 
	and the multiplicative group of invertibles $\mathscr{A}^\times$,
	respectively.
	The trace and the norm fit into a neat quadratic equation:
	\begin{equation}\label{Tr_Nr_equation}
	q^2-\mathrm{Tr}_{\mathscr{A}/F}(q)\cdot q+\mathrm{Nr}_{\mathscr{A}/F}(q)=0.
	\end{equation}
	(See \cite[Chapter 2, Section 2.1]{Maclachlan--Reid_book}.)

	\begin{proposition}\label{trace_relations}
		Let $\mathscr{A}$ be a quaternion algebra over a field $F$ of characteristic $0$.
		For any $A,B,C,D\in \mathscr{A}$ with $\mathrm{Nr}(A)=\mathrm{Nr}(B)=\mathrm{Nr}(C)=\mathrm{Nr}(D)=1$, 
		the following identities all hold true:
		\begin{enumerate}
		\item
		$\mathrm{Tr}(1)=2$;\\
		$\mathrm{Tr}(A^{-1})=\mathrm{Tr}(A)$;\\ 
		$\mathrm{Tr}(A^2)=\mathrm{Tr}(A)^2-2$.
		\item
		$\mathrm{Tr}(BA)=\mathrm{Tr}(AB)$;\\
		$\mathrm{Tr}(BAB^{-1})=\mathrm{Tr}(A)$;\\
		$\mathrm{Tr}(AB^{-1})=\mathrm{Tr}(A)\cdot\mathrm{Tr}(B)-\mathrm{Tr}(AB)$;\\
		$\mathrm{Tr}(ABA^{-1}B^{-1})=
		\mathrm{Tr}(A)^2+\mathrm{Tr}(B)^2+\mathrm{Tr}(AB)^2
		-\mathrm{Tr}(A)\cdot\mathrm{Tr}(B)\cdot\mathrm{Tr}(AB)-2$.
		\item
		$2\cdot\mathrm{Tr}(ABCD)=\\
		\mathrm{Tr}(A)\cdot\mathrm{Tr}(BCD)+\mathrm{Tr}(B)\cdot\mathrm{Tr}(ACD)
		+\mathrm{Tr}(C)\cdot\mathrm{Tr}(ABD)+\mathrm{Tr}(D)\cdot\mathrm{Tr}(ABC)
		+\mathrm{Tr}(AB)\cdot\mathrm{Tr}(CD)-\mathrm{Tr}(AC)\cdot\mathrm{Tr}(BD)+\mathrm{Tr}(AD)\cdot\mathrm{Tr}(BC)
		-\mathrm{Tr}(A)\cdot\mathrm{Tr}(B)\cdot\mathrm{Tr}(CD)
		-\mathrm{Tr}(C)\cdot\mathrm{Tr}(D)\cdot\mathrm{Tr}(AB)
		-\mathrm{Tr}(A)\cdot\mathrm{Tr}(D)\cdot\mathrm{Tr}(BC)
		-\mathrm{Tr}(B)\cdot\mathrm{Tr}(C)\cdot\mathrm{Tr}(AD)
		+\mathrm{Tr}(A)\cdot\mathrm{Tr}(B)\cdot\mathrm{Tr}(C)\cdot\mathrm{Tr}(D).$
		\end{enumerate}
		Here we drop the subscripts in $\mathrm{Tr}_{\mathscr{A}/F}$ and $\mathrm{Nr}_{\mathscr{A}/F}$ for simplicity.
	\end{proposition}
	
	\begin{remark}\label{trace_relations_remark}
		These trace relations are all proved in \cite[Chapter 3, Section 3.4]{Maclachlan--Reid_book}
		for matrices $A,B,C,D\in\mathrm{SL}(2,\Complex)$.
		They also hold in the $\mathscr{A}/F$ setting.
		In fact, one may take a suitable subfield $F'$ of $F$ finitely generated over $\Rational$,
		and a quaternion subalgebra $\mathscr{A}'$ over $F'$ containing 
		the present quaternions $A,B,C,D\in\mathscr{A}$.
		By embedding $F'$ into $\Complex$, there is an algebra embedding of
		$\mathscr{A}'$ into $\mathscr{\mathscr{A}}'\otimes_{F'}\Complex\cong \mathrm{Mat}(2,\Complex)$.
		The determinant function on complex $2\times 2$--matrices
		agrees with $\mathrm{Nr}_{\mathscr{A}'/F'}=\mathrm{Nr}_{\mathscr{A}/F}$.
		This adapts the $\mathrm{SL}(2,\Complex)$ case to our setting.
	\end{remark}

\subsection{Pseudo-Anosov automorphisms}
	Let $S$ be an orientable connected closed surface of genus $\geq2$.
	A \emph{pseudo-Anosov} automorphism $f\colon S\to S$
	refers to an orientation-preserving homeomorphism with the following property:
	There exist measured foliations 
	$(\mathscr{F}^{\mathtt{s}},\mu^{\mathtt{s}})$ and $(\mathscr{F}^{\mathtt{u}},\mu^{\mathtt{u}})$
	on $S$,	having identical singular points and transverse anywhere else,
	such that
	$f\cdot(\mathscr{F}^{\mathtt{u}},\mu^{\mathtt{u}})=(\mathscr{F}^{\mathtt{u}},\lambda\mu^{\mathtt{u}})$
	and
	$f\cdot(\mathscr{F}^{\mathtt{s}},\mu^{\mathtt{s}})=(\mathscr{F}^{\mathtt{s}},\lambda^{-1}\mu^{\mathtt{s}})$
	hold for some constant $\lambda>1$.
	
	We recall that 
	a \emph{measured foliation} on $S$ is 
	a foliation $\mathscr{F}$ with prong singularities
	together with a transverse measure $\mu$ that is invariant under the holonomy of $\mathscr{F}$.
	A \emph{prong singularity} is modeled on the singularity
	of a holomorphic quadratic differential $z^{k-2}\ud z^2$ on $\Complex$, 
	where $k=3,4,5,\cdots$ is called the \emph{prong number}.
	(See \cite[{Expos\'e 1}]{FLP_book}.)
	The constant $\lambda$ is called the \emph{stretch factor}.
	The measured foliations $(\mathscr{F}^{\mathtt{s}},\mu^{\mathtt{s}})$ and $(\mathscr{F}^{\mathtt{u}},\mu^{\mathtt{u}})$
	are called the \emph{stable} and the \emph{unstable} measured foliations, respectively.
	They all uniquely determined by $f$.
	Moreover, any other pseudo-Anosov automorphism in the isotopy class of $f$
	is a conjugate of $f$ by a self-homeomorphism of $S$ that is isotopic to the identity,
	(see \cite[Expos\'e 12]{FLP_book}).
	
	Let $f\colon S\to S$ be a pseudo-Anosov automorphism.
	We obtain an orientable connected closed $3$--manifold
	\begin{equation}\label{M_f_def}
	M_f=\frac{S\times\Real}{(x,r+1)\sim(f(x),r)}.
	\end{equation}
	This is called the \emph{mapping torus} of $(S,f)$.
	The inclusion of $S\times\{0\}$ identifies $S$
	as a distinguished surface fiber in $M_f$.
	There is a the continuous family of homeomorphisms
	\begin{equation}\label{theta_t_def}
	\theta_t\colon M_f\to M_f
	\end{equation}	
	parametrized by $t\in\Real$ and determined by $(x,r)\mapsto (x,r+t)$,
	This is called the \emph{suspension flow} on $M_f$.
	There is a distinuished cohomology class of $M_f$, denoted as
	\begin{equation}\label{phi_f_def}
	\phi_f\in H^1(M_f;\Integral).
	\end{equation}
	It is	represented by the distinguished projection $M_f\to \Sph^1$,
	where $\Sph^1$ is identified with $\Real/\Integral$ 
	and the projection is determined by $(x,r)\mapsto r$.
		
	For any (nonzero) natural number $m\in\Natural$, 
	an \emph{$m$--periodic point} of $f$ 
	refers to a fixed point $p\in S$ of $f^m$.
	We obtain a line $\Real\to S\times\Real\colon r\mapsto (p,r)$,
	which descends to a loop $\Real/m\Integral\to M_f$.
	This gives rise to an $m$--periodic trajectory of the suspension flow.
	The construction sets up a natural bijective correspondence
	between the $f$--iteration orbits
	of $m$--periodic points on the surface 
	and 
	the $m$--periodic trajectories of 
	the suspension flow on the mapping torus.
	As $f$ is pseudo-Anosov,
	no periodic trajectories in $M_f$ are freely homotopic 
	to any other \cite[Corollary 2.3]{Jiang--Guo} 
	(see also \cite[Remark 2.1]{Liu_vhsr}).
	Therefore, the periodic orbits of $f$
	can be identified with mutually distinct
	conjugate classes in the fundamental group of $M_f$.
	For all $m\in\Natural$,
	the numbers of $m$--periodic orbit of $f$ are always finite,
	denoted as $N_m(f)$.
	They are called the \emph{Nielsen numbers} (of periodic orbit classes),
	and can be characterized purely in terms of the homotopy class of $f$.
	
	In this paper,
	we often fix a universal covering space $\widetilde{S}_\univ\to S$,
	and denote by $\pi_1(S)$ the deck transformation group acting on $\widetilde{S}_\univ$.
	Then the composite map 
	$$\xymatrix{
	\widetilde{S}_\univ\times\Real \ar[rr]^{\mathrm{cov}\times\mathrm{id}} & &
	S\times \Real \ar[rr]^{\mathrm{quot}} & & M_f
	}$$
	is a universal covering projection onto $M_f$.
	We denote by $\pi_1(M_f)$ the deck transformation group acting on 
	$\widetilde{S}_\univ\times \Real$,
	and by $\mathrm{Orb}(\pi_1(M_f))$ the set of conjugacy classes of $\pi_1(M_f)$.
	For any $m\in\Natural$,
	denote by 
	$\mathrm{Per}_m(f)\subset S$ the set of $m$--periodic points of $f$,
	and by
	$\mathrm{Orb}_m(f)=\mathrm{Per}_m(f)/\langle f\rangle$ 
	the set of $m$--periodic orbits.
	For any $\mathbf{O}\in\mathrm{Orb}_m(f)$,
	the $m$--periodic trajectory in $M_f$
	represents a free-homotopy loop. 
	We treat the free-homotopy loop 
	equivalently as a conjugacy class of $\pi_1(M_f)$,
	and denote it as $\ell_m(f,\mathbf{O})$.
	As mentioned, we obtain an injective map
	\begin{equation}\label{ell_m}
	\begin{array}{cc}
	\mathrm{Orb}_m(f)\to \mathrm{Orb}(\pi_1(M_f))\colon &
	\mathbf{O}\mapsto \ell_m(f,\mathbf{O})
	\end{array}
	\end{equation}
	for every $m\in\Natural$.
	Note that the homology class $[\ell_m(f,\mathbf{O})]\in H_1(M_f;\Integral)$
	satisfies $\phi_f(\ell_m(f,\mathbf{O}))=m$,
	treating $\phi_f$ naturally as a $\Integral$--linear homomorphism 
	$H_1(M_f;\Integral)\to \Integral$.
	
	The measure theoretical dynamics of 
	a pseudo-Anosov automorphism is 
	equivalent to a mixing Markov process.
	In this characterization,
	$f$ is comprehended 
	as a self-homeomorphism of $S$ that preserves 
	the product measure $\mu^{\mathtt{s}}\times\mu^{\mathtt{u}}$.
	With a 
	Markov partition,
	the system can be  effectively by a symbolic approach,
	as a subshift of finite type.
	The symbolic model is usually represented as 
	a directed graph or its adjacency matrix,
	called the \emph{transition graph} or the \emph{transition matrix}.
	In this paper, 
	we adopt the transition graph model and 
	use its dynamical cycles 
	to study periodic trajectories of the suspension flow.
	
	A \emph{directed graph} is a cell $1$--complex with oriented $1$--cells.
	We call the $0$--cells \emph{vertices}, and the oriented $1$--cells
	\emph{directed edges}.
	A \emph{dynamical cycle} is an immersed loop
	obtained by cyclically concatenating finitely many 
	directly consecutive directed edges. 
	Unless otherwise declared,
	we do not assume a base point	or a preferred parametrization
	when using this term.
	A directed graph is said to be \emph{irreducible} 
	(also called \emph{strongly connected})
	if every pair of (not necessarily distinct) vertices occurs
	in some dynamical cycle.
	In particular, irreducible directed graphs 
	are all topologically connected.
	
	\begin{proposition}\label{dc_pt_correspondence}
		Let $S$ be an orientable connected closed surface of genus $\geq2$.
		Let $f\colon S\to S$ be a pseudo-Anosov automorphism.
		
		Then, there exist an irreducible directed graph $T$,
		and a continuous map $T\to M_f$ 
		which is surjective	on the fundamental group level.
		Moreover, up to free homotopy of loops,
		every dynamical cycle in $T$ projects a periodic trajectory in $M_f$,
		and every periodic trajectory in $M_f$ admits a finite cyclic cover
		which lifts to a dynamical cycle in $T$, among finitely many choices.
		
		In fact, $T$ can be construction as the transition graph $T_{f,\mathcal{R}}$
		associated to a Markov partition $\mathcal{R}$ of $S$ with respect to $f$,
		and there is a continuous map	$T_{f,\mathcal{R}}\to M_f$, 
		which is canonical up to homotopy.
	\end{proposition}
	
	\begin{remark}
		See \cite[Section 5]{Liu_vhsr}.
		In fact, Proposition \ref{dc_pt_correspondence}
		is a simplified summary of 
		Lemmas 5.2, 5.3, 5.6, and 5.7 thereof.
		See also \cite[Expos\'es 9 and 10]{FLP_book}
		for exposition of the pseudo-Anosov theory
		from a Markov partition approach.		
	\end{remark}
	
	We briefly recall the construction of 
	$T_{f,\mathcal{R}}$ in \cite[Section 5]{Liu_vhsr}
	for the reader's reference.
		Roughly speaking, 
		a \emph{Markov partition} $\mathcal{R}=\{R_1,\cdots,R_k\}$ with respect to 
		a pseudo-Anosov automorphism $f\colon S\to S$
		is a collection of compact
		$(\mathscr{F}^{\mathtt{s}},\mathscr{F}^{\mathtt{u}})$--rectangles
		(or \emph{birectangles}), 
		which form a cover of $S$ without any overlap of positive area.
		Moreover, $f$ 
		streches any $R_i$ in the $\mathscr{F}^{\mathtt{u}}$--transverse direction by the factor $\lambda$,
		and shrinks it in the $\mathscr{F}^{\mathtt{s}}$--transverse direction by the same factor,
		such that $f(R_i)$ becomes a long and thin birectangle in $S$
		whose $\mathscr{F}^{\mathtt{u}}$--sides are contained, as subsegments,
		in some $\mathscr{F}^{\mathtt{u}}$--sides of some birectangles from $\mathcal{R}$.
		It is also required that $f(R_i)$ intersects $R_j$ 
		in at most one sub-birectangle, ignoring any null area components.
		Hence, there is a completely similar description of $f^{-1}(R_i)$,
		switching the roles of $\mathscr{F}^{\mathtt{u}}$ and $\mathscr{F}^{\mathtt{s}}$.
	
		Given a Markov partition $\mathcal{R}=\{R_1,\cdots,R_k\}$,
		the \emph{transition graph} $T_{f,\mathcal{R}}$ is constructed as follows.
		The vertices $v_i$ of $T_{f,\mathcal{R}}$ correspond bijectively with 
		all the birectangles $R_i\in\mathcal{R}$,
		and the directed edges $e_{ij}$ of $T_{f,\mathcal{R}}$ correspond bijectively with
		all the pairs $(R_i,R_j)\in\mathcal{R}\times\mathcal{R}$ 
		with intersection $R_i\cap f^{-1}(R_j)$ of positive area.
		Obtain a subset $\mathring{X}_{f,\mathcal{R}}$ of $M_f$
		as the union of all the subsets
		$\mathrm{int}(R_i)\times\{0\}$ and 
		$(\mathrm{int}(R_i)\cap f^{-1}(\mathrm{int}(R_j)))\times(0,1)$,
		(represented with parameters in $S\times\Real$).
		Then there is an obvious collapsing map
		$\mathring{X}_{f,\mathcal{R}}\to T_{f,\mathcal{R}}$,
		which is a homotopy equivalence.
		Up to homotopy,
		this gives rise to a distinguished map $T_{f,\mathcal{R}}\to M_f$,
		as the homotopy inverse followed by the inclusion.
		In fact, the path compactification of $\mathring{X}_{f,\mathcal{R}}$
		can be identified with the flow-box complex $X_{f,\mathcal{R}}$ 
		as introduced in \cite[Definition 5.1]{Liu_vhsr}.

\section{Periodic trajectories and arithmetic invariants}\label{Sec-traj_arithm}
In this section, we establish two criteria regarding 
irreducible $\mathrm{SL}(2,\Complex)$--representations of 
fundamental groups of pseudo-Anosov mapping tori
(Lemmas \ref{trace_field_pt} and \ref{quaternion_algebra_pt}).
These criteria relate the trace field and the ramification type
of the quaternion algebra (in the real case) 
with character values on the periodic trajectories.

For any group $\pi$,
an \emph{irreducible representation} $\rho\colon\pi\to\mathrm{SL}(2,\Complex)$
is a group homomorphism such that $\rho(\pi)$ preserves
no nontrivial invariant subspace in $\Complex^2$ \cite[Section 1.2]{CS_rep_var}.
In this case,
the \emph{trace field} of $\rho$ refers to the subfield of $\Complex$
generated by $\{\mathrm{tr}(\rho(g))\colon g\in\pi\}$ over $\Rational$,
and the \emph{quaternion algebra} of $\rho$ refers to
the subalgebra of $\mathrm{Mat}(2,\Complex)$ generated by $\{\rho(g)\colon g\in\pi\}$
over the trace field, which is indeed a quaternion algebra.
Note that if the trace field of $\rho$ contained in $\Real$,
then the completion of the quaternion algebra of $\rho$ is a quaternion algebra over $\Real$.
(See also Lemma \ref{irr_qarep_justification} and Definition \ref{irr_qarep_def}
in Section \ref{Sec-rep_qa} for justificaiton of the terminology 
in a slightly more general setting.)

\begin{lemma}\label{trace_field_pt}
	Let $S$ be an orientable connected surface of genus $\geq2$.
	Let $f\colon S\to S$ be a pseudo-Anosov automorphism.
	Denote by $M_f$ the mapping torus, furnished with the suspension flow.
	Suppose that $\rho\colon\pi_1(M_f)\to\mathrm{SL}(2,\Complex)$ is an irreducible representation.
			
	Then, any value of $\chi_\rho$ on $\pi_1(M_f)$ lies in the $\Integral$--subalgebra of $\Complex$
	generated by the values of $\chi_\rho$ on the set of periodic trajectories.
	In particular,
	the trace field of $\rho$ is the field of fractions of that subalgebra.
\end{lemma}

\begin{proof}
	Take a Markov partition $\mathcal{R}$ of $S$ with respect to $f$.
	We obtain the transition graph $T=T_{f,\mathcal{R}}$,
	whose dynamical cycles encode the periodic trajectories of $M_f$
	in the sense of Proposition \ref{dc_pt_correspondence}.
	
	We denote by $V=V(T)$ the set of vertices,
	and by $E=E(T)$ the set of directed edges.
	For any directed edge $e\in E$, denote by $\bar{e}$ the orientation-reversal of $e$.
	This operation also switches the roles of the initial and the terminal vertices.
	Denote by $\bar{E}$ the set of the reversed edges.
	A \emph{combinatorial path} in $T$ refers to 
	the concatenation of a finite sequence of	
	consecutive directed edges or reversed edges.
	We denote a combinatorial path as $\alpha=e_1e_2\cdots e_n$,
	where $e_i\in E\cup\bar{E}$,
	(so the initial and the terminal vertices of $\alpha$ 
	are the initial vertex of $e_1$
	and the terminal vertex of $e_n$, respectively).
	There is a reversed combinatorial path
	$\bar{\alpha}=\bar{e}_n\cdots \bar{e}_2\bar{e}_1$.
	We call $\alpha$ a \emph{dynamical path} if $e_1,\cdots,e_n\in E$.
	When the terminal vertex of a combinatorial path $\alpha$
	coincides with the initial vertex of another $\beta$,
	there is a concatenated combinatorial path $\alpha\beta$.
	When the initial vertex of a combinatorial path $\gamma$ 
	coincides with the terminal vertex,
	$\gamma$ is closed,
	then it makes sense to evaluate $\chi_\rho$ 
	on $\gamma$, 
	treating $\gamma$ as a free-homotopy loop in $T$ and
	using the homotopically distinguished map $T\to M_f$.
	The value of $\chi_\rho$ is invariant under 
	cyclical permutations of the edge sequence of $\gamma$ and the reversion of $\gamma$,
	by the conjugacy invariance and the inversion invariance
	of the trace on $\mathrm{SL}(2,\Complex)$.
	Moreover,
	any loop in $T$ can be freely homotoped to a closed combinatorial path in $T$,
	(up to parametrization).
	
	%
	
	We claim that any combinatorial closed path $\gamma$ in $T$
	is freely homotopic to another one of the concatenated form
	$$\xi_0\bar{\eta}_1\xi_1\cdots\bar{\eta}_n\xi_n,$$
	or possibly without $\xi_0$ or $\xi_n$,
	such that $\xi_i,\eta_i$ are all closed dynamical paths.
	If $\gamma$ has an assigned base vertex $v_0$,
	one may also require $\xi_i$ and $\eta_i$ all based at $v_0$.
	This basically follows from the irreducibility of the directed graph $T$
	(Proposition \ref{dc_pt_correspondence}).
	In fact, $\gamma$ can be written as a unique concatenated form
	$\alpha_0\bar{\beta}_1\alpha_1\cdots\bar{\beta}_n\alpha_n$,
	or possibly without $\alpha_0$ or $\alpha_n$,
	such that $\alpha_i,\beta_i$ are all dynamical paths.
	For each $\bar{\beta}=\bar{\beta}_j$,
	there is a dynamical cycle passing through 
	the initial vertex $v$ and the terminal vertex $w$
	of $\bar{\beta}$, since $T$ is irreducible.
	Then there is a dynamical path $\delta$ from $v$ to $w$.
	Replace the term $\bar{\beta}$ in $\gamma$ with
	$\delta\bar{\delta}\bar{\beta}=\delta\overline{(\beta\delta)}$.
	The modified part becomes a dynamical path $\delta$ concatenated
	with a reversed closed dynamical path $\overline{(\beta\delta)}$.
	Modifying all $\beta_j$ in $\gamma$ this way,
	we end up getting a new concatenated form 
	$\alpha_0\bar{\beta}_1\alpha_1\cdots\bar{\beta}_n\alpha_n,$
	based at $v_0$ and homotopic to $\gamma$ relative to $v_0$,
	such that every $\beta_j$ is closed,
	and hence $\alpha_0,\alpha_1,\cdots,\alpha_n$ are cyclically consecutive.
	Denote $\alpha_{[p,q]}=\alpha_p\alpha_{p+1}\cdots\alpha_q$.
	We take the closed dynamical paths 
	$\xi_0=\xi_1=\cdots=\xi_n=\alpha_{[0,n]}$,
	$\eta_j=\alpha_{[0,j-1]}\beta_j\alpha_{[j,n]}$,
	all based at $v_0$.
	Then the concatenated form 
	$\alpha_0\bar{\beta}_1\alpha_1\cdots\bar{\beta}_n\alpha_n$
	can be further modified into
	$\xi_0\bar{\eta}_1\xi_1\cdots\bar{\eta}_n\xi_n,$
	by inserting between $\alpha_{j-1}$ and $\bar{\beta}_j$ 
	the term $\alpha_{[j,n]}\overline{\alpha_{[j,n]}}$,
	and inserting between $\bar{\beta}_j$ and $\alpha_j$
	the term $\overline{\alpha_{[0,j-1]}}\alpha_{[0,j-1]}$.
	The resulting concatenated form
	$\xi_0\bar{\eta}_1\xi_1\cdots\bar{\eta}_n\xi_n$ 
	is as claimed.
		
	Since $T\to M_f$ is $\pi_1$--surjective (Proposition \ref{dc_pt_correspondence}),
	any conjugacy class in $\pi_1(M_f)$ admits a lift to $\pi_1(T)$,
	and can be represented as a closed combinatorial path $\gamma$.
	We may assume $\gamma$ takes the claimed form
	$\xi_0\bar{\eta}_1\xi_1\cdots\bar{\eta}_n\xi_n$,
	or possibly without $\xi_0$ or $\xi_n$, as above.
	We show that the value of $\chi_\rho$ at $\gamma$
	can be expressed as a $\Integral$--polynomial function of
	values of $\chi_\rho$ at dynamical cycles.
	This will immediately imply Lemma \ref{trace_field_pt},
	because the dynamical cycles all project periodic trajectories
	in $M_f$ up to free homotopy (Proposition \ref{dc_pt_correspondence}).
	
	We argue by induction on the number $n$ of reversed closed dynamical path terms.
	Possibly after a cyclic permutation,
	we may rewrite $\gamma$ as $\gamma'\bar{\eta}_n$,
	where $\gamma'=(\xi_n\xi_0)\bar{\eta}_1\xi_1\cdots\bar{\eta}_{n-1}\xi_{n-1}$.
	By trace relations (Proposition \ref{trace_relations}),
	we obtain
	$$\chi_\rho(\gamma)=
	\chi_\rho(\gamma'\bar{\eta}_n)=\chi_\rho(\gamma')\cdot\chi_\rho(\bar{\eta}_n)-\chi_\rho(\gamma'\eta_n)
	=\chi_\rho(\gamma')\cdot\chi_\rho(\eta_n)-\chi_\rho(\gamma'\eta_n).$$
	Note that $\gamma'$ and $\gamma'\eta_n$ 
	are both in the claimed form with $n-1$
	reversed closed dynamical cycle terms.
	By induction, $\chi_\rho(\gamma')$ and $\chi_\rho(\gamma'\eta_n)$
	are both $\Integral$--polynomial functions
	of values of $\chi_\rho$ at dynamical cycles,
	so $\chi_\rho(\gamma)$ also has the same property.
	This completes the induction.
\end{proof}

We also obtain the following byproduct of independent interest.
It is probably known to experts.

\begin{corollary}\label{normal_generation_pt}
	For any pseudo-Anosov automorphism,
	the fundamental group of the mapping torus
	is normally generated by the periodic trajectories
	of the suspension flow.
	In other words, 
	the union of the conjugacy classes represented by the periodic trajectories
	forms a generating set.
\end{corollary}
%

\begin{proof}
	As shown in the proof Lemma \ref{trace_field_pt},
	any closed combinatorial path $\gamma$ in an irreducible directed graph $T$
	is freely homotopic to 
	a concatenation of finitely many closed dynamical paths 
	or their reversals $\xi_0\bar{\eta_1}\cdots\bar{\eta}_n\xi_n$. 
	It follows that attaching $2$--cells along all the dynamical cycles in $T$
	will result in a simply connected CW $2$--complex,
	or equivalently,
	the union of the conjugacy classes represented by the dynamical cycles
	forms a generating set of $\pi_1(T)$.
	Then the normal generation of $\pi_1(M_f)$ 
	by periodic trajectories follows from the $\pi_1$--surjectivity
	of the map $T\to M_f$ as provided in Proposition \ref{dc_pt_correspondence}.
\end{proof}

%
%
%

\begin{lemma}\label{quaternion_algebra_pt}
Let $S$ be an orientable connected surface of genus $\geq2$.
	Let $f\colon S\to S$ be a pseudo-Anosov automorphism.
	Suppose that $\rho\colon\pi_1(M_f)\to\mathrm{SL}(2,\Complex)$ 
	is an irreducible representation whose trace field is real.
	
	Then, the following dichotomy holds true.
	\begin{enumerate}
	\item If $\chi_\rho$ is uniformly bounded on the set of periodic trajectories,
	then the real complete quaternion algebra of $\rho$ is ramified.
	In this case, $\rho$ is conjugate to a representation in $\mathrm{SU}(2)$.
	\item If $\chi_\rho$ is unbounded on the set of periodic trajectories,
	then the real complete quaternion algebra of $\rho$ is unramified.
	In this case, $\rho$ is conjugate to a representation in $\mathrm{SL}(2,\Real)$.
	\end{enumerate}
\end{lemma}

\begin{proof}
	The real complete quaternion algebra is isomorphic to either
	Hamilton's quaternion algebra
	$\mathcal{H}$ or 
	the real $2\times2$--matrix algebra $\mathrm{Mat}(2,\Real)$, (see Example \ref{qa_example}).
	The elements of norm $1$ form a real Lie group,
	isomorphic to $\mathrm{SU}(2)$ or $\mathrm{SL}(2,\Real)$, respectively.
	The natural inclusion of the image of $\rho$ 
	into the real complete quaternion algebra
	can be regarded as a representation of $\pi$ in $\mathrm{SU}(2)$ or $\mathrm{SL}(2,\Real)$,
	which has the same character as $\chi_\rho$.
	Therefore, the conjugacy characterizations in the dichotomy
	follow from the well-known fact
	that irreducible $\mathrm{SL}(2,\Complex)$--representations
	are uniquely determined by their characters up to conjugacy
	\cite[Proposition 1.5.2]{CS_rep_var}.
	
	Recall any matrix $A\in\mathrm{SL}(2,\Real)$ other than $\pm I$
	falls into one of three types, 
	according to the sign of its discriminant $\mathrm{tr}(A)^2-4$:
	elliptic if negative, parabolic if zero, or hyperbolic if positive.
	In the hyperbolic case,
	$|\mathrm{tr}(A^n)|$ grows exponentially
	as $n$ tends to $\infty$.
	On the other hand,
	$|\mathrm{tr}(A)|\leq 2$ holds for any matrix $A\in\mathrm{SU}(2)$. 
	With these facts,
	the proof of Lemma \ref{quaternion_algebra_pt}
	essentially reduces to the following claim:
	\emph{
	For any representation $\rho\colon\pi_1(M_f)\to \mathrm{SL}(2,\Real)$
	that is irreducible in $\mathrm{SL}(2,\Complex)$,
	there exists some periodic trajectory 
	which is, up to conjugacy, represented by a matrix of hyperbolic type.
	}
	
	To prove the claim,
	it suffices to assume that the image of $\rho$ is torsion-free.
	In fact, if $G=\mathrm{Im}(\rho)$ contains nontrivial torsion,
	there is a torsion-free finite-index normal subgroup $G'$ of $G$,
	by Selberg's lemma \cite[Chapter 7, \S 7.6, Corollary 4]{Ratcliffe_book}.
	The preimage of $G'$ in $\pi=\pi_1(M_f)$ is a finite-index normal subgroup
	$\pi'=\pi_1(M')$, 
	which corresponds to a connected regular finite cover $M'\to M_f$.
	We can identify $M'$ as the mapping torus $M_{f'}$
	of a pseudo-Anosov automorphism $f'\colon S'\to S'$,
	where $S'$ is a preimage component in $M'$ of the distingushed fiber $S$ in $M_f$,
	and where $f'$ is the return map of the lifted suspension flow on $M'$.
	The restriction of $\rho$ to $\pi'$ becomes a representation 
	$\rho'\colon\pi_1(M_{f'})\to\mathrm{SL}(2,\Real)$.
	As $\rho$ is irreducible in $\mathrm{SL}(2,\Complex)$,
	$\rho'$ is also irreducible in $\mathrm{SL}(2,\Complex)$.
	(This follows from the fact that a subgroup in $\mathrm{SL}(2,\Real)$
	has a fixed point in $\Complex\cup\{\infty\}$ if and only if
	it has a finite orbit in $\Complex\cup\{\infty\}$.)
	Therefore, we can deduce the general case from the torsion-free case
	by arguing with $f'\colon S'\to S'$ and $\rho'$.
	
	Below we suppose that $\rho$ is torsion-free.	
	Take a Markov partition $\mathcal{R}$ of $S$ with respect to $f$.
	We obtain the transition graph $T=T_{f,\mathcal{R}}$,
	whose dynamical cycles encode the periodic trajectories of $M_f$
	in the sense of Proposition \ref{dc_pt_correspondence}.
	We adopt the notations and terminology as in the proof of Lemma \ref{trace_field_pt}.
	Recall that $V=V(T)$ denotes the set of vertices,
	and $E=E(T)$ the set of directed edges.
	To speak of fundamental groups without ambiguity,
	we fix basepoints of $T$ and $M_f$ as follows.
	Take a vertex $v\in V$ as the basepoint of $T$. 
	Fix a map $T\to M_f$ among its distinguished homotopy class (Proposition \ref{dc_pt_correspondence}),
	and take the image $*$ of $v$ in $M_f$ as the basepoint of $M_f$.
	Denote by $\rho_T\colon \pi_1(T,v)\to \pi_1(M_f,*)\to \mathrm{SL}(2,\Real)$ 
	the pull-back representation of $\rho$.	
	
	We start by
	observing that there exists a pair of closed combinatorial paths $\mu$ and $\nu$,
	both based at $v$, such that $\rho_T$ is $\mathrm{SL}(2,\Complex)$--irreducible
	restricted to the subgroup $\langle \mu,\nu\rangle$ of $\pi_1(T,v)$.
	This is	because $\rho_T$ is $\mathrm{SL}(2,\Complex)$--irreducible on $\pi_1(T,v)$.
	By inserting terms of the form $\delta\bar{\delta}$,
	these combinatorial paths can be modified	into concatenation forms
	$$
	\begin{array}{cc}
	\mu=\xi_0\bar{\eta}_1\xi_1\cdots\bar{\eta}_m\xi_m, & \nu=\psi_0\bar{\varphi}_1\psi_1\cdots\bar{\varphi}_n\psi_n,
	\end{array}
	$$
	possibly without $\xi_0$, $\xi_m$, $\psi_0$, or $\psi_n$,
	such that $\xi_i,\psi_j$ are all dynamical paths,
	and $\eta_i,\varphi_j$ are all closed dyanmical paths.
	This is the same trick as explained in the proof of Lemma \ref{trace_field_pt},
	which relies on the irreducibility of the directed graph $T$ (Proposition \ref{dc_pt_correspondence}).
	(One may also make $\xi_i$ and $\psi_j$ closed,
	although our subsequent argument does not require it.)
	In particular, the modification does not change the homotopy class of the closed paths
	relative to the base point $v$.	
	
	We show that there exists a pair of closed dynamical paths $\alpha$ and $\beta$,
	both based at $v$, such that $\rho_T$ is $\mathrm{SL}(2,\Complex)$--irreducible
	restricted to the subgroup $\langle \alpha,\beta\rangle$ of $\pi_1(T,v)$.
	This is done based on the existence 
	of closed combinatorial paths $\mu$ and $\nu$ as above,
	(already modified into specified concatenation forms).
	We argue by induction on the complexity $\max(m,n)$.
	In the induction step, we pass to other simpler $(\mu,\nu)$
	keeping the $\mathrm{SL}(2,\Complex)$--irreducibility of $\rho_T$ on $\langle\mu,\nu\rangle$.

	When $\max(m,n)=0$, we can simply take $\alpha=\mu$ and $\beta=\nu$.
	Suppose that $\alpha$ and $\beta$ can be constructed 
	when $\max(m,n)=l$, for some nonnegative integer $l$.
	When $\max(m,n)=l+1$, we may assume $m=l+1$, without loss of generality.
	Write $\mu=\mu'\bar{\eta}_m\xi_m$, where $\mu'=\xi_0\bar{\eta}_1\xi_1\cdots\bar{\eta}_{m-1}\xi_{m-1}$.
	
	We observe that $\rho_T$ is $\mathrm{SL}(2,\Complex)$--irreducible
	either on $\langle \mu'\eta_m\xi_m,\mu'\xi_m\rangle$
	or on $\langle \mu'\eta_m^k\xi_m,\nu\rangle$ for some integer $k\geq0$.
	In fact, $\rho_T(\nu)\in\mathrm{SL}(2,\Real)$ must be noncentral,
	since $\rho_T$ is irreducible on $\langle \mu,\nu\rangle$.
	Hence, $\rho_T(\nu)$ has only one or two fixed points on $\Complex\cup\{\infty\}$.
	Suppose $\rho_T$ is $\mathrm{SL}(2,\Complex)$--reducible 
	on both 
	$\langle \mu'\xi_m,\nu\rangle$ and $\langle \mu'\eta_m\xi_m,\nu\rangle$,
	then both $\rho_T(\mu'\xi_m)$ and $\rho_T(\mu'\eta_m\xi_m)$
	fix some fixed point of $\rho_T(\nu)$.
	If they	fix 
	a common fixed point $p\in\Complex\cup\{\infty\}$ of $\rho_T(\nu)$, 
	then $p$ is also fixed by
	$$\rho_T(\mu)=\rho_T(\mu'\bar{\eta}_m\xi_m)
	=\rho_T\left((\mu'\xi_m)\overline{(\mu'\eta_m\xi_m)}(\mu'\xi_m)\right)
	=\rho_T(\mu'\xi_m)\cdot\rho_T(\mu'\eta_m\xi_m)^{-1}\cdot\rho_T(\mu'\xi_m).$$
	However, this would contradict the assumed $\mathrm{SL}(2,\Complex)$--irreducibility
	of $\rho_T$ on $\langle \mu,\nu\rangle$.
	Therefore,
	if $\rho_T$ is $\mathrm{SL}(2,\Complex)$--reducible 
	on both 
	$\langle \mu'\eta_m\xi_m,\nu\rangle$ and $\langle \mu'\xi_m,\nu\rangle$,
	we infer that $\rho_T(\mu'\xi_m)$ and $\rho_T(\mu'\eta_m\xi_m)$ 
	are hyperbolic or parabolic with fixed point sets
	neither containing the other.
	In particular, neither of them is central.
	If their fixed point sets are disjoint from each other,
	$\rho_T$ is $\mathrm{SL}(2,\Complex)$--irreducible on $\langle \mu'\eta_m\xi_m,\mu'\xi_m\rangle$.
	Otherwise, the only possibility is that $\rho_T(\mu'\xi_m)$ and $\rho_T(\mu'\eta_m\xi_m)$ 
	are both hyperbolic, each fixing a distinct fixed point of $\rho_T(\nu)$,
	and having another common fixed point $q$.
	In this case, $\rho_T(\bar{\xi}_m\eta_m\xi)=\rho_T(\mu'\xi_m)^{-1}\cdot\rho_T(\mu'\eta_m\xi_m)$
	fixes $q$, but does not commute with either $\rho_T(\mu'\xi_m)$ or $\rho_T(\mu'\eta_m\xi_m)$.
	Therefore, the fixed point set of $\rho_T(\bar{\xi}_m\eta_m\xi_m)$ is 
	disjoint from the fixed point set of $\rho_T(\nu)$.
	It follows that for any sufficiently large integer $k>0$,
	$\rho_T(\mu'\eta_m^k\xi_m)=
	\rho_T(\mu'\xi_m)\cdot\rho_T(\bar{\xi}_m\eta_m\xi_m)^k$
	does not fix any fixed point of $\rho_T(\nu)$.
	We conclude that $\langle \mu'\eta_m^k\xi_m,\nu\rangle$ is
	$\mathrm{SL}(2,\Complex)$--irreducible in this case.	
	
	With the above observation,
	we replace the pair $(\mu,\nu)$
	with either $(\mu'\eta_m\xi_m,\mu'\xi_m)$ or some $(\mu'\eta_m^k\xi_m,\nu)$, 
	denoted as $(\tilde{\mu},\tilde{\nu})$,
	such that $\rho_T$ is $\mathrm{SL}(2,\Complex)$--irreducible
	on $\langle \tilde{\mu},\tilde{\nu}\rangle$.
	Passing to $(\tilde{\mu},\tilde{\nu})$
	causes the numbers of reversed dynamical cycles $(m,n)$ 
	changing into $(m-1,m-1)$ or $(m-1,n)$, respectively.
	If $n$ is smaller than $m=l+1$, we have decreased $\max(m,n)$ by $1$,
	so we can construct the asserted $\alpha$ and $\beta$, by induction.
	If $n$ equals $m=l+1$,
	we need to redo the above construction
	with respect to $(\tilde{\mu},\tilde{\nu})$,
	working on $\tilde{\nu}$.
	This will decrease $\max(m,n)$ by $1$,
	so we can still construct $\alpha$ and $\beta$.
	
	To finish the proof,
	let $\alpha$ and $\beta$ be a pair of closed dynamical paths in $T$ based at $v$,
	such that $\rho_T\colon \pi_1(T,v)\to\mathrm{SL}(2,\Real)$ 
	is $\mathrm{SL}(2,\Complex)$--irreducible on $\langle \alpha,\beta\rangle$.
	We show that $\rho_T(\alpha^r\beta^s)$ is hyperbolic for some 
	nonnegative integers $r$ and $s$.
	This will complete the proof since the dynamical cycle $\alpha^r\beta^s$ in $T$
	projects a periodic trajectory in $M_f$ 
	as claimed (Proposition \ref{dc_pt_correspondence}).
	
	When $\sigma_\alpha=\rho_T(\alpha)$ or $\sigma_\beta=\rho_T(\beta)$ is hyperbolic,
	we can take $(r,s)$ as $(1,0)$ or $(0,1)$, respectively.
	
	When $\sigma_\alpha$ is parabolic,
	it fixes some point $p\in\Real\cup\{\infty\}$.
	As $\rho_T$ is $\mathrm{SL}(2,\Complex)$--irreducible on $\langle\alpha,\beta\rangle$,
	the point $p$ is distinct from $\sigma_\beta^{-1}(p)$.
	Take a small compact arc $J\in\Real\cup\{\infty\}$
	containing $p$ in the interior,
	such that $J$ and $\sigma_\beta^{-1}(J)$ are disjoint.
	Then for some sufficiently large $r$,
	we can make $\sigma_\alpha^r\sigma_\beta(J)$ contained in $J$,
	and
	$\sigma_\beta^{-1}\sigma_\alpha^{-r}\sigma_\beta^{-1}(J)$ contained in $\sigma_\beta^{-1}(J)$.
	In particular, $\sigma_\alpha^r\sigma_\beta$
	has two distinct fixed points on $\Real\cup\{\infty\}$,
	so it is hyperbolic.
	We can take $(r,s)$ as $(r,1)$ for some large $r$.
	Similarly, 
	when $\sigma_\beta$ is parabolic,
	we can take $(r,s)$ as $(1,s)$ for some large $s$.
	
	When $\sigma_\alpha$ and $\sigma_\beta$ are both elliptic,
	they fix distinct points $p$ and $q$ in $\Hyp^2$, 
	since $\rho_T$ is $\mathrm{SL}(2,\Complex)$--irreducible on $\langle \alpha,\beta\rangle$.
	They act on $\Hyp^2$ as rotations about those points, in view of hyperbolic geometry.
	Note that composing the $\pi$--rotation about $p$ 
	with the $\pi$--rotation about $q$ yields a hyperbolic translation on $\Hyp^2$
	along the axis through $p$ and $q$.
	As we have assumed $\rho_T$ torsion-free,
	the rotation angles of $\sigma_\alpha$ and $\sigma_\beta$ are irrational multiples of $\pi$.
	There are positive integers $r$ and $s$,
	such that the rotations angles of $\sigma_\alpha^r$ and $\sigma_\beta^s$
	are very close to $\pi$.
	This way, $\sigma_\alpha^r\sigma_\beta^s$ is hyperbolic.
	We can take any such $(r,s)$ as claimed.
\end{proof}

\section{Profinite correspondence of periodic trajectories}\label{Sec-profin_corresp_traj}
In this section, we explain a bijective correspondence between the periodic trajectories
in pseudo-Anosov mapping tori
with profinitely isomorphic fundamental groups.
This correspondence has been established \cite{Liu_profinite_almost_rigidity},
although it is not stated explicitly therein.
Therefore, the purpose of this section is largely expository,
and we extract a precise form of the correspondence (Lemma \ref{profinite_correspondence_pt})
from results in \cite{Liu_profinite_almost_rigidity}.
We draw another useful consequence  (Lemma \ref{Mu_cube}) 
from a main result of \cite{Liu_profinite_almost_rigidity},
which appears for the first time.

Let $M_A$ and $M_B$ be a pair of closed hyperbolic $3$--manifolds.
If $\Psi\colon \widehat{\pi}_A\to \widehat{\pi}_B$ is an isomorphism
between the profinite completions of their fundamental groups 
$\pi_A=\pi_1(M_A)$ and $\pi_B=\pi_1(M_B)$,
then for any surface fiber $S_A$ of $M_A$,
there is a unique surface fiber $S_B$ of $M_B$,
such that the closure of the normal subgroup $\Sigma_A=\pi_1(S_A)$
in $\widehat{\pi}_A$ projects isomorphically onto
the closure of $\Sigma_B=\pi_1(S_B)$ in $\widehat{\pi}_B$, under $\Psi$.
Note that the closures of $\Sigma_A$ and $\Sigma_B$ can be identified with
their own profinite completions $\widehat{\Sigma}_A$ and $\widehat{\Sigma}_B$, respectively.

Fix transverse orientations of the surface fibers in the $3$--manifolds.
Then $M_A$ can be identified homeomorphically with the mapping torus of a
pseudo-Anosov automorphism $f_A\colon S_A\to S_A$, 
which is isotopically unique,
and which has a distinguished cohomology class $\phi_A\in H^1(M_A;\Integral)$
dual to $S_A$.
Similarly, we obtain $f_B\colon S_B\to S_B$ and 
$\phi_B\in H^1(M_B;\Integral)$ associated with $M_B$.
There is a unique invertible profinite integer $\mu\in\widehat{\Integral}^\times$,
(according to the fiber orientation),
such that $\mu$ fits into the following commutative diagram of profinite groups:
\begin{equation}\label{Sigma_mu_diagram}
\xymatrix{
\{1\} \ar[r] & \widehat{\Sigma}_A \ar[rr] \ar[d]_{\cong}^{\Psi|} 
& & \widehat{\pi}_A \ar[rr]^{t^{\widehat{\phi}_A}} \ar[d]_{\cong}^{\Psi}
& & t^{\widehat{\Integral}} \ar[r] \ar[d]^{\mathrm{pwr}_\mu}_{\cong} & \{1\} \\
\{1\} \ar[r] & \widehat{\Sigma}_B \ar[rr] & & \widehat{\pi}_B \ar[rr]^{t^{\widehat{\phi}_B}}
& & t^{\widehat{\Integral}} \ar[r] & \{1\}
}
\end{equation}
where $t^{\widehat{\Integral}}$ denotes the profinite completion of 
the infinite cyclic group $\langle t\rangle=t^{\Integral}$, 
and where we define $\mathrm{pwr}_\mu\colon t^{\nu}\mapsto t^{\mu\nu}$
and $t^{\widehat{\phi}}\colon g\mapsto t^{\widehat{\phi}(g)}$.

\begin{convention}\label{profinite_isomorphism_setting_pA}
	By declaring $(M_A,M_B,\Psi)$ as 
	a \emph{profinite isomorphism setting of pseudo-Anosov mapping tori},
	we assume that $M_A$ and $M_B$ are a pair of closed fibered hyperbolic $3$--manifolds,
	with fixed universal covering spaces 
	and fixed transversely oriented connected surface fibers,
	and 
	we assume that $\Psi\colon\widehat{\pi}_A\to\widehat{\pi}_B$
	is an isomorphism between the profinite completions of the universal deck transformation groups,
	fitting into the diagram (\ref{Sigma_mu_diagram}).
	We adopt the notations as involved in the diagram (\ref{Sigma_mu_diagram}),
	such as $\widehat{\phi}_A$, $\widehat{\phi}_B$, and $\mu$.
\end{convention}

If $\pi$ is a finitely generated residually finite group 
with profinite completion $\widehat{\pi}$,
any conjugacy class $\mathbf{c}$ in $\pi$ determines a unique conjugacy class
$\hat{\mathbf{c}}$ in $\widehat{\pi}$, 
which is the closure of $\mathbf{c}$.
(Different $\mathbf{c}$ determines different $\hat{\mathbf{c}}$
if and only if $\pi$ is conjugacy separable.)
For any $n\in\Integral$, there is a well-defined conjugacy class $\mathbf{c}^n$
in $\pi$, represented by the $n$--th power of any representative of $\mathbf{c}$.
Similarly, for any $\nu\in\widehat{\Integral}$, there is a well-defined conjugacy class
$\hat{\mathbf{c}}^\nu$ in $\widehat{\pi}$.
This can be defined as the inverse limit of $\mathbf{c}_\Gamma^n$,
where $\mathbf{c}_\Gamma$ denotes the residual conjugacy class of $\mathbf{c}$ in 
any finite quotient group $\Gamma$ of $\pi$,
and where $n\in\Integral$ is congruent to $\nu$ modulo $|\Gamma|$.

\begin{lemma}\label{profinite_correspondence_pt}
Under any profinite isomorphism setting of pseudo-Anosov mapping tori
$(M_A,M_B,\Psi)$ (Convention \ref{profinite_isomorphism_setting_pA}),
$\Psi$
witnesses a bijective correspondence between the periodic trajectories 
in $M_A$ and in $M_B$ of their suspension flows, as follows.

For any periodic trajactory $\mathbf{c}_A$ of $M_A$,
there exists a unique periodic trajectory $\mathbf{c}_B$ of $M_B$,
such that $\hat{\mathbf{c}}_A$ projects $\hat{\mathbf{c}}_B^\mu$
under $\Psi$.
Here, periodic trajectories are treated
as conjugacy classes in the fundamental group.
Hence, $\hat{\mathbf{c}}_B$ projects $\hat{\mathbf{c}}_A^{1/\mu}$
under the inverse of $\Psi$.
\end{lemma}

\begin{proof}
	In any profinite group $\Pi$, 
	we say that two conjugacy classes $\mathbf{C}$ and $\mathbf{C}'$
	are \emph{$\widehat{\Integral}^\times$--equivalent}
	if $\mathbf{C}^{\mu'}=\mathbf{C}'$ holds
	for some $\mu'\in\widehat{\Integral}^\times$.
	Denote by $\mathrm{Orb}(\Pi)$ the set of conjugacy classes in $\Pi$,
	and denote by $\Omega(\Pi)$ the set of $\widehat{\Integral}^\times$--equivalence classes
	in $\mathrm{Orb}(\Pi)$.
	Note that if $\Gamma$ is a finite quotient group of $\Pi$,
	then every $\widehat{\Integral}^\times$--equivalence class in $\mathrm{Orb}(\Pi)$
	projects a unique $\widehat{\Integral}^\times$--equivalence class in $\mathrm{Orb}(\Gamma)$,
	so the quotient homomorphism $\Pi\to \Gamma$ naturally induces 
	a quotient map of sets $\Omega(\Pi)\to\Omega(\Gamma)$.
		
	One may state the condition in Lemma \ref{profinite_correspondence_pt} alternatively
	as 
	$\Psi(\hat{\mathbf{c}}_A)$ to be $\widehat{\Integral}^\times$--equivalent to $\hat{\mathbf{c}}_B$.
	In fact, if $\Psi(\hat{\mathbf{c}}_A)=\hat{\mathbf{c}}_B^{\mu'}$
	holds for some $\mu'\in\widehat{\Integral}^\times$,
	the following equation holds in $\widehat{\Integral}$
	by the commutative diagram (\ref{Sigma_mu_diagram}):
	$$\mu\cdot\phi_A(\mathbf{c}_A)=
	\mu\cdot\widehat{\phi}_A\left(\hat{\mathbf{c}}_A\right)=
	\widehat{\phi}_B\left(\Psi(\hat{\mathbf{c}}_A)\right)=
	\widehat{\phi}_B\left(\mu'\cdot\hat{\mathbf{c}}_B\right)=
	\mu'\cdot\widehat{\phi}_B\left(\hat{\mathbf{c}}_B\right)=
	\mu'\cdot\phi_B(\mathbf{c}_B).$$
	Note that $\phi_A(\mathbf{c}_A)$ and $\phi_B(\mathbf{c}_B)$
	are positive integers.
	Then
	$\phi_A(\mathbf{c}_A)=\phi_B(\mathbf{c}_B)$ holds in $\Integral$,
	and $\mu'=\mu$ holds in $\widehat{\Integral}^\times$.
	
	For any finite group $\Gamma$ and any $\widehat{\Integral}^\times$--equivalence class
	$\omega\in\Omega(\Gamma)$,
	we denote by $\chi_\omega$ 
	the characteristic function defined on $\mathrm{Orb}(\Gamma)$,
	such that
	$\chi_\omega(\mathbf{c})=1$ for all $\mathbf{c}\in\omega$,
	and $\chi_\omega(\mathbf{c})=0$ for all $\mathbf{c}\not\in\omega$.
	For any pseudo-Anosov automorphism $f\colon S\to S$ and 
	any homomorphism $\gamma\colon\pi_1(M_f)\to \Gamma$,
	we recall the following defining expression
	of the analogous twisted Lefschetz number,
	as in	\cite[Section 8, (8.7)]{Liu_profinite_almost_rigidity}:
	\begin{equation}\label{L_m_def}
	L_m(f;\gamma^*\chi_\omega)=
	\sum_{\mathbf{O}\in\mathrm{Orb}_m(f)} \chi_\omega(\gamma(\ell_m(f,\mathbf{O})))\cdot\mathrm{ind}_m(f;\mathbf{O}),
	\end{equation}
	(see (\ref{ell_m})).
	The factor $\mathrm{ind}_m(f;\mathbf{O})\in\Integral$ denotes the $m$--periodic index of $f$ at $\mathbf{O}$,
	(see \cite[Section 2, (2.4)]{Liu_profinite_almost_rigidity});
	for current purpose, it suffices to mention that $\mathrm{ind}_m(f;\mathbf{O})$
	is nonzero for every $\mathbf{O}\in\mathrm{Orb}_m(f)$.
	Note that the profinite isomorphism $\Psi$
	witnesses a bijective correspondence between
	the homomorphisms $\gamma_A\colon\pi_A\to\Gamma$ and 
	$\gamma_B\colon\pi_B\to\Gamma$,
	such that $\widehat{\gamma}_A\colon\widehat{\pi}_A\to\Gamma$ is the $\Psi$--pullback
	of $\widehat{\gamma}_B\colon\pi_B\to\Gamma$.	
	For 
	any $\omega\in\Omega(\Gamma)$, and $m\in\Natural$,
	and
	any $\Psi$--corresponding homomorphism pair $\gamma_A$ and $\gamma_B$ as above,
	the following version of
	profinite invariance of twisted Lefschetz numbers follows from
	\cite[Theorems 2.7, 7.2, and Lemma 8.4]{Liu_profinite_almost_rigidity}:
	\begin{equation}\label{profinite_invariance_L_m}
	L_m\left(f_A;\gamma_A^*\chi_\omega\right)=L_m\left(f_B;\gamma_B^*\chi_\omega\right).
	\end{equation}
	
	For any $m\in\Natural$,
	there exists some finite group $\Gamma_m$
	and some quotient homomorphism $\gamma_{m,A}\colon\pi_A\to \Gamma_m$,
	such that the $m$--periodic trajectories of $M_A$
	project into mutually distinct $\widehat{\Integral}^\times$--equivalence classes
	in $\mathrm{Orb}(\Gamma_m)$ under $\gamma_{m,A}$ \cite[Lemma 8.6]{Liu_profinite_almost_rigidity}.
	Then (\ref{L_m_def}) and (\ref{profinite_invariance_L_m})
	imply that those $\widehat{\Integral}^\times$--equivalence classes
	also receive $m$--periodic trajectories of $M_B$
	under the $\Psi$--corresponding $\gamma_{m,B}$.
	Moreover, $M_B$ has the same number of $m$--periodic trajectories
	as $M_A$ does \cite[Theorem 8.1]{Liu_profinite_almost_rigidity}.
	Then the $m$--periodic trajectories of $M_B$
	must project exactly 
	the same set of $\widehat{\Integral}^\times$--equivalence classes
	in $\mathrm{Orb}(\Gamma_m)$ as those of $M_A$ do,
	(applying (\ref{profinite_invariance_L_m}) for every $\omega\in\Omega(\Gamma_m)$).
	This property persists as we pass to 
	any finer $\Psi$--corresponding quotients,
	(that is,
	$\pi_A\to \Gamma$ and $\pi_B\to \Gamma$
	which descend to $\gamma_{m,A}$ and $\gamma_{m,B}$ 
	under a further quotient $\Gamma\to\Gamma_m$).
	Then the property also persists as we pass to the profinite completions
	$\pi_A\to\widehat{\pi}_A$ and $\pi_B\to\widehat{\pi}_B$.
	
	To summarize,
	we have established bijective correspondences between the $m$--periodic trajectories 
	$\mathbf{c}_A=\ell_m(f_A;\mathbf{O}_A)$ of $M_A$
	and $\mathbf{c}_B=\ell_m(f_B;\mathbf{O}_B)$ of $M_B$,
	for any $m\in\Natural$.
	They satisfy the condition that $\Psi(\hat{\mathbf{c}}_A)$ 
	is $\widehat{\Integral}^\times$--equivalent to $\hat{\mathbf{c}}_B$,
	as asserted.
\end{proof}

Next,
we draw another consequence of the $\widehat{\Integral}^\times$--regularity theorem
for profinite isomorphisms between closed hyperbolic $3$-manifold groups 
\cite[Theorem 1.2]{Liu_profinite_almost_rigidity}.
The following lemma is needed in Section \ref{Sec-profin_invar_vol}
for proving Lemma \ref{profinite_invariance_vol_pA}.
In the statement, the notation $H_*(\pi_{A/B};\Integral)$ denotes the group homology
(naturally isomorphic to the homology of $M_{A/B}$),
and the notation $H_*^{\mathtt{cont}}(\widehat{\pi}_{A/B};\widehat{\Integral})$
denotes the continuous group homology
(also known as the profinite group homology in \cite{Ribes--Zalesskii_book}).	

\begin{lemma}\label{Mu_cube}
Let $(M_A,M_B,\Psi)$ be
a profinite isomorphism setting of pseudo-Anosov mapping tori,
(Convention \ref{profinite_isomorphism_setting_pA}).
Then, 
$\Psi_*\colon H_3^{\mathtt{cont}}(\widehat{\pi}_A;\widehat{\Integral})\to H_3^{\mathtt{cont}}(\widehat{\pi}_B;\widehat{\Integral})$
takes the form 
$$\Psi_*[M_A]=\pm\mu^3\cdot[M_B].$$
Here, $[M_A]\in H_3(\pi_A;\Integral)$ and $[M_B]\in H_3(\pi_B;\Integral)$ 
are fundamental classes of $M_A$ and $M_B$,
with respect to arbitrarily fixed orientations,
and also treated as in
$H_3^{\mathtt{cont}}(\widehat{\pi}_{A/B};\widehat{\Integral})$
via the natural inclusion;
$\mu\in\widehat{\Integral}^\times$ is the same as in Convention \ref{profinite_isomorphism_setting_pA}.
\end{lemma}

\begin{proof}
We start by mentioning some simple facts regarding profinite group homology
of finitely generated free abelian groups.

For any finitely generated free abelian groups 
$H_A=\Integral u_1\oplus \cdots\oplus \Integral u_a$ and $H_B=\Integral v_1\oplus \cdots\oplus\Integral v_b$,
we have natural identifications of graded $\widehat{\Integral}$--modules
$H_*^{\mathtt{cont}}(\widehat{H}_A;\widehat{\Integral})\cong \Lambda_\Integral(u_1,\cdots,u_a)\otimes_\Integral\widehat{\Integral}$
and
$H_*^{\mathtt{cont}}(\widehat{H}_B;\widehat{\Integral})\cong \Lambda_\Integral(v_1,\cdots,v_b)\otimes_\Integral\widehat{\Integral}$,
where $\Lambda_\Integral$ denotes the exterior $\Integral$--algebra (with generators of degree $1$),
and where  $H_*^{\mathtt{cont}}$ denotes the profinite group homology.
Any group homomorphism $\Phi\colon \widehat{H}_A\to \widehat{H}_B$
is naturally identified as an element
$\Phi\in \mathrm{Hom}(H_A,H_B)\otimes_\Integral\widehat{\Integral}$.
Moreover,
the induced homomorphism
$\Phi_*\colon H_*^{\mathtt{cont}}(\widehat{H}_A;\widehat{\Integral})\to H_*^{\mathtt{cont}}(\widehat{H}_B;\widehat{\Integral})$
can be naturally identified with the induced homomorphisms of graded $\widehat{\Integral}$--module,
which preserves the exterior product,
namely, $\Phi_*(u_{k_1}\wedge\cdots u_{k_j})=\Phi(u_{k_1})\wedge\cdots\wedge\Phi(u_{k_j})$ for all $j$.

In fact, the identification of $H_*^{\mathtt{cont}}(\widehat{H}_{A/B};\widehat{\Integral})$
follows from the fact that that finitely generated free abelian group are homologically good,
which implies 
$H_*^{\mathtt{cont}}(\widehat{H}_{A/B};\widehat{\Integral})\cong H_*(H_{A/B};\Integral)\otimes_\Integral\widehat{\Integral}$;
the identification of $\Phi_*$ can be follows easily
from the K\"unneth isomorphism
$H_*^{\mathtt{cont}}(\widehat{\Integral} u_1; \widehat\Integral)\otimes_{\widehat\Integral}\cdots\otimes_{\widehat{\Integral}}
H_*^{\mathtt{cont}}(\widehat{\Integral} u_a; \widehat\Integral)
\cong
H_*^{\mathtt{cont}}(\widehat{H}_A;\widehat{\Integral})$ in this case,
and similarly with $\widehat{H}_B$,
and the isomorphism can be directly set up
by using inclusions of and projections onto the direct summands, just as usual.

With the above identifications understood, 
we turn to the setting $(M_A,M_B,\Psi)$ of Lemma \ref{Mu_cube}.
Denote by $H_A$ and $H_B$ the free abelianization of $\pi_A$ and $\pi_B$, respectively.
Denote by $\Phi\colon \widehat{H}_A\to \widehat{H}_B$
the profinite isomorphism induced by $\Psi\colon\widehat{\pi}_A\to \widehat{\pi}_B$.
Fix auxiliary fundamental classes $[M_A]\in H_3(M_A;\Integral)\cong H_3(\pi_A;\Integral)$
and similarly $[M_B]$.

The key result \cite[Theorem 1.2]{Liu_profinite_almost_rigidity}
shows that $\Phi$ must take the form $\mu\cdot \widehat{F}$
for some $\mu\in\widehat{\Integral}^\times$ and some $F\in \mathrm{Isom}(H_A,H_B)$.
Since the classes of periodic trajectories in $M_{A/B}$ 
span $H_{A/B}$ (Lemma \ref{normal_generation_pt})
and arise in correspondence (Lemma \ref{profinite_correspondence_pt}),
we infer that $\mu$ is 
up to sign the same as provided in Convention \ref{profinite_isomorphism_setting_pA}.
We obtain the commutative diagram
$$
\xymatrix{
H_3^{\mathtt{cont}}(\widehat{\pi}_A;\widehat{\Integral}) \ar[r] \ar[d]_{\cong}^{\Psi_*} & 
H_3^{\mathtt{cont}}(\widehat{H}_A;\widehat{\Integral}) \ar[d]_{\cong}^{\Phi_*} \\
H_3^{\mathtt{cont}}(\widehat{\pi}_B;\widehat{\Integral}) \ar[r] & H_3^{\mathtt{cont}}(\widehat{H}_B;\widehat{\Integral}) 
}$$
Under the natural identification as we explained, $H_3^{\mathtt{cont}}(\widehat{H}_{A/B};\widehat{\Integral})\cong 
(\Lambda^3_\Integral H_{A/B})\otimes_\Integral \widehat{\Integral}$,
the induced isomorphism
$\Phi_*\colon H_3^{\mathtt{cont}}(\widehat{H}_{A};\widehat{\Integral})\to H_3^{\mathtt{cont}}(\widehat{H}_{B};\widehat{\Integral})$ 
takes the form
$\wedge^3\Phi=\mu^3\cdot (\wedge^3\widehat{F})$ as a $\widehat{\Integral}$--linear isomorphism
$(\Lambda^3_\Integral H_{A})\otimes_\Integral \widehat{\Integral}\to
(\Lambda^3_\Integral H_{B})\otimes_\Integral \widehat{\Integral}$,
since $\Phi_*$ respects the exterior algebra structure in this case.

In order to prove the lemma, 
observe 
$H_3^{\mathtt{cont}}(\widehat{\pi}_A;\widehat{\Integral})\cong
 H_3^{\mathtt{cont}}(\widehat{\pi}_B;\widehat{\Integral})\cong \widehat{\Integral}$.
Then, for some $\lambda\in\widehat{\Integral}^\times$,
we have the relation
$$\Psi_*[M_A]=\lambda\cdot[M_B].$$
The commutativity of the above diagram implies 
$\Phi_*[M_A]_H=\lambda\cdot [M_B]_H$, where the subscript $H$ means passing to
$H_3^{\mathtt{cont}}(H_{A/B};\widehat{\Integral})$. 
Under the natural identification as explained,
the above equation can be rewritten as
$$\mu^3\cdot ((\wedge^3 F)\,[M_A]_H)=\lambda\cdot [M_B]_H,$$
as an equation in $(\Lambda^3_\Integral H_B)\otimes_\Integral \widehat\Integral$.
Note that both $[M_B]_H$ and $(\wedge^3 F)[M_A]_H$ lie 
in the natural $\Integral$--submodule $\Lambda^3_\Integral H_B$,
(so we can simply write $\wedge^3 F$ instead of $\wedge^3 \widehat{F}$ without affecting the result).

Let us first consider the special case where $[M_A]_H\in \Lambda^3_{\Integral} H_A$ is nonzero.
In this case, $[M_B]_H$ and $(\wedge^3F)\,[M_A]_H$ are both nonzero,
so we obtain 
$$\lambda=\pm\mu^3$$
as desired,
(using the fact $\widehat{\Integral}^\times\cap\Integral=\{\pm1\}$).

In order to prove the general case,
we make use of the fact that some finite cover $M'_A\to M_A$ satisfies 
the condition that $[M'_A]_H\in \Lambda^3_{\Integral} H'_A$ is nonzero.
This condition holds whenever the triple product 
$H^1(M'_A;\Integral)\times H^1(M'_A;\Integral)\times H^1(M'_A;\Integral)\to H^3(M'_A;\Integral)
\colon (\alpha,\beta,\gamma)\mapsto \alpha\smile\beta\smile\gamma$
does not vanish.
For example, this is the case when $M'_A$ admits a nonzero degree map onto
the $3$--torus $S^1\times S^1\times S^1$,
(because the cohomology ring $H^*(S^1\times S^1\times S^1;\Integral)$
embeds into $H^*(M'_A;\Integral)$ in this case, 
as a standard exercise of the Poincar\'e duality).
Such a finite cover $M'_A$ always exists, 
and indeed,
it is known that any closed hyperbolic $3$-manifold admits a finite cover
that maps onto any prescribed closed orientable $3$--manifold
with any prescribed nonzero degree \cite{Liu--Sun_domination}. 

Take a finite cover $\kappa_A\colon M'_A\to M_A$ for which $[M'_A]_H\in \Lambda^3_{\Integral} H'_A$ is nonzero.
Let $\kappa_B\colon M'_B\to M_B$ be the corresponding finite cover with respect to $\Psi$,
and $\Psi'\colon \widehat{\pi}'_A\to\widehat{\pi}'_B$ be the restriction of $\Psi$.
For the induced profinite isomorphism setting $(M'_A,M'_B,\Psi')$,
the fibered classes $\phi'_A$ and $\phi'_B$ are just the pull-backs $\kappa_A^*\phi_A$ and $\kappa_B^*\phi_B$,
the irregularity factor $\mu'$ remains the same as $\mu\in\widehat{\Integral}^\times$.
Applying the special case, we obtain $\Psi'_*[M'_A]=\pm \mu^3\cdot[M'_B]$.
Using the relation $\Psi_*\circ\kappa_{A*}=\kappa_{B*}\circ\Psi'_*$
and $[M'_A:M_A]=[M'_B:M_B]=d$,
we obtain the relation $\lambda d=\pm \mu^3d$ in $\widehat{\Integral}$,
which implies $\lambda=\pm \mu^3$ as desired,
(using the fact any nonzero $d\in\Integral$ is not a zero divisor in $\widehat{\Integral}$).
\end{proof}

\section{On the profinite power operation}\label{Sec-profin_power_operation}
In this section, we prove some algebraic lemmas
(Lemmas \ref{mu_square} and \ref{A_mu}), 
regarding the profinite power operation
$t\mapsto t^\nu$ in the profinite cyclic group $t^{\widehat{\Integral}}$
and $A\mapsto A^\nu$ in the matrix group $\mathrm{SL}(2,R)$ over some profinite ring $R$,
where $\nu$ is a profinite integer, namely, $\nu\in\widehat{\Integral}$.
These lemmas are prepared for Sections \ref{Sec-mu_square} and \ref{Sec-profin_corresp_rep}.

\begin{lemma}\label{mu_square}
	Let $f(t)\in\Integral[t]$ be a monic reciprocal polynomial
	in an indeterminate $t$. 
	Denote by $\llbracket\widehat{\Integral}t^{\widehat{\Integral}}\rrbracket$
	the profinite ring completion of 
	the group ring $[\Integral t^\Integral]=\Integral[t,t^{-1}]$.	
	Suppose that 
	the principal ideals $(f(t))$ and $(f(t^\mu))$ are equal in 
	$\llbracket\widehat{\Integral}t^{\widehat{\Integral}}\rrbracket$
	for some $\mu\in\widehat{\Integral}^\times$,
	
	Then, either $f(t)$ is a product of cyclotomic factors, 
	or $\mu^2=1$ holds in $\widehat{\Integral}^\times$.
\end{lemma}

\begin{proof}
	For any rational prime $p$,
	let $L$ be an algebraic field extension over $\Rational_p$ of finite degree.
	Denote by $|\cdot|_p$ the unique $p$--adic norm on $L$ 
	extending that of $\Rational_p$.
	Denote by $R=\{x\in L\colon |x|_p\leq 1\}$ the valuation ring of $L$, 
	and $R^\times=\{x\in L\colon |x|_p=1\}$ the multiplicative group of units in $R$.
	
	For any $\alpha\in R^\times$, we obtain a unique continuous ring homomorphism
	$\mathrm{ev}_\alpha\colon\llbracket\widehat{\Integral}t^{\widehat{\Integral}}\rrbracket\to R$,
	such that $t\mapsto \alpha$.
	The principal ideal $(f(t))$ is contained in the kernel of the $\mathrm{ev}_\alpha$
	if and only if $f(\alpha)=0$ holds in $L$.
	Since $f$ is monic and reciprocal, any zero of $f$ in $L$
	is an algebraic integer, hence lying in $R^\times$. 
	The ideal equation $(f(t))=(f(t^\mu))$ in $\llbracket\widehat{\Integral}t^{\widehat{\Integral}}\rrbracket$
	implies that $\alpha^{1/\mu}$ is a zero of $f$ in $L$
	if and only if $\alpha$ is a zero of $f$ in $L$.
	Since there are only finitely many zeros,
	there exists some $d\in\Natural$,
	such that the group automorphism
	$R^\times\to R^\times\colon \alpha\mapsto \alpha^{\mu^d}$ 
	fixes every zero of $f$ in $L$.
	
	Below we assume that $f$ is not a product of cyclotomic factors.
	Fix a root $\xi$ of $f(t)\in\Integral[t]$ that is not a root of unity.
	We work with the field $L=\Rational_p(\xi)$ for an arbitrary $p$,
	fixing an embedding of $L$ into $\Complex_p$
	(the $p$--adic completion of an algebraic closure of $\Rational_p$).
	There exists some $m\in\Natural$, such that $|\alpha^m-1|_p<p^{-1}$
	holds for all $\alpha\in R^\times$.
	(For example, $m$ can be taken as $p^k\cdot(p^l-1)$
	where $l=[L:\Rational_p]$ and $k=\lfloor \log l/\log p\rfloor +1$;
	see Remark \ref{mu_square_remark}.)
	Since the multiplicative subgroup $1+pR$ of $R^\times$ is pro--$p$,
	the power action of $\widehat{\Integral}$ on $1+pR$
	factors through the natural projection onto $\Integral_p$.
	Therefore,
	$\xi^{ma}=\xi^{ma_p}$ holds for any $a\in\widehat{\Integral}$,
	where $a_p\in\Integral_p$ denotes the $p$--summand of $a$.
	Therefore, we obtain an equation in $\Complex_p$:
	$$\xi^m=\xi^{m\mu_p^d}.$$
	Since $\xi$ is not a root of unity,
	we infer $\mu_p^d=1$ in $\Integral_p$. 
	(Otherwise, the equation $(1+x)^\delta=1$ with $\delta=\mu^d_p-1$
	would have arbitrarily small solutions of the form $x=\xi^{mp^n}-1$,
	where $n\in\Natural$,
	contradicting the invertibility of $(1+x)^{\delta}$ near $x=0$ in $\Complex_p$.)	
	This means $\mu_p$ is a root of unity.
	(Compare \cite[Proposition 1.3]{Ueki_TAP}.)
	
	We have observed that $\eta=\xi^\mu$ is again some zero of $f$ in $L$.
	Raising to the $m$--th power, we obtain
	$$\xi^{m\mu_p}=\eta^m.$$
	For any algebraic integers $\alpha,\beta\in\Complex_p\setminus\{0,1\}$,
	such that $|\alpha-1|_p<p^{-1}$ and $|\beta-1|_p<p^{-1}$,
	the $p$--adic logarithmic ratio $\log_p\alpha/\log_p\beta$
	is either rational or transcendental,
	(or equivalently, $\alpha^\theta=\beta$
	implies $\theta$ rational or transcendental in $\Complex_p$).
	This is known as Mahler's $p$--adic analogue 
	of the Gelfond--Schneider theorem,
	\cite{Mahler_GS}	(see also \cite{Veldkamp}).
	Since $\xi$ is not a root of unity,
	we infer that $\mu_p\in\Integral_p^\times$ must be a rational integer,
	and hence $\pm1$.
	(Besides, 
	$\eta$ should be $\xi$ or $\xi^{-1}$ times some root of unity in $L$.)
	In particular, $\mu_p^2=1$ holds in $\Integral_p^\times$.
	Since $p$ is arbitrary,
	we obtain $\mu^2=1$ 
	in $\widehat{\Integral}^\times=\prod_p\Integral_p^\times$,
	as asserted.	
\end{proof}

\begin{remark}\label{mu_square_remark}
	Let $p$ be a rational prime and
	$L$ be a field extension $L$ over $\Rational_p$ of finite algebraic degree.
	Denote by $R$ the $p$--adic valuation ring of $L$
	and $\mathfrak{M}$ the maximal ideal of $R$.
	For any $\lambda\in R^\times$,
	there is a unique factorization $\lambda=(1+\beta)\cdot\zeta$,
	where $|\beta|_p<1$, and	
	where $\zeta$ is a root of unity of order coprime to $p$.
	This agrees with the canonical decomposition
	of the multiplicative group
	$R^\times\cong (1+\mathfrak{M})\times (R/\mathfrak{M})^\times$.
	Note that the additive coset $1+\mathfrak{M}$ forms
	a multiplicative pro--$p$ subgroup.
	Then for any $a\in\widehat{\Integral}$, 
	we obtain a formula
	$$\lambda^a=(1+\beta)^{a_p}\cdot \zeta^{a \bmod (q-1)}
	=\left[\sum_{m=0}^{\infty}\frac{a_p\cdot(a_p-1)\cdot\cdots\cdot(a_p-m+1)}{m!}\cdot\beta^m\right]\cdot\zeta^{a \bmod (q-1)},$$
	where $q$ denotes the cardinality of the residue field $R/\mathfrak{M}$.
	Since $a_p\in\Integral_p$,
	the above $p$--adic binomial series 
	converges to a continuous function of $\beta\in\mathfrak{M}$,
	(see \cite[Chapter IV, Section 1]{Koblitz_book}).
	The function agrees with the $\Integral_p$--power operation 
	$1+\beta\mapsto(1+\beta)^{a_p}$ on $1+\mathfrak{M}$,
	because	for every $a_p\in\Natural$,
	the agreement is obvious by the usual binomial expansion.	
\end{remark}

\begin{lemma}\label{A_mu}
	Let $p$ be a rational prime and 
	$L$ be a field extension $L$ over $\Rational_p$ of finite algebraic degree.
	Denote by $R$ the $p$--adic valuation ring of $L$
	and $\mathfrak{M}$ the maximal ideal of $R$.
	
	If $A\in\mathrm{SL}(2,R)$ and $\mu\in\widehat{\Integral}^\times$
	satisfy $A\equiv I\bmod \mathfrak{M}$, and $\mu^2=1$,
	then $\mathrm{tr}(A^\mu)=\mathrm{tr}(A)$.
\end{lemma}

\begin{proof}
	The principal $\mathfrak{M}$--congruence subgroup $I+\mathrm{Mat}(2,\mathfrak{M})$
	of $\mathrm{GL}(2,R)$ is pro--$p$.
	In fact, for all $n\in\Natural$, the quotient groups of
	$I+\mathrm{Mat}(2,\mathfrak{M}^n)$ by $I+\mathrm{Mat}(2,\mathfrak{M}^{n+1})$
	are all cosets of finite-dimensional vector spaces 
	over the residue field of characteristic $p$, 
	affinely isomorphic to $\mathrm{Mat}(2,\mathfrak{M}^n)\otimes_{R}(R/\mathfrak{M})$.
	It follows that the intersection of $\mathrm{SL}(2,R)$ with $I+\mathrm{Mat}(2,\mathfrak{M})$
	is also pro--$p$.
	
	If $A\equiv I\bmod \mathfrak{M}$,
	we obtain $A^\mu=A^{\mu_p}$ in $\mathrm{SL}(2,R)$,
	where $\mu_p$ denotes the $\Integral_p$--summand of $\mu$ in $\widehat{\Integral}$.
	If $\mu^2=1$, we obtain $\mu_p=\pm1$ in $\Integral_p$.
	Then 
	$$\mathrm{tr}(A^\mu)=
	\mathrm{tr}(A^{\mu_p})=\mathrm{tr}(A^{\pm1})=\mathrm{tr}(A),$$
	the last step using $\mathrm{det}(A)=1$.
\end{proof}

\section{Mu square equals one}\label{Sec-mu_square}
In this section, we improve a main result in \cite{Liu_profinite_almost_rigidity},
obtaining more information about the $\widehat{\Integral}^\times$--regularity factor $\mu$.


\begin{lemma}\label{p_regularity}
Let $(M_A,M_B,\Psi)$ be
a profinite isomorphism setting of pseudo-Anosov mapping tori,
(Convention \ref{profinite_isomorphism_setting_pA}).
Then, $\mu^2=1$ holds in $\widehat{\Integral}^\times$.
\end{lemma}

\begin{proof}
	Let $M_f$ be the mapping torus of any pseudo-Anosov automorphism $f\colon S\to S$
	of a closed orientable surface. For any quotient homomorphism $\pi_1(M_f)\to\Gamma$
	onto a finite group $\Gamma$, the kernel of the homomorphism corresponds to 
	a regular finite cover $M'\to M_f$. 
	The suspension flow lifts to $M'$, transverse to any preimage component $S'$ of $S$,
	so the return map of the lifted flow gives rise to a pseudo-Anosov automorphism $f'\colon S'\to S'$.
	This way, $M'$ can be identified as the mapping torus $M_{f'}$.
	We denote by $\phi_f\in H^1(M_f;\Integral)$ and $\phi_{f'}\in H^1(M_{f'};\Integral)$
	the distinguished cohomology classes dual to the distinguished fibers $S$ and $S'$,
	respectively. Then $[M':M]/[S':S]$ times $\phi_{f'}$ agrees with the pullback of $\phi_f$
	with respect to the covering projection $M_{f'}\to M_f$.
	
	In \cite{Liu_vhsr}, it is shown that for any pseudo-Anosov mapping torus $M_f$ as above,
	there exists a covering mapping torus $M'=M_{f'}$, such that the Alexander polynomial
	$\Delta_{M'}^{\phi'}\in\Integral[t,t^{-1}]$ (up to a factor in 
	the group of units $\{\pm1\}\times t^{\Integral}$)
	with respect to $\phi'=\phi_{f'}$ has a root in $\Complex$ outside the unit circle.
	In the case of mapping tori,
	the Alexander polynomial $\Delta_{M'}^{\phi'}(t)$ can be identified
	with the characteristic polynomial of the cohomological action
	$f'_*\colon H_1(S';\Integral)\to H_1(S';\Integral)$,
	namely,
	$$\Delta_{M'}^{\phi'}(t)\doteq \mathrm{det}_{\Integral[t]}\left(t\cdot\mathrm{id}-f'_*\right),$$
	where $\doteq$ means an equality in $\Integral[t,t^{-1}]$ 
	up to a factor in $\{\pm1\}\times t^{\Integral}$.
	In particular, we may represent $\Delta_{M'}^{\phi'}(t)$
	with a unique monic reciprocal polynomial in $\Integral[t]$ of degree $\chi(S')+2$.
	The Reidemeister torsion $\tau_{M'}^{\phi'}(t)$ lives in 
	the field of fractions $\Rational(t)=\mathrm{Frac}(\Integral[t,t^{-1}])$
	up to a factor in $\{\pm1\}\times t^{\Integral}$.
	In the mapping torus case, $\tau_{M'}^{\phi'}(t)$ can be identified with
	the alternating product of the characteristic polynomials
	of $f'_*\colon H_*(S';\Integral)\to H_*(S';\Integral)$,
	namely,
	$$\tau_{M'}^{\phi'}(t)\doteq\Delta_{M'}^{\phi'}(t)\cdot(t-1)^{-2},$$
	(see \cite[Section 2.3 and Lemma 7.7]{Liu_profinite_almost_rigidity}, 
	for example).
	
	Under a profinite isomorphism setting $(M_A,M_B,\Psi)$ 
	of pseudo-Anosov mapping tori (Convention \ref{profinite_isomorphism_setting_pA}),
	we obtain a similar profinite isomorphism 
	setting of the covering mapping tori
	$(M'_A,M'_B,\Psi')$,
	associated to any $\Psi$--corresponding pair of finite quotients
	$\pi_A\to \Gamma$ and $\pi_B\to \Gamma$.
	(Being $\Psi$--corresponding means 
	that the completion homomorphism $\widehat{\pi}_A\to\Gamma$
	is the $\Psi$--pullback of $\widehat{\pi}_B\to \Gamma$;
	so $M'_A$ and $M'_B$ are the finite covers corresponding to 
	the quotient kernels $\pi'_A$ and $\pi'_B$,
	and $\Psi'\colon\widehat{\pi}'_A\to \widehat{\pi}'_B$ 
	is the restriction of $\Psi$ to their profinite completions.)
	The covering setting can be summarized with a similar diagram as follows:
	\begin{equation}\label{covering_Sigma_mu_diagram}
	\xymatrix{
	\{1\} \ar[r] & \widehat{\Sigma}'_A \ar[rr] \ar[d]_{\cong}^{\Psi'|} 
	& & \widehat{\pi}'_A \ar[rr]^{t^{\widehat{\phi}'_A}} \ar[d]_{\cong}^{\Psi'}
	& & t^{\widehat{\Integral}} \ar[r] \ar[d]^{\mathrm{pwr}_{\mu'}}_{\cong} & \{1\} \\
	\{1\} \ar[r] & \widehat{\Sigma}'_B \ar[rr] & & \widehat{\pi}'_B \ar[rr]^{t^{\widehat{\phi}'_B}}
	& & t^{\widehat{\Integral}} \ar[r] & \{1\}
	}
	\end{equation}
	We can identify $\mu'\in \widehat{\Integral}^\times$ 
	in the diagram (\ref{covering_Sigma_mu_diagram}),
	with $\mu\in\widehat{\Integral}^\times$ in the diagram (\ref{Sigma_mu_diagram}),
	because they are related via the following parallel commutative diagrams:
	$$
	\xymatrix{
	\widehat{\pi}'_A \ar[rr]^{t^{\widehat{\phi}'_A}} \ar[d]_{\mathrm{incl}}
	& & t^{\widehat{\Integral}} \ar[d]^{\mathrm{pwr}_{k}} 
	& & &
	\widehat{\pi}'_B \ar[rr]^{t^{\widehat{\phi}'_B}} \ar[d]_{\mathrm{incl}}
	& & t^{\widehat{\Integral}} \ar[d]^{\mathrm{pwr}_{k}} 
	\\
	\widehat{\pi}_A \ar[rr]^{t^{\widehat{\phi}_A}}
	& & t^{\widehat{\Integral}}
	& & &
	\widehat{\pi}_B \ar[rr]^{t^{\widehat{\phi}_B}}
	& & t^{\widehat{\Integral}}
	}	
	$$	
	where $k=[M_A':M_A]/[S_A':S_A]=[M_B':M_B]/[S_B':S_B]$.
	
	By \cite[Theorem 7.2]{Liu_profinite_almost_rigidity},
	applied with the trivial representation $\Gamma\to\mathrm{GL}(1,\Integral)$,
	the Reidemeister torsions of
	profinitely corresponding pairs of pseudo-Anosov mapping tori are equal.
	We obtain 
	$$\tau^{\phi'_A}_{M'_A}\doteq\tau^{\phi'_B}_{M'_B},$$
	and hence 
	$$\Delta^{\phi'_A}_{M'_A}\doteq\Delta^{\phi'_B}_{M'_B}.$$
	(Note that our notation $\phi'_A\in H^1(M'_A;\Integral)$ corresponds to the notation $\psi'_A\in H^1(M'_A;\Integral)$ 
	in \cite{Liu_profinite_almost_rigidity}.)
	We represent both sides with a common monic reciprocal polynomial $P(t)\in\Integral[t]$.
	Moreover, there is an equality of the principal ideals in $\llbracket \widehat{\Integral}t^{\widehat{\Integral}}\rrbracket$:
	$$(P(t^{\mu}))=(P(t)),$$
	(see \cite[Lemmas 7.7 and 7.8]{Liu_profinite_almost_rigidity}).
	
	We apply the above results to some covering setting
	$(M'_A,M'_B,\Psi')$,
	where $P(t)$ has a root in $\Complex$ outside the unit circle.
	Then, by Lemma \ref{mu_square},
	we see that $\mu\in\widehat{\Integral}^\times$ is multiplicative $2$--torsion,
	as asserted.
\end{proof}

\section{Representations in quaternions}\label{Sec-rep_qa}
In this section, we introduce some working terminology regarding 
group representations in an abstract quaternion algebra.
Most of the materials below are well-known in the case of $\mathrm{SL}(2,\Complex)$--representations.
The main purpose of this section is 
to justify the formulation 
of results in Sections \ref{Sec-profin_corresp_rep} and \ref{Sec-profin_invar_vol}.



\begin{definition}\label{qarep_def}
Let $F$ be any (abstract) field of characteristic $0$,
and $\mathscr{A}$ be a quaternion algebra over $F$.
For any group $\pi$, 
we call any group homomorphism
$\rho\colon \pi\to \mathrm{SL}(1,\mathscr{A})$
a (\emph{special linear} or \emph{unimodular}) \emph{representation} 
of $\pi$ in $\mathscr{A}$.
Here, $\mathrm{SL}(1,\mathscr{A})$
denotes the normal subgroup of $\mathrm{GL}(1,\mathscr{A})=\mathscr{A}^\times$
which consists of the quaternions of norm $1$.
The \emph{character} of $\rho$ 
refers to the $\pi$--conjugacy invariant function 
$\chi_\rho\colon \pi\to F$,
defined as
$\chi_\rho(g)=\mathrm{Tr}_{\mathscr{A}/F}(\rho(g))$
for all $g\in\pi$.
\end{definition}

%
%

\begin{lemma}\label{irr_qarep_justification}
	With $\mathscr{A}/F$ as in Definition \ref{qarep_def},
	let $\pi$ be a group and
	$\rho\colon\pi\to\mathrm{SL}(1,\mathscr{A})$ be a representation.
	Then the following conditions are all equivalent.
	\begin{enumerate}
	\item 
	For some (hence any) field extension $E$ of $F$
	over which $\mathscr{A}\otimes_FE$ 
	splits as a matrix algebra $\mathrm{Mat}(2,E)$,
	the induced representation $\rho_E\colon \pi\to \mathrm{SL}(2,E)$ 
	is irreducible over $E$,
	that is, 
	$\rho_E$ has no proper $E$--linear invariant subspace in ${E^{2}}$ other than $\{0\}$.
	\item 
	The image of $\rho$ generates $\mathscr{A}$ as an algebra over $F$.
	\item
	The inequality $\chi_\rho([g,h])\neq2$ holds
	for some $g,h\in\pi$,
	where $[g,h]=ghg^{-1}h^{-1}$ denotes the commutator.
	\end{enumerate}
\end{lemma}

\begin{proof}
	The equivalence (1) $\Leftrightarrow$ (3) is well-known.
	We refer to \cite[Lemma 1.2.1 and Corollary 1.2.2]{CS_rep_var} for a standard proof.
	The implication (2) $\Rightarrow$ (1) is obvious (see Example \ref{qa_example}).
	It remains to prove the implication (1) $\Rightarrow$ (2),
	or essentially,
	the following (well-known) claim in linear algebra:
	If $E$ is an algebraically closed field of characteristic $0$,
	then any proper $E$--subalgebra $B$ of $\mathrm{Mat}(2,E)$ 
	preserves some $1$--dimensional $E$--linear subspace in ${E^{2}}$.
	
	When $\mathrm{dim}_EB<3$, 
	the claim is obvious, since $B$ is central, and hence commutative.
	When $\mathrm{dim}_EB=3$,
	we take a basis $I,X,Y\in\mathrm{Mat}(2,E)$ of $B$ over $E$,
	where $I$ is the identity matrix.
	Possibly after a generic modification of $X$ within $B$,
	we may require $\mathrm{det}(X)\neq0$ and $\mathrm{tr}(X)^2-4\cdot\mathrm{det}(X)\neq0$.
	So, possibly after conjugating $B$ in $\mathrm{Mat}(2,E)$,
	we may assume that $X$ is diagonal with distinct diagonal entries.
	It follows that $Y$ has some nonzero off-diagonal entries.
	If the upper and lower off-diagonal entries of $Y$ 
	were both nonzero,
	we observe that $XY$ and $Y$ would be linearly independent.
	Moreover, the matrices $I,X,Y,XY$ would span $\mathrm{Mat}(2,E)$ over $E$,
	contradicting the assumption $\mathrm{dim}_EB=3$.
	This means $Y$ is either upper triangular or lower triangular.
	Therefore,
	$B$ is conjugate to a subalgebra of triangular matrices in $\mathrm{Mat}(2,E)$.
	It preserves a $1$--dimensional $E$--linear subspace in ${E^{2}}$, as claimed.
\end{proof}

\begin{definition}\label{irr_qarep_def}
	With $\mathscr{A}/F$ as in Definition \ref{qarep_def},
	we say that a representation $\rho\colon\pi\to\mathrm{SL}(1,\mathscr{A})$
	is \emph{irreducible} 
	if $\rho$ satisfies any of the equivalent conditions
	as in Lemma \ref{irr_qarep_justification}.
\end{definition}

\begin{lemma}\label{irr_qachi}
	With $\mathscr{A}/F$ as in Definition \ref{qarep_def},
	let $\pi$ be a finitely generated group and
	$\rho\colon\pi\to\mathrm{SL}(1,\mathscr{A})$ be an irreducible representation.
	If $\tilde{\rho}\colon \pi\to\mathrm{SL}(1,\mathscr{A})$ is a representation 
	with character $\chi_{\tilde{\rho}}=\chi_{\rho}$ on $\pi$,
	then $\tilde{\rho}$ is conjugate to $\rho$ in $\mathrm{GL}(1,\mathscr{A})=\mathscr{A}^\times$,
	namely, $\tilde{\rho}=\tau\circ\rho$ for some inner automorphism $\tau$ of $\mathscr{A}^\times$.
	The converse also holds true.
\end{lemma}

\begin{proof}
	We only prove the main assertion, as the converse is obvious.
	Since $\pi$ is finitely generated, 
	we may assume $F$ finitely generated over $\Rational$.
	We fix a field embedding of $F$ into $\Complex$, 
	and an isomorphism $\mathscr{A}\otimes_F\Complex\cong\mathrm{Mat}(2,\Complex)$.
	Since the induced irreducible $\mathrm{SL}(2,\Complex)$--representations
	$\rho_\Complex$ and $\tilde{\rho}_\Complex$ of the finitely generated group $\pi$
	have identical characters,
	they are conjugate in $\mathrm{GL}(2,\Complex)$,
	and indeed, conjugate in $\mathrm{SL}(2,\Complex)$
	\cite[Proposition 1.5.2]{CS_rep_var}.
	
	Take a finite generating set $g_1,\cdots,g_r$ of $\pi$.
	We set up $r$ equations in $\mathscr{A}$ with an unknown $T$,
	namely, $T\cdot\rho(g_i)=\tilde{\rho}(g_i)\cdot T$	for all $i\in\{1,\cdots,r\}$.
	It suffices to show that there exists some solution $T\in\mathscr{A}$
	with $\mathrm{Nr}_{\mathscr{A}/F}(T)\neq0$.
	To this end,
	we take a basis $X_0,X_1,X_2,X_3$ of $\mathscr{A}$ over $F$,
	and write $T=u_0\cdot X_0+u_1\cdot X_1+u_2\cdot X_2+u_3\cdot X_3$.
	Then the above equations becomes a linear system of $4r$ equations over $F$
	in the unknowns $u_0,u_1,u_2,u_3$.
	The $F$--solution space $V$ of these equations is a linear subspace in $\mathscr{A}$ over $F$,
	and the $\Complex$--solution space is the scalar extension $V\otimes_F\Complex$ in $\mathrm{Mat}(2,\Complex)$.
	Moreover, the norm $\mathrm{Nr}_{\mathscr{A}/F}$ on $\mathscr{A}$
	restricts to a $F$--quadratic form on $V$, 
	which agrees with the determinant function on $\mathrm{Mat}(2,\Complex)$ under our identification.	
	Since $\rho_\Complex$ and $\tilde{\rho}_\Complex$ are conjugate in $\mathrm{GL}(2,\Complex)$,
	we infer that $V$ is nontrivial and $\mathrm{Nr}_{\mathscr{A}/F}$ is nonvanishing on $V$.
	This means there exists some solution $T$ in $\mathrm{GL}(1,\mathscr{A})=\mathscr{A}^\times$
	as desired.
\end{proof}

%

Recall that the \emph{$2$--Frattini subgroup}
of any group $\pi$ is by definition the characteristic subgroup
generated by the commutators and the squares of elements in $\pi$.
In fact, the squares of elements already form a generating set.
(This is not true in general for Frattini subgroups associated with other primes $p$.)
In this paper, we denote the $2$--Frattini subgroup of any group $\pi$ as $\mathrm{Fratt}_2(\pi)$.

\begin{lemma}\label{Zd_qarep_justification}
	With $\mathscr{A}/F$ as in Definition \ref{qarep_def},
	let $\pi$ be a group
	and $\rho\colon\pi\to\mathrm{SL}(1,\mathscr{A})$ be a representation.
	Then the following conditions are all equivalent.
	\begin{enumerate}
	\item 
	For some (hence any) algebraically closed field extension $E$ of $F$
	over which $\mathscr{A}\otimes_FE$ 
	splits as a matrix algebra $\mathrm{Mat}(2,E)$,
	the induced representation $\rho_E\colon \pi\to \mathrm{SL}(2,E)$ 
	is Zariski dense,
	that is, the image of $\rho$ is dense 
	in the linear algebraic group $\mathrm{SL}(2,E)$ over $E$
	with respect to the Zariski topology.
	\item For every finite-index normal subgroup $\pi'$ of $\pi$,
	the representation $\rho'\colon\pi'\to \mathrm{SL}(1,\mathscr{A})$
	restricting $\rho$ to $\pi'$ is irreducible.
	\item 
	The image of $\rho$ is infinite,
	and
	the representation $\rho^{\natural}\colon\mathrm{Fratt}_2(\pi)\to \mathrm{SL}(1,\mathscr{A})$
	restricting $\rho$ to the $2$--Frattini subgroup	is irreducible.
	\end{enumerate}
\end{lemma}

%
%

\begin{proof}
	For any algebraically closed field $E$,
	the Zariski closure $G$ of the image of $\rho(\pi)$
	is a Zariski closed subgroup in $\mathrm{SL}(2,E)$ \cite[Chapter II, Proposition 7.4.A]{Humphreys_book_LAG}.
	If $\rho'\colon \pi'\to \mathrm{SL}(1,\mathscr{A})$ is reducible
	for some finite-index subgroup $\pi'$ of $\pi$,
	the Zariski closure $G'$ of $\rho'(\pi')$ in $\mathrm{SL}(2,E)$
	can be conjugated into the upper-triangular subgroup by some element of $\mathrm{SL}(2,E)$
	(Lemma \ref{irr_qarep_justification}),
	while $G'$ has finite-index in $G$.
	Then $G$ is not Zariski dense as it has dimension at most $2$ in $\mathrm{SL}(2,E)$.
	Conversely,
	if $G$ is not Zariski dense in $\mathrm{SL}(2,E)$,
	the connected component $G^o$ 
	(that is, the irreducible component of $G$ as an algebraic variety)
	has finite index in $G$.
	Moreover,
	the Lie algebra $\mathfrak{g}$ of $G$ has dimension at most $2$ 
	over $E$ \cite[Chapter III, Theorem 9.1]{Humphreys_book_LAG}, 
	so $\mathfrak{g}$ must be solvable 
	(see \cite[Chapter I, Section 1.4]{Humphreys_book_LA},
	and can be conjugated into the upper-triangular Lie subalgebra $\mathfrak{sl}(2,E)$,
	since $E$ is algebraically closed of characteristic $0$
	\cite[Chapter II, Section 4.1]{Humphreys_book_LA}.
	In turn, this implies that $G^o$ can be conjugated into the upper-triangular subgroup
	of $\mathrm{SL}(2,E)$,
	within $\mathrm{GL}(2,E)$ and hence also within $\mathrm{SL}(2,E)$
	\cite[Chapter V, Theorem 13.1]{Humphreys_book_LAG}.
	This means that
	the restricted representation $\rho'\colon\pi'\to \mathrm{SL}(1,\mathscr{A})$ 
	is reducible on the finite-index subgroup $\pi'=G^0\cap\pi$ of $\pi$ (Lemma \ref{irr_qarep_justification}).
	This proves the equivalence (1) $\Leftrightarrow$ (2).

	If $\rho_E$ is Zarisk dense, 
	the Zariski closure	of $\rho_E(\mathrm{Fratt}_2(\pi))$ 
	is an infinite, normal, Zariski closed subgroup of $\mathrm{SL}(2,E)$,
	which has to be $\mathrm{SL}(2,E)$ itself.
	It follows that $\rho$ has infinite image and	$\rho^\natural$ is irreducible.
	On the other hand, if $\rho_E$ is not Zariski dense,
	there exists a finite-index normal subgroup $\pi'$,
	such that the restricted representation 
	$\rho'\colon\pi'\to\mathrm{SL}(1,\mathscr{A})$ is reducible.
	Therefore, $\rho_E(\pi')$ fixes some points on the projective line 
	$E\mathbf{P}^1$ over $E$,
	on which $\mathrm{SL}(2,E)$ acts by projective $E$--linear transformations
	(Lemma \ref{irr_qarep_justification}).
	When $\rho_E(\pi')$ fixes only one or two points on $E\mathbf{P}^1$,
	we observe that $\rho_E(\pi)$ stabilizes the fixed point set.
	In this case, $\rho^{\natural}$ is reducible,
	since $\rho_E(\mathrm{Fratt}_2(\pi))$ fixes those fixed points of $\rho_E(\pi')$
	on $E\mathbf{P}^1$.
	When $\rho_E(\pi')$ fixes at least three points on $E\mathbf{P}^1$,
	it fixes $E\mathbf{P}^1$, 
	so $\rho_E(\pi')$ is contained in the center $\{\pm I\}$ of $\mathrm{SL}(2,E)$.
	In this case, $\rho$ has finite image,
	since $\pi'$ has finite index in $\pi$.
	This proves the equivalence (2) $\Leftrightarrow$ (3).
\end{proof}

\begin{remark}\label{Zd_qarep_justification_remark}
	Considering $\mathrm{SL}(1,\mathscr{A})$ as an algebraic group over $F$,
	one may actually replace $\mathrm{SL}(2,E)$ with $\mathrm{SL}(1,\mathscr{A})$
	in the condition (1).	
	When $\pi$ is finitely generated, one may remove the requirement of being normal
	in the condition (2), 
	because any finite-index subgroup of a finitely generated group
	contains a finite-index normal subgroup.
\end{remark}

\begin{definition}\label{Zd_qarep_def}
	With $\mathscr{A}/F$ as in Definition \ref{qarep_def},
	we say that a representation $\rho\colon\pi\to\mathrm{SL}(1,\mathscr{A})$
	is \emph{Zariski dense} 
	if $\rho$ satisfies any of the equivalent conditions
	as in Lemma \ref{Zd_qarep_justification}.
\end{definition}

\begin{definition}\label{invar_qa_def}
	With $\mathscr{A}/F$ as in Definition \ref{qarep_def},
	let $\pi$ be a group, and
	$\rho\colon\pi\to\mathrm{SL}(1,\mathscr{A})$ 
	be a representation.
	\begin{enumerate}
	\item When $\rho$ is irreducible,
	the \emph{trace field} of $\rho$ is defined as 
	the subfield of $F$ 
	generated by $\{\chi_\rho(g)\in F\colon g\in\pi\}$ over $\Rational$,
	and the \emph{quaternion algebra} of $\rho$ is defined as
	the quaternion algebra of $\mathscr{A}$ 
	generated by $\{\rho(g)\in\mathscr{A}\colon g\in\pi\}$
	over the trace field.
	\item
	When $\rho$ is Zariski dense,
	the \emph{invariant trace field} of $\rho$ is defined as 
	the subfield of $F$ 
	generated by $\{\chi_\rho(g)\in F\colon g\in\mathrm{Fratt}_2(\pi)\}$ over $\Rational$,
	and the \emph{invariant quaternion algebra} of $\rho$ is defined as
	the subalgebra of $\mathscr{A}$
	generated by $\{\rho(g)\in\mathscr{A}\colon g\in \mathrm{Fratt}_2(\pi)\}$ 
	over the invariant trace field.
	\end{enumerate}
\end{definition}

\begin{remark}\label{invar_qa_remark}
	Denote by $\mathrm{PSL}(1,\mathscr{A})$ the quotient of $\mathrm{SL}(1,\mathscr{A})$
	by the center $\{\pm\mathbf{1}\}$.
	For any \emph{projective special linear representation} $\rho^{\flat}\colon \pi\to\mathrm{PSL}(1,\mathscr{A})$,
	one may define the \emph{adjoint trace field} of $\rho^{\flat}$
	as the subfield of $F$ generated over $\Rational$ by
	the traces of the adjoint representations $\mathrm{Ad}(\rho^{\flat})\colon \pi\to \mathrm{Aut}_F(\mathfrak{sl}(1,\mathscr{A}))$,
	where $\mathfrak{sl}(1,\mathscr{A})$ can be identified as 
	the $3$--dimensional $F$--linear subspace of $\mathscr{A}$ 
	consisting of all the trace zero elements.
	When $\rho^{\flat}$
	comes from a representation $\rho\colon \pi\to \mathrm{SL}(1,\mathscr{A})$
	as the projectivization $\mathbf{P}(\rho)$,
	one obtains the relation $\mathrm{tr}_F(\mathrm{Ad}(\rho^{\flat})(g))=\chi_\rho(g)^2-1$ for all $g\in\pi$.
	Therefore, the adjoint trace field of any liftable $\mathrm{PSL}(1,\mathscr{A})$--representation
	coincides with the trace field of any of its $\mathrm{SL}(1,\mathscr{A})$--lifts
	restricted to the $2$--Frattini subgroup.
	For any algebraically closed field extension $E$ over $F$,
	we can regard 
	$\mathrm{PSL}(1,\mathscr{A}\otimes_FE)
	\cong \mathrm{PSL}(2,E) \cong \mathrm{SO}(3,E)$
	as an algebraic group over $E$.
	Observe that a subgroup of $\mathrm{PSL}(1,\mathscr{A}\otimes_FE)$
	is Zariski closed if and only if its preimage in 
	$\mathrm{SL}(1,\mathscr{A}\otimes_FE)$
	is Zariski closed.
	Therefore, 
	one may define a \emph{Zariski dense} representation
	$\rho^{\flat}\colon \pi\to\mathrm{PSL}(1,\mathscr{A})$
	by the property that 
	$\rho^{\flat}(\pi)$ is Zariski dense in $\mathrm{PSL}(1,\mathscr{A}\otimes_FE)$
	for some algebraically closed field extension $E$ over $F$.	
	In this case,
	the adjoint trace field of $\rho^{\flat}$ coincides
	with the invariant trace field of $\Pi$,
	where $\Pi$ denotes the full preimage of $\rho^{\flat}(\pi)$ 
	in $\mathrm{SL}(1,\mathscr{A})$.
	Therefore, one may define 
	the \emph{invariant quaternion algebra} of 
	a Zariski dense 
	$\mathrm{PSL}(1,\mathscr{A})$--representation $\rho^{\flat}$
	as the invariant quaternion algebra of $\Pi$,
	over the invariant trace field of $\Pi$.
	Compare \cite{Vinberg_field}.
\end{remark}

\begin{lemma}\label{virtual_Frattini_character}
	With $\mathscr{A}/F$ as in Definition \ref{qarep_def},
	let	$\pi$ is a finitely generated group
	and $\rho\colon\pi\to \mathrm{SL}(1,\mathscr{A})$ be 
	a Zariski dense representation.
	Then, the following conditions 
	on a representation $\tilde{\rho}\colon\pi\to \mathrm{SL}(1,\mathscr{A})$
	are all equivalent.
	\begin{enumerate}
	\item
	For some group homomorphism $\varepsilon\colon\pi\to\{\pm1\}$,
	$\tilde{\rho}$ is conjugate to $\varepsilon\cdot\rho$ in $\mathrm{GL}(1,\mathscr{A})$.
	\item
	For some group homomorphism $\varepsilon\colon\pi\to\{\pm1\}$,
	$\chi_{\tilde{\rho}}=\varepsilon\cdot\chi_{\rho}$ holds on $\pi$.	
	\item
	The equality $\chi_{\tilde{\rho}}=\chi_{\rho}$ holds on $\mathrm{Fratt}_2(\pi)$.	
	\item 
	For some (hence any) finite-index subgroup $\pi'$ of $\pi$,
	$\chi_{\tilde{\rho}}=\chi_{\rho}$ holds on $\mathrm{Fratt}_2(\pi')$.
	\end{enumerate}
\end{lemma}

\begin{proof}
	The equivalence (1) $\Leftrightarrow$ (2)
	follows from Lemma \ref{irr_qachi}.
	The implication (2) $\Rightarrow$ (3) follows immediately from the fact
	that $\mathrm{Fratt}_2(\pi)$ is generated by squares of elements in $\pi$.
	The implication (3) $\Rightarrow$ (4) with `any' is trivial.
	It remains to prove (4) with `some' $\Rightarrow$ (2).
	To this end,
	we may assume without loss of generality that $\pi'$ is normal in $\pi$,
	and argue with $\mathscr{A}=\mathrm{Mat}(2,\Complex)$ over $F=\Complex$,
	since $\pi$ is finitely generated.
	By Lemma \ref{Zd_qarep_justification},
	and Remark \ref{Zd_qarep_justification_remark}, 
	$\rho\colon\pi\to\mathrm{SL}(2,\Complex)$
	is irreducible restricted to 
	the finite-index subgroup $\mathrm{Fratt}(\pi')$.
	Since $\tilde{\rho}$ has the same character as 
	that of $\rho$ restricted to $\mathrm{Fratt}_2(\pi')$,
	we may adjust $\tilde{\rho}$ with a conjugation in $\mathrm{SL}(2,\Complex)$,
	and assume henceforth $\tilde{\rho}=\rho$ on $\mathrm{Fratt}_2(\pi')$,
	(see \cite[Proposition 1.5.2]{CS_rep_var}).
	
	Since $\rho$ is irreducible on $\mathrm{Fratt}_2(\pi')$,
	there exist $g_1,g_2,g_3\in \mathrm{Fratt}_2(\pi')$,
	such that the matrices 
	$I,\rho(g_1),\rho(g_2),\rho(g_3)$
	span $\mathrm{Mat}(2,\Complex)$ over $\Complex$.
	We denote $X_i=2\cdot\rho(g_i)-\chi_\rho(g_i)\cdot I$
	in $\mathrm{Mat}(2,\Complex)$, for $i=1,2,3$.
	Then the elements $X_1,X_2,X_3$
	span the complex Lie subalgebra $\mathfrak{sl}(2,\Complex)$,	
	(which consists of all the traceless matrices).
	It is well-known that the adjoint action
	$\mathrm{Ad}_g(X)=gXg^{-1}$
	defines a group homomorphism
	$\mathrm{Ad}\colon\mathrm{SL}(2,\Complex)\to \mathrm{Aut}(\mathfrak{sl}(2,\Complex))$
	with kernel $\{\pm I\}$.
	As we have assumed $\pi'$ normal in $\pi$,
	its characteristic subgroup $\mathrm{Fratt}_2(\pi')$ is also normal in $\pi$.
	We observe 
	$\mathrm{Ad}_{\rho(u)}(X_i)=
	2\cdot\rho(ug_iu^{-1})-\chi_\rho(g_i)\cdot I=
	2\cdot\tilde{\rho}(ug_iu^{-1})-\chi_\rho(g_i)\cdot I=
	2\cdot\tilde{\rho}(u)\tilde{\rho}(g_i)\tilde{\rho}(u^{-1})-\chi_\rho(g_i)\cdot I=
	2\cdot\tilde{\rho}(u)\rho(g_i)\tilde{\rho}(u^{-1})-\chi_\rho(g_i)\cdot I=
	\mathrm{Ad}_{\tilde{\rho}(u)}(X_i)$,
	for $i=1,2,3$, and for all $u\in\pi$.	
	It follows that $\rho(u)=\varepsilon(u)\cdot\tilde{\rho}(u)$
	holds	for some $\varepsilon(u)\in\{\pm 1\}$.
	Clearly $\varepsilon\colon\pi\to \{\pm1\}$
	must be a group homomorphism,
	so we obtain $\chi_{\tilde{\rho}}=\varepsilon\cdot\chi_\rho$ on $\pi$,
	as desired.	
\end{proof}

\begin{definition}\label{Frattini_equivalence_def}
	With $\mathscr{A}/F$ as in Definition \ref{qarep_def},
	let	$\pi$ is a finitely generated group.
	We say that 
	a representation $\tilde{\rho}\colon\pi\to \mathrm{SL}(1,\mathscr{A})$ 
	is \emph{$\mathrm{Hom}(\pi;\{\pm1\})$--equivalent}
	to a Zariski dense representation $\rho\colon\pi\to \mathrm{SL}(1,\mathscr{A})$
	if $\tilde{\rho}$ satisfies any of the equivalent conditions
	in Lemma \ref{virtual_Frattini_character}.
	Note that $\tilde{\rho}$ is necessarily Zariski dense in this case.
\end{definition}

We draw the following immediate consequence from the above discussion,
(compare \cite[Chapter 3, Theorem 3.3.4]{Maclachlan--Reid_book}).

\begin{corollary}\label{virtual_invar_qa}
	With $\mathscr{A}/F$ as in Definition \ref{qarep_def},
	let $\pi$ be a finitely generated group,
	and $\rho\colon\pi\to \mathrm{SL}(1,\mathscr{A})$
	be a Zariski dense representation.
	Then the invariant trace field 
	and the invariant quaternion algebra of $\rho$
	depend only on the $\mathrm{Hom}(\pi,\{\pm1\})$--equivalence class of $\rho$,
	and they remain unchanged 
	under passage	to finite-index subgroups of $\pi$ and restrictions of $\rho$.
\end{corollary}
%
%
%
%
%
%

\begin{definition}\label{bounded_qarep_def}
	With $\mathscr{A}/F$ as in Definition \ref{qarep_def}.
	We say that a representation $\rho\colon \pi\to \mathrm{SL}(1,\mathscr{A})$
	is \emph{bounded} at a place $v$,
	if the image of the induced homomorphism
	$\rho_v\colon\pi\to \mathrm{SL}(1,\mathscr{A}_v)$ is bounded 
	in $\mathrm{Mat}(1,\mathscr{A}_v)=\mathscr{A}_v$.
	Here, $\mathscr{A}_v=\mathscr{A}\otimes_FF_v$
	denotes the complete quaternion algebra over the complete field $F_v$.
\end{definition}

\begin{lemma}\label{bounded_character_property}
	With $\mathscr{A}/F$ as in Definition \ref{qarep_def},
	let $\pi$ be a group and
	 $\rho\colon \pi\to\mathrm{SL}(1,\mathscr{A})$ be an irreducible representation.
	Then, $\rho$ is bounded at a non-Archimedean place $v$ 
	if and only if $\chi_{\rho}\colon \pi\to F$ 
	has image contained in the valuation ring of $F$ at $v$.
\end{lemma}

\begin{proof}
	Denote by $R=\{x\in F\colon |x|_v\leq1\}$ the valuation ring of $F$ at $v$,
	fixing a representative norm $|\cdot|_v$.
	The trace relation
	$\mathrm{Tr}_{\mathscr{A}/F}(\rho(g^2))=\mathrm{Tr}_{\mathscr{A}/F}(\rho(g))^2-2$
	(Proposition \ref{trace_relations}) 
	and the super-triangular inequality for non-Archimedean valuations
	imply $|\chi_\rho(g^2)|_v=|\chi_\rho(g)|_v^2$
	whenever $|\chi_\rho(g)|_v>1$.
	By iteration, we see that $\chi_\rho$ diverges
	along the $2$--exponent powers of $g$ 
	whenever $|\chi_\rho(g)|_v>1$.
	Therefore, 
	if $\rho$ is bounded at $v$, we obtain
	$\chi_\rho(g)\in R$ for all $g\in\pi$.
	
	Conversely, if $\chi_\rho(g)\in R$ holds for all $g\in\pi$,
	we show that $\rho$ is bounded at $v$.
	We make use of the canonical \emph{trace pairing} 
	$\mathscr{A}\times\mathscr{A}\to F$,
	defined as $(A,B)\mapsto \mathrm{Tr}_{\mathscr{A}/F}(AB)$.
	The pairing is symmetric, $F$--bilinear, and nondegenerate.
	(This	can be easily seen with any Hilbert symbol 
	$\mathscr{A}\cong(a,b/F)$:
	The pairing matrix is diagonalized over 
	the standard basis $\mathbf{1},\mathbf{i},\mathbf{j},\mathbf{k}$
	with diagonal entries $2,2a,2b,-2ab$.)
	Since $\rho$ is irreducible, 
	there exist elements $g_0,g_1,g_2,g_3\in\pi$,
	such that 
	$X_0=\rho(g_0),X_1=\rho(g_1),X_2=\rho(g_2),X_3=\rho(g_3)$
	form a basis of $\mathscr{A}$ over $F$.
	Denote by $X_0^*,X_1^*,X_2^*,X_3^*$
	the dual basis with respect to the trace pairing.
	Then 
	$\rho(u)=\chi_\rho(ug_0)\cdot X_0^*+\chi_\rho(ug_1)\cdot X_1^*+\chi_\rho(ug_2)\cdot X_2^*+\chi_\rho(ug_3)\cdot X_3^*$
	holds in $\mathscr{A}$ for all $u\in\pi$.
	Since $\chi_\rho$ has image in $R$,
	this means $\rho_v(u)$ lies in the bounded subset
	$R_vX_0^*+R_vX_1^*+R_vX_2^*+R_vX_3^*\subset \mathscr{A}_v$, 
	for all $u\in \pi$,
	so $\rho$ is bounded at $v$.
\end{proof}

\section{Profinite correspondence of representations}\label{Sec-profin_corresp_rep}
In this section, we prove an essential case of Theorem \ref{main_rep_EZ},
namely, the case with pseudo-Anosov mapping tori (Lemma \ref{profinite_correspondence_rho_pA}).

\begin{lemma}\label{profinite_correspondence_rho_pA}
	Let $\Rational^{\algcl}$ be an algebraic closure of $\Rational$.
	Let $(M_A,M_B,\Psi)$ be a profinite isomorphism setting of pseudo-Anosov mapping tori
	(Convention \ref{profinite_isomorphism_setting_pA}).
	
	Then, for any Zariski dense representation
	$\rho_B\colon \pi_B\to \mathrm{SL}(2,\Rational^{\algcl})$,
	there exists a Zariski dense representation
	$\rho_A\colon \pi_A\to \mathrm{SL}(2,\Rational^{\algcl})$
	with all the following properties:
	\begin{itemize}
	\item The invariant trace fields $k_A$ of $\rho_A$ 
	and $k_B$ of $\rho_B$ are identical in $\Rational^{\algcl}$.
	\item 
	The invariant quaternion algebras $\mathscr{B}_A$ of $\rho_A$
	and $\mathscr{B}_B$ of $\rho_B$ are ramified 
	over identical sets of places in $k=k_A=k_B$.
	Hence, $\mathscr{B}_A$ and $\mathscr{B}_B$ 
	are $\mathrm{GL}(2,\Rational^{\algcl})$--conjugate 
	in $\mathrm{Mat}(2,\Rational^{\algcl})$.
	\item 
	For any prime $\mathfrak{p}$ in $k$,
	the restricted representation 
	$\rho^{\natural}_{A}\colon \mathrm{Fratt}_2(\pi_A)\to\mathrm{SL}(1,\mathscr{B}_A)$
	is bounded at $\mathfrak{p}$
	if and only if 
	$\rho^{\natural}_B\colon \mathrm{Fratt}_2(\pi_B)\to\mathrm{SL}(1,\mathscr{B}_B)$
	is bounded at $\mathfrak{p}$.
	\item
	For any prime $\mathfrak{p}$ in $k$
	where $\rho^{\natural}_A$ and $\rho^{\natural}_B$ are bounded,
	the equality
	$\chi^{\natural}_{A,\mathfrak{p}}=\chi^{\natural}_{B,\mathfrak{p}}\circ\Psi$ 
	holds on $\mathrm{Fratt}_2(\widehat{\pi}_A)=\Psi^{-1}(\mathrm{Fratt}_2(\widehat{\pi}_B))$.
	Here, 
	we denote by 
	$\chi^{\natural}_{B,\mathfrak{p}}\colon\mathrm{Fratt}_2(\widehat{\pi}_B)\to k_{\mathfrak{p}}$
	the character of the continuous representation 
	$\rho^{\natural}_{B,\mathfrak{p}}\colon\mathrm{Fratt}_2(\widehat{\pi}_B)\to 
	\mathrm{SL}(1,\mathscr{B}_{B,\mathfrak{p}})$
	obtained by completion,
	and similarly
	$\chi^{\natural}_{A,\mathfrak{p}}$
	of 
	$\rho^{\natural}_{A,\mathfrak{p}}$.
	\end{itemize}
	
	Subject to the above properties,
	the $\mathrm{Hom}(\pi_A,\{\pm1\})$--equivalence class of $\rho_A$
	depends only on $\Psi$ and the $\mathrm{Hom}(\pi_B,\{\pm1\})$--equivalence class of $\rho_B$.
\end{lemma}

\begin{proof}
	Starting from any Zariski dense representation
	$\rho_B\colon \pi_B\to \mathrm{SL}(2,\Rational^{\algcl})$,
	we obtain its invariant trace field $k=k_B$ in $\Rational^{\algcl}$,
	and its invariant quaternion algebra $\mathscr{B}_B$
	in $\mathrm{Mat}(2,\Rational^{\algcl})$.
	There is an associated representation
	$\rho^{\natural}_B\colon\mathrm{Fratt}_2(\pi_B)\to \mathrm{SL}(1,\mathscr{B}_B)$.
	We also choose a sufficiently large finite extension field $K$ of $k_B$
	in $\Rational^{\algcl}$,
	such that $\rho_B$ factors through a matrix representation over $K$,
	still denoted as $\rho_B\colon\pi_B\to \mathrm{SL}(2,K)$.
	Denote by $O_K$ the ring of integers in $K$.
	The characters $\chi^{\natural}_B$ and $\chi_B$
	are given as composition along the horizontal arrows 
	in the following commutative diagram:
	\begin{equation}\label{chi_sharp_natural}
	\xymatrix{
	\mathrm{Fratt}_2(\pi_B) \ar[r]^-{\rho^{\natural}_B} \ar[d]_{\mathrm{incl}} & 
	\mathrm{SL}(1,\mathscr{B}_B) \ar[r]^-{\mathrm{Tr}} \ar[d]_{\mathrm{incl}} &
	k \ar[d]^{\mathrm{incl}} \\
	\pi_B \ar[r]^-{\rho_B} & 
	\mathrm{SL}(2,K) \ar[r]^-{\mathrm{Tr}} &
	K 	
	}
	\end{equation}
	
	Note that $\rho_B$ is bounded at all but finitely many primes in $K$.
	In fact, if $g_1,\cdots,g_r$ is a finite generating set of $\pi_B$,
	and if a prime $\mathfrak{P}$ in $K$ does not divide
	the denominator of any nonzero matrix entry in $\rho_B(g_1),\cdots,\rho_B(g_r)$,
	then $\rho_B$ is obviously bounded at $\mathfrak{P}$.
	Moreover,
	it is easy to check that $\rho^\natural_B$ is bounded at a prime $\mathfrak{p}$
	in $k$ if and only if $\rho_B$ is bounded at every prime $\mathfrak{P}$ in $K$
	above $\mathfrak{p}$,
	using the trace relation 
	$\chi_B(g)^2=\chi^{\natural}_B(g^2)+2$ for all $g\in\pi_B$,
	(see Lemmas \ref{trace_relations} and \ref{bounded_character_property}).
	
	We construct a continuous character 
	$\chi^{\Psi}_{A,\mathfrak{P}}\colon\widehat{\pi}_A\to K_{\mathfrak{P}}$
	as follows.
	Choose any prime $\mathfrak{P}$ in $K$ where $\rho_B$ is bounded.
	Then there is a unique continuous representation
	$\rho_{B,\mathfrak{P}}\colon
	\widehat{\pi}_B\to \mathrm{SL}(2,K_{\mathfrak{P}})$
	that extends the induced representation 
	$\pi_B\to \mathrm{SL}(2,K_{\mathfrak{P}})$ of $\rho_B$.
	We obtain a continuous representation
	$\rho^{\Psi}_{A,\mathfrak{P}}\colon \widehat{\pi}_A\to \mathrm{SL}(2,K_{\mathfrak{P}}),$
	by pulling back $\rho_{B,\mathfrak{P}}$ with 
	the profinite isomorphism $\Psi$.
	Denote by 
	$\chi^{\Psi}_{A,\mathfrak{P}}$
	the character of $\rho^{\Psi}_{A,\mathfrak{P}}$.
	
	We record the following obvious relation
	$$\rho^{\Psi}_{A,\mathfrak{P}}=\rho_{B,\mathfrak{P}}\circ \Psi.$$
	
	We claim that
	any continuous character $\chi^{\Psi}_{A,\mathfrak{P}}$
	as constructed above, by restriction,
	gives rise to a character
	\begin{equation}\label{chi_natural_A}
	\chi^{\natural}_{A}\colon \mathrm{Fratt}_2(\pi_A)\to k,
	\end{equation}
	which does not depend on the choice of $K$ or $\mathfrak{P}$.
	Moreover, the image of $\chi^{\natural}_{A}$ generates $k$ over $\Rational$.
	This claim is proved as follows.
	
	With respect to $\rho_{B,\mathfrak{P}}$,
	the preimage of the principal $\mathfrak{P}$--congruence subgroup
	$\mathrm{SL}(2,K_{\mathfrak{P}})\cap(I+\mathrm{Mat}(2,\mathfrak{P}O_{K,\mathfrak{P}}))$
	is an open normal subgroup of $\widehat{\pi}_B$.
	We denote the preimage as $\widehat{\pi}_B[\mathfrak{P}]$.
	Denote by $\pi_B[\mathfrak{P}]$
	the finite-index normal subgroup of $\pi_B$
	obtained by intersecting with $\widehat{\pi}_B[\mathfrak{P}]$ in $\widehat{\pi}_B$.
	We obtain a regular finite cover $M_B'$ of $M_B$
	with the fundamental group $\pi'_B=\mathrm{Fratt}_2(\pi_B[\mathfrak{P}])$.
	Using $\Psi$,
	we obtain a profinitely corresponding regular finite cover $M_A'$ of $M_A$
	with the fundamental group $\pi'_A=\mathrm{Fratt}_2(\pi_A[\mathfrak{P}])$.
	Note that $M_A'$ and $M_B'$
	can be identified as pseudo-Anosov mapping tori covering $M_A$ and $M_B$,
	respectively.
	The restricted profinite isomorphism $\Psi'\colon\widehat{\pi}_A'\to\widehat{\pi}_B'$	
	witnesses a bijective correspondence between the periodic trajectories
	$\mathbf{c}_A$ of $M_A'$ and $\mathbf{c}_B$ of $M_B'$,
	(Lemma \ref{profinite_correspondence_pt}).	
	Representing the conjugacy classes $\mathbf{c}_A$ and $\mathbf{c}_B$
	with $g_A\in\pi'_A$ and $g_B\in\pi'_B$, respectively,
	it follows that $\rho^{\Psi}_{A,\mathfrak{P}}(g_A)$ 
	and $\rho_{B,\mathfrak{P}}(g_B)^\mu$ are conjugate in $\mathrm{SL}(2,O_{\mathfrak{P}})$.
	In $K_{\mathfrak{P}}$, we obtain
	\begin{equation}\label{same_chi_pt}
	\chi^{\Psi}_{A,\mathfrak{P}}(\mathbf{c}_A)=\mathrm{tr}\left(\rho^{\Psi}_{A,\mathfrak{P}}(g_A)\right)
	=\mathrm{tr}\left(\rho_{B,\mathfrak{P}}(g_B)^\mu\right)
	=\mathrm{tr}\left(\rho_{B,\mathfrak{P}}(g_B)\right)
	=\chi_{B,\mathfrak{P}}(\mathbf{c}_B),
	\end{equation}
	using $\mu^2=1$, (Lemmas \ref{mu_square} and \ref{A_mu}).
	It follows that the invariant trace field $k_A$ of 
	$\rho^{\Psi}_{A,\mathfrak{P}}$ coincides with $k_B=k$,
	regardless of which $K$ or $\mathfrak{P}$ we have chosen,
	(Lemma \ref{trace_field_pt} and Corollary \ref{virtual_invar_qa}).
	
	The rest of the claim is
	the equality $\chi^{\Psi}_{A,\mathfrak{P}}=\chi^{\Psi}_{A,\mathfrak{Q}}$
	on $\mathrm{Fratt}_2(\pi_A)$,
	assuming $\rho_B$ bounded at 
	a pair of primes $\mathfrak{P}$ and $\mathfrak{Q}$ in
	a sufficiently large common field $K$.
	To this end, we observe that the finite-index normal subgroup 
	$\pi_B[\mathfrak{P}]\cap\pi_B[\mathfrak{Q}]$ of $\pi_B$
	corresponds to $\pi_A[\mathfrak{P}]\cap\pi_A[\mathfrak{Q}]$ of $\pi_A$,
	with respect to $\Psi$.
	Take the regular finite covers $M''_A$ of $M_A$ and $M''_B$ of $M_B$
	with the fundamental groups $\pi''_A=\mathrm{Fratt}_2(\pi_A[\mathfrak{P}]\cap\pi_A[\mathfrak{Q}])$
	and $\pi''_B=\mathrm{Fratt}_2(\pi_B[\mathfrak{P}]\cap\pi_B[\mathfrak{Q}])$, respectively.
	In $k=k_A=k_B$, we obtain 
	\begin{equation}\label{same_chi_pt_PQ}
	\chi^{\Psi}_{A,\mathfrak{P}}(\mathbf{c}_A)=\chi_{B,\mathfrak{P}}(\mathbf{c}_B)=\chi_B(\mathbf{c}_B)
	=\chi_{B,\mathfrak{Q}}(\mathbf{c}_B)=\chi^{\Psi}_{A,\mathfrak{Q}}(\mathbf{c}_A),
	\end{equation}
	for any $\Psi$--corresponding periodic trajectories
	$\mathbf{c}_A$ of $M_A''$ and $\mathbf{c}_B$ of $M_B''$
	(Lemma \ref{profinite_correspondence_pt}),
	the first and the last steps using the same computation as in (\ref{same_chi_pt}).
	Then we obtain $\chi^{\Psi}_{A,\mathfrak{P}}=\chi^{\Psi}_{A,\mathfrak{Q}}$
	on $\pi''_A$ (Lemma \ref{trace_field_pt}), 
	and hence on $\mathrm{Fratt}_2(\pi_A)$ (Lemma \ref{virtual_Frattini_character}).
	This proves the above claim about (\ref{chi_natural_A}).
	
	There exists some Zariski dense representation 
	\begin{equation}\label{rho_A_algcl}
	\rho_A\colon\pi_A\to\mathrm{SL}(2,\Rational^{\algcl}),
	\end{equation}
	and its character $\chi_A$ restricted to $\mathrm{Fratt}_2(\pi_A)$ coincides with $\chi^\natural_A$.
	In fact, 
	for any continuous character
	$\chi^\Psi_{A,\mathfrak{P}}\colon\widehat{\pi}_A\to K_{\mathfrak{P}}$
	in the above construction,
	the values of $\chi^{\Psi}_{A,\mathfrak{P}}$ on $\pi_A$
	generates a multi-quadratic finite extension over $k$,
	since $\pi_A/\mathrm{Fratt}_2(\pi_A)$
	is a finitely generated abelian $2$--torsion group.
	Then,
	possibly after passing to a finite extension of $K$ in $\Rational^{\algcl}$
	containing those values and a prime above $\mathfrak{P}$,
	one may take $\chi_A$ to be the restricted character
	$\chi^{\Psi}_{A,\mathfrak{P}}\colon \pi_A\to K_{\mathfrak{P}}$,
	and take $\rho_A$ to be	
	any $\mathrm{SL}(2,\Rational^{\algcl})$--representation with character $\chi_A$.
	The existence of $\rho_A$ given $\chi_A$ is well-known \cite[Chapter 3, Corollary 3.2.4]{Maclachlan--Reid_book},
	and the entries of $\rho_A$ generally live in a finite extension of the trace field.
	Anyway, the representation $\rho_A$ in (\ref{rho_A_algcl}) 
	is unique	up to $\mathrm{Hom}(\pi_A;\{\pm1\})$--equivalence,
	and is Zariski dense as $\rho_B$ is,
	(Lemma \ref{virtual_Frattini_character}).
	Fix a representation $\rho_A$ as in (\ref{rho_A_algcl}).
	Denote by $\mathscr{B}_A$ the invariant quaternion algebra of $\rho_A$
	over the invariant trace field $k$.
	By restricting $\rho_A$, we obtain a representation 
	\begin{equation}\label{rho_natural_A}
	\rho^{\natural}_A\colon\mathrm{Fratt}_2(\pi_A)\to\mathrm{SL}(1,\mathscr{B}_A)
	\end{equation}
	with character $\chi^\natural_A$ as in (\ref{chi_natural_A}).
	
	We verify the asserted properties in Lemma \ref{profinite_correspondence_rho_pA}
	as follows.	
	Because of the claim with (\ref{chi_natural_A}),
	we see $k=k_A=k_B$.
	Moreover, 
	for any prime $\mathfrak{p}$ in $k$
	and for any prime $\mathfrak{P}$ in $K$ above $\mathfrak{p}$,
	$\rho_A$ is bounded at $\mathfrak{P}$ if and only if $\rho_B$ is bounded at $\mathfrak{P}$,
	by Lemma \ref{bounded_character_property} and 
	the relation 
	$\rho^{\Psi}_{A,\mathfrak{P}}=\rho_{B,\mathfrak{P}}\circ\Psi$ 
	in $\mathrm{SL}(2,K_{\mathfrak{P}})$.	
	Hence,
	$\rho^{\natural}_A$ is bounded at $\mathfrak{p}$ 
	if and only if $\rho^{\natural}_B$ is bounded at $\mathfrak{p}$.
	Recall the notations $\rho^\natural_{A,\mathfrak{p}}$, $\rho^\natural_{B,\mathfrak{p}}$,
	$\chi^\natural_{A,\mathfrak{p}}$, and $\chi^\natural_{B,\mathfrak{p}}$
	in the statement of Lemma \ref{profinite_correspondence_rho_pA}.
	In this case, we see that
	$$\chi^\natural_{A,\mathfrak{p}}=\chi^\natural_{B,\mathfrak{p}}\circ\Psi$$
	holds on $\mathrm{Fratt}_2(\widehat{\pi}_A)=\Psi^{-1}(\mathrm{Fratt}_2(\widehat{\pi}_B))$,
	as an equality in $k_{\mathfrak{p}}$.
	
	Therefore, it remains to verify $\mathscr{B}_A\cong\mathscr{B}_B$ over $k$.
	Note that any isomorphism over $k$ extends to an isomorphism
	$\mathscr{B}_A\otimes_k\Rational^{\algcl}\cong \mathscr{B}_B\otimes_k\Rational^{\algcl}$,
	or equivalently, an automorphism of the matrix algebra $\mathrm{Mat}(2,\Rational^{\algcl})$
	over $\Rational^{\algcl}$, 
	which is necessarily a $\mathrm{GL}(2,\Rational^{\algcl})$--conjugation.
	
	If $\mathscr{B}_B$ is ramified at a non-Archimedean place in $k$,
	represented as a prime $\mathfrak{p}$ in $k$,
	$\rho^{\natural}_B$ is necessarily bounded at $\mathfrak{p}$,
	because in this case,
	$\mathrm{SL}(1,\mathscr{B}_{\mathfrak{p}})$	is the group of units
	in the maximal compact subring 
	of $\mathrm{Mat}(1,\mathscr{B}_\mathfrak{p})=\mathscr{B}_{\mathfrak{p}}$ 
	(see \cite[Chapter 2, Exercises 2.6.1 and 2.6.2]{Maclachlan--Reid_book}).
	By construction,
	$\chi^{\natural}_A=\chi^{\Psi}_{A,\mathfrak{P}}$ holds on $\mathrm{Fratt}_2(\pi_A)$ 
	in $K_{\mathfrak{P}}$ for any $\mathfrak{P}$ in $K$ above $\mathfrak{p}$.
	Possibly after passing to a finite extension of $K$ in $\Rational^{\algcl}$
	and a prime above $\mathfrak{P}$,
	we may assume that $\rho^{\natural}_A$ and 
	the restriction of $\rho^{\Psi}_{A,\mathfrak{P}}=\rho_{B,\mathfrak{P}}\circ\Psi$
	are conjugate to each other as representations
	$\mathrm{Fratt}_2(\pi_A)\to \mathrm{SL}(2,K_{\mathfrak{P}})$.
	The complete quaternion algebra $\mathscr{B}_{A,\mathfrak{p}}$
	can be identified with the closure of 
	the $k$--subalgebra spanned by $\rho^{\natural}_A(\mathrm{Fratt}_2(\pi_A))$,
	which is also the $k_{\mathfrak{p}}$--subalgebra
	topologically spanned by 
	the closed subgroup 
	$\rho^{\Psi}_{A,\mathfrak{P}}(\mathrm{Fratt}_2(\widehat{\pi}_A))$.
	There is a similar description of $\mathscr{B}_{B,\mathfrak{p}}$.
	It follows that $\mathscr{B}_{A,\mathfrak{p}}$ and $\mathscr{B}_{B,\mathfrak{p}}$
	are conjugate to each other in $\mathrm{Mat}(2,K_{\mathfrak{P}})$,
	as quaternion algebras over $k_{\mathfrak{p}}$.
	In particular, $\mathscr{B}_A$ is ramified at $\mathfrak{p}$ 
	as $\mathscr{B}_B$ is.
	By symmetry,
	$\mathscr{B}_A$ is ramified $\mathfrak{p}$
	if and only if $\mathscr{B}_B$ is ramified at $\mathfrak{p}$.
	
	If $\mathscr{B}_B$ is ramified at an Archimedean place in $k$,
	represented as a necessarily real embedding $\sigma\colon k\to \Real$,
	the image of the induced character $\chi_{B,\sigma}=\sigma\circ\chi_B$
	is contained in the real interval $[-2,2]$.
	Take a $\Psi$--corresponding pair of finite-index subgroups
	$\pi'_A$ of $\mathrm{Fratt}_2(\pi_A)$ and $\pi'_B$ in $\mathrm{Fratt}_2(\pi_B)$,
	such that $\chi_A(\mathbf{c}_A)=\chi_B(\mathbf{c}_B)$ holds in $k$
	for any $\Psi$--corresponding pair
	of periodic trajectories $\mathbf{c}_A$ in $M'_A$ and $\mathbf{c}_B$ in $M'_B$.
	For example,
	$\pi'_A=\mathrm{Fratt}_2(\pi_A[\mathfrak{P}])$
	and $\pi'_B=\mathrm{Fratt}_2(\pi_B[\mathfrak{P}])$,
	for some prime $\mathfrak{P}$ in $K$,
	with the desired equality as provided in (\ref{same_chi_pt}).
	Note that
	the invariant trace fields $\mathscr{B}_A$ and $\mathscr{B}_B$
	can be identified with the trace fields of the restricted representations
	$\rho'_A\colon \pi'_A\to \mathrm{SL}(1,\mathscr{B}_A)$ of $\rho^{\natural}_A$
	and $\rho'_B\colon \pi'_B\to \mathrm{SL}(1,\mathscr{B}_B)$
	of $\rho^{\natural}_B$, respectively
	(Lemma \ref{virtual_invar_qa}).
	Applying Lemma \ref{quaternion_algebra_pt} to $\rho'_{A,\sigma}=\sigma\circ\rho'_A$,
	we see that $\mathscr{B}_A$ is ramified at $\sigma$
	as $\mathscr{B}_B$ is.
	By symmetry,
	$\mathscr{B}_A$ is ramified at $\sigma$
	if and only if $\mathscr{B}_B$ is ramified at $\sigma$.
	
	We conclude that $\mathscr{B}_A$ and $\mathscr{B}_B$
	have identical sets of ramification places in $k$,
	so they are isomorphic to each other over $k$,
	by the classification of 
	quaternion algebras over number fields (Example \ref{qa_example}).	
	This completes our verification.
\end{proof}

\section{On profinite invariance of volume of representations}\label{Sec-profin_invar_vol}
In this section, we prove an essential case of Theorem \ref{main_vol_EZ},
namely, the case with pseudo-Anosov mapping tori (Lemma \ref{profinite_invariance_vol_pA}).

\begin{lemma}\label{profinite_invariance_vol_pA}
	Assume Hypothesis \ref{Borel_hypothesis_p}.
	Let $\Rational^{\algcl}$ be an algebraic closure of $\Rational$.
	Let $(M_A,M_B,\Psi)$ be a profinite isomorphism setting of pseudo-Anosov mapping tori
	(Convention \ref{profinite_isomorphism_setting_pA}).
	
	Suppose that $\rho_B\colon\pi_B\to \mathrm{SL}(2,\Rational^\algcl)$ 
	is a Zariski dense representation.
	Let $\rho_A\colon \pi_A\to \mathrm{SL}(2,\Rational^\algcl)$ be
	a corresponding representation of $\rho_B$ with respect to $\Psi$,
	with declared properties as in Lemma \ref{profinite_correspondence_rho_pA}.
	Then, for any field embedding $\sigma\colon\Rational^{\algcl}\to \Complex$,
	the following equality holds in $\Real$ up to sign 
	(denoted with a dotted symbol):
	$$\mathrm{Vol}_{\mathrm{SL}(2,\Complex)}\left(M_A,\rho_{A,\sigma}\right)
	\doteq
	\mathrm{Vol}_{\mathrm{SL}(2,\Complex)}\left(M_B,\rho_{B,\sigma}\right).$$
\end{lemma}

\begin{proof}
	Given any $\Psi$--corresponding Zariski dense representations $\rho_A$ and $\rho_B$
	in $\mathrm{SL}(2,\Rational^{\algcl})$, 
	there exists some subfield $K$ in $\Rational^{\algcl}$ of finite degree over $\Rational$,
	such that $\mathrm{SL}(2,K)$ contains the images of $\rho_A$ and $\rho_B$.
	Fixing some such $K$, 
	we rewrite $\rho_A\colon\pi_A\to \mathrm{SL}(2,K)$ and $\rho_B\colon\pi_B\to \mathrm{SL}(2,K)$
	for simplicity.
	Restricting to the $2$--Frattini subgroups,
	these representations factor through the representations in the invariant quaternion algebras,
	namely,	$\rho^{\natural}_A$ and $\rho^{\natural}_B$ 
	as in Lemma \ref{profinite_correspondence_rho_pA}.
	
	Fix some prime $p$ in $\Rational$ as follows.
	For any prime $\mathfrak{p}$ in the invariant trace field $k=k_A=k_B$
	above $p$,
	we require that $\rho^\natural_A$ and $\rho^\natural_B$
	are bounded at $\mathfrak{p}$.
	Moreover, we require that Hypothesis \ref{Borel_hypothesis_p} holds for $K$ and $p$.
	Note that the boundedness condition holds for all but finitely many 
	rational primes	$p$,
	with respect to $\rho^\natural_B$, and hence 
	with respect to $\rho^\natural_A$,
	(Lemma \ref{profinite_correspondence_rho_pA}).
	Therefore,
	the Hypothesis \ref{Borel_hypothesis_p} implies the existence of 
	(infinitely many) $p$ as required.
	
	For any prime $\mathfrak{p}$ in $k$ above $p$,	
	and for any prime $\mathfrak{P}$ in $K$ above $\mathfrak{p}$,
	we obtain a commutative diagram of groups as follows.	
	\begin{equation}\label{vol_diagram_group_B}
	\xymatrix{
	\mathrm{Fratt}_2(\pi_A) \ar[rr]^-{\rho^{\natural}_A} \ar[rd]^{\mathrm{incl}} & & 
	\mathrm{SL}(1,\mathscr{B}_A) \ar[rd]^{\mathrm{incl}} \ar[r]^{\mathrm{incl}} & \mathrm{SL}(2,K) \ar[rd]^{\mathrm{incl}}
	\\
	& \mathrm{Fratt}_2(\widehat{\pi}_A) \ar[rr]^-{\rho^{\natural}_{A,\mathfrak{p}}} \ar[d]_{\Psi|}^{\cong}
	& & \mathrm{SL}(1,\mathscr{B}_{A,\mathfrak{p}}) 
	\ar[r]^{\mathrm{incl}} & \mathrm{SL}(2,K_{\mathfrak{P}}) \ar[d]^{\tau_{\mathfrak{P}}}_{\cong}
	\\
	& \mathrm{Fratt}_2(\widehat{\pi}_B) \ar[rr]^-{\rho^{\natural}_{B,\mathfrak{p}}} 
	& & \mathrm{SL}(1,\mathscr{B}_{B,\mathfrak{p}}) 
	\ar[r]^{\mathrm{incl}} & \mathrm{SL}(2,K_{\mathfrak{P}}) 
	\\
	\mathrm{Fratt}_2(\pi_B) \ar[rr]^-{\rho^{\natural}_B} \ar[ru]^{\mathrm{incl}} & & \mathrm{SL}(1,\mathscr{B}_B)	\ar[ru]^{\mathrm{incl}}
	\ar[r]^{\mathrm{incl}} & \mathrm{SL}(2,K) \ar[ru]^{\mathrm{incl}}
	}
	\end{equation}
	The isomorphism $\tau_{\mathfrak{P}}$ is 
	the restriction of some inner automorphism of $\mathrm{GL}(2,K_{\mathfrak{P}})$.
	It exists because $\chi^\natural_{A,\mathfrak{p}}=\chi^{\natural}_{B,\mathfrak{p}}\circ\Psi$
	holds on $\mathrm{Fratt}_2(\pi_A)$,
	valued in $k$ and hence in $K_{\mathfrak{P}}$ (Lemma \ref{irr_qachi}).
	
	The intermediate terms in the diagram (\ref{vol_diagram_group_B})
	involving invariant quaternion algebras are no longer needed in the rest of the proof.
	On the other hand, 
	the noncompact group $\mathrm{SL}(2,K_{\mathfrak{P}})$ is a bit too large
	for applying available locally analytic or continuous group cohomology theory.
	We need some manipulation with the diagram (\ref{vol_diagram_group_B}).
	To this end, we denote by $O=O_K$ the ring of integers in $K$.
	We adopt the notations
	$$O_{p}=O\otimes_\Integral\Integral_p\cong\prod_{\mathfrak{P}}O_{\mathfrak{P}},$$
	and 
	$$O_{\{p\}}=K\cap O_p,$$ 
	as viewed in
	$$K_p=K\otimes_\Rational\Rational_p\cong \prod_{\mathfrak{P}}K_{\mathfrak{P}},$$
	where $\mathfrak{P}$ ranges over all the primes in $K$ above $p$.
	
	We claim that 
	there are representations $\tilde{\rho}_A\colon \pi_A\to \mathrm{SL}(2,O_{\{p\}})$,
	and $\tilde{\rho}_B\colon \pi_B\to \mathrm{SL}(2,O_{\{p\}})$,
	and isomorphisms
	$\tilde{\tau}_{\mathfrak{P}}
	\colon\mathrm{SL}(2,O_{\mathfrak{P}})\to \mathrm{SL}(2,O_{\mathfrak{P}})$
	for all ${\mathfrak{P}}$ above $p$,	
	as follows.
	The representation $\tilde{\rho}_{A/B}$ is conjugate to $\rho_{A/B}$ in $\mathrm{GL}(2,K)$,
	(where the notation $A/B$ means $A$ or $B$ respectively);
	each isomorphism $\tilde{\tau}_{\mathfrak{P}}$ is the restriction of an inner automorphism
	of $\mathrm{GL}(2,O_{\mathfrak{P}})=\mathrm{Mat}(2,O_{\mathfrak{P}})\cap\mathrm{GL}(2,K_{\mathfrak{P}})$;
	and moreover, $\tilde{\rho}_{A}$, $\tilde{\rho}_{B}$ and each $\tilde{\tau}_{\mathfrak{P}}$
	fit into the following commutative diagram of groups.
	\begin{equation}\label{vol_diagram_group_O}
	\xymatrix{
	\mathrm{Fratt}_2(\pi_A) \ar[rr]^-{\tilde{\rho}_A|} \ar[d]_{\mathrm{incl}} & & 
	 \mathrm{SL}\left(2,O_{\{p\}}\right) \ar[d]^{\mathrm{incl}}
	\\
	\mathrm{Fratt}_2(\widehat{\pi}_A) \ar[rr]^-{\tilde{\rho}_{A,{\mathfrak{P}}}|} \ar[d]_{\Psi|}^{\cong}
	& & \mathrm{SL}(2,O_{\mathfrak{P}}) \ar[d]^{\tilde{\tau}_{\mathfrak{P}}}_{\cong}
	\\
	\mathrm{Fratt}_2(\widehat{\pi}_B) \ar[rr]^-{\tilde{\rho}_{B,{\mathfrak{P}}}|} 
	& &  \mathrm{SL}(2,O_{\mathfrak{P}}) 
	\\
	\mathrm{Fratt}_2(\pi_B) \ar[rr]^-{\tilde{\rho}_B|} \ar[u]^{\mathrm{incl}} & & 
	\mathrm{SL}\left(2,O_{\{p\}}\right) \ar[u]_{\mathrm{incl}}
	}
	\end{equation}
	
	We prove the above claim as follows.
	For every prime $\mathfrak{P}$ in $K$ above $p$,
	since $K_{\mathfrak{P}}$ is local and complete,
	every maximal order in the quaternion algebra $\mathrm{Mat}(2,K_{\mathfrak{P}})$ over $O_{\mathfrak{P}}$
	is conjugate to $\mathrm{Mat}(2,O_{\mathfrak{P}})$
	\cite[Chapter 6, Theorem 6.5.3]{Maclachlan--Reid_book}.
	In particular, there exists some $S_{\mathfrak{P}}\in\mathrm{GL}(2,K_{\mathfrak{P}})$,
	such that $S_{\mathfrak{P}}\rho_B(g)S^{-1}_{\mathfrak{P}}\in \mathrm{SL}(2,O_{\mathfrak{P}})$ 
	holds for every	$g\in\pi_B$.
	Indeed, one may require that 
	$S_{\mathfrak{P}}$ conjugates 
	the order $O_{\mathfrak{P}}[\{\rho_B(g)\colon g\in\pi_B\}]$
	into $\mathrm{Mat}(2,O_{\mathfrak{P}})$. 
	Recall that the group $\mathrm{GL}(2,K_{\mathfrak{P}})$
	is topologized as a subspace of 
	$\mathrm{Mat}(2,K_{\mathfrak{P}})\times \mathrm{Mat}(2,K_{\mathfrak{P}})$,
	with respect to the embedding $A\mapsto (A,A^{-1})$.
	The direct product $\mathrm{GL}(2,K_p)=\prod_{{\mathfrak{P}}|p} \mathrm{GL}(2,K_{\mathfrak{P}})$
	contains an open subgroup
	$\mathrm{GL}(2,O_p)$,
	and a dense subgroup $\mathrm{GL}(2,K)$.
	Therefore,
	the direct product $S_p=\prod_{{\mathfrak{P}}|p } S_{\mathfrak{P}}$ 
	in $\mathrm{GL}(2,K_p)$
	can be factorized as $S'_pS''$,
	for some $S'_p\in\mathrm{GL}(2,O_p)$
	and some $S''\in \mathrm{GL}(2,K)$.
	We obtain $\tilde{\rho}_B$ 
	as the conjugate of $\rho_B$ by $S''$.
	Then the image of $\tilde{\rho}_B$ lies in $\mathrm{SL}(2,O_p)$.
	Let $T_p\in\mathrm{GL}(2,K_p)$
	be a matrix that defines
	the direct-product automorphism $\tau_p=\prod_{{\mathfrak{P}} | p} \tau_{\mathfrak{P}}$
	on $\mathrm{SL}(2,O_p)\cong \prod_{{\mathfrak{P}} | p} \mathrm{SL}(2,O_{\mathfrak{P}})$,
	(namely, $\tau_p(\rho_A(g))=T_p\rho_A(g)T_p^{-1}$ for every $g\in\pi_A$).
	We factorize $T_p$ into $T'_pT''$,
	with some $T'_p\in\mathrm{GL}(2,O_p)$
	and some $T''\in \mathrm{GL}(2,K)$.
	We obtain $\tilde{\rho}_A$
	as the conjugate of $\rho_A$ by $T''$,
	and obtain $\tilde{\tau}_p$
	as the conjugation by $(S'_p)^{-1}T'_p\in\mathrm{GL}(2,O_p)$.
	By construciton,
	each direct factor $\tilde{\tau}_{\mathfrak{P}}$ of $\tilde{\tau}_p$
	is the restriction of an inner automorphism in $\mathrm{GL}(2,O_{\mathfrak{P}})$,
	and the image of $\tilde{\rho}_A$ lies in
	$\mathrm{GL}(2,O_p)\cap\mathrm{GL}(2,K)=\mathrm{GL}(2,O_{\{p\}})$,
	as desired. This proves the claim.
	
	We can replace the target groups in the diagram (\ref{vol_diagram_group_O})
	without affecting commutativity,
	using the standard subgroup embeddings
	$\mathrm{SL}(2,O_{\{p\}})\to \mathrm{GL}(N,O_{\{p\}})$
	and $\mathrm{SL}(2,O_{\mathfrak{P}})\to \mathrm{GL}(N,O_{\mathfrak{P}})$
	for sufficiently large $N$.
	Below we actually fix $N=7$ for explicity,
	(according to Remarks \ref{Borel_theorem_infinity_remark} and \ref{Borel_hypothesis_p_remark}).
	We decorate with a subscript `$\stable$' (as abbreviate for \emph{stable})
	to indicate a representation with the enlarged target group.
	Pass to suitable versions of group cohomology from the diagram (\ref{vol_diagram_group_O}).
	Make sure that the regularity levels (locally $K_{\mathfrak{P}}$--analytic, continuous, or abstract)
	do not get raised along the arrows,
	and the coefficient modules get naturally embedded.
	Take the direct product over all $\mathfrak{P}$ above $p$.
	Then, we obtain	a commutative diagram of $K_p$--modules:
	\begin{equation}\label{vol_diagram_cohomology_O}
	\xymatrix{
	H^3(\mathrm{Fratt}_2(\pi_A);K_p) 
	& &
	H^3\left(\mathrm{GL}\left(7,O_{\{p\}}\right);K_p\right) 
	\ar[ll]_-{ \tilde{\rho}_{A,\stable}^* } 
	\\
	H^3_{\mathtt{cont}}(\mathrm{Fratt}_2(\widehat{\pi}_A);K_p)
	\ar[u]^{ \mathrm{incl}^* } 
	& &
	\bigoplus_{\mathfrak{P}}
	H^3_{\mathtt{la}}\left(\mathrm{GL}(7,O_{\mathfrak{P}});K_{\mathfrak{P}}\right) 
	\ar[u]_{ \sum_{\mathfrak{P}} \mathrm{incl}^* }
	\ar[ll]_-{ \sum_{\mathfrak{P}} \tilde{\rho}_{A,\stable,{\mathfrak{P}}}^* } 
	\\
	H^3_{\mathtt{cont}}(\mathrm{Fratt}_2(\widehat{\pi}_B);K_p)
	\ar[u]^{ \Psi^* }_{ \cong }
	\ar[d]_{ \mathrm{incl}^* }
	& &
	\bigoplus_{\mathfrak{P}}
	H^3_{\mathtt{la}}\left(\mathrm{GL}(7,O_{\mathfrak{P}});K_{\mathfrak{P}}\right) 
	\ar[ll]_-{ \sum_{\mathfrak{P}} \tilde{\rho}_{B,\stable,{\mathfrak{P}}}^* } 
	\ar[u]_{ \bigoplus_{\mathfrak{P}} \tilde{\tau}_{\mathfrak{P}}^* }^{ \cong }
	\ar[d]^{ \sum_{\mathfrak{P}} \mathrm{incl}^* }
	\\
	H^3(\mathrm{Fratt}_2(\pi_B);K_p) 
	& &
	H^3\left(\mathrm{GL}\left(7,O_{\{p\}}\right);K_p\right)
	\ar[ll]_-{ \tilde{\rho}_{B,\stable}^* }  
	}
	\end{equation}
	
	To be precise, the continuous group cohomology
	$H^3_{\mathtt{cont}}(\mathrm{Fratt}_2(\widehat{\pi}_{A/B});K_p)$
	in the diagram (\ref{vol_diagram_cohomology_O})
	is legal because 
	$\mathrm{Fratt}_2(\widehat{\pi}_{A/B})$
	is profinite of type $p$--$\mathrm{FP}_\infty$ \cite[Section 3.8]{Symonds--Weigel_cohomology}.
	This property follows from the fact that 
	finitely generated $3$--manifold groups are all cohomologically good 
	and of type $\mathrm{FP}_\infty$ (see \cite[Section 2.1]{Liu_profinite_almost_rigidity}).
	
	More explicitly, for any cohomologically good $\mathrm{FP}_\infty$ group $\pi$,
	one may take a projective $[\Integral \pi]$--resolution
	$P_\bullet\to\Integral$ by finitely generated free left $[\Integral \pi]$--modules
	(for example, as given by a cell decomposition lifted to the universal cover
	in the aspherical compact $3$--manifold group case).
	Then 
	$\llbracket \widehat{\Integral} \widehat{\pi} \rrbracket\otimes_{[\Integral \pi]}P_\bullet\to \widehat{\Integral}$
	is a projective $\llbracket \widehat{\Integral}\widehat{\pi}\rrbracket$--resolution
	of (the trivial profinite $\llbracket \widehat{\Integral}\widehat{\pi}\rrbracket$--module) $\widehat\Integral$,
	which we rewrite as $\widehat{P}_\bullet$ for simplicity.
	By definition and cohomological goodness,
	the continuous group homology 
	(also known as the profinite group homology in \cite{Ribes--Zalesskii_book})
	$H_*^{\mathtt{cont}}(\widehat{\pi};\Integral_p)$
	and the continuous group cohomology
	$H^*_{\mathtt{cont}}(\widehat{\pi};\Integral_p)$
	are naturally isomorphic to the homologies of
	$\widehat{P}_\bullet\otimes_{ \llbracket \hat{\Integral}\hat{\pi}\rrbracket }\Integral_p$
	and
	$\mathrm{Hom}_{ \llbracket \hat{\Integral}\hat{\pi}\rrbracket }(\widehat{P}_\bullet,\Integral_p)$,
	respectively,
	\cite[Section 3.7, (3.7.10)]{Symonds--Weigel_cohomology}
	(see also \cite[Proposition 3.1]{JZ}).
	The (rational) continuous group cohomology
	$H^*_{\mathtt{cont}}(\widehat{\pi};K_p)$ 
	is naturally isomorphic to
	$H^*_{\mathtt{cont}}(\widehat{\pi};\Integral_p)\otimes_{ \Integral_p } K_p$
	\cite[Section 3.8, Theorem 3.8.2]{Symonds--Weigel_cohomology}.	
	
	Hence,
	we observe natural isomorphisms
	$H^3_{\mathtt{cont}}(\mathrm{Fratt}_2(\widehat{\pi}_{A/B});K_p)
	\cong
	H^3(\mathrm{Fratt}_2(\pi_{A/B});K_p)$,
	and $H_3(\mathrm{Fratt}_2(\pi_{A/B}); \Integral_p)
	\cong
	H_3^{\mathtt{cont}}( \mathrm{Fratt}_2(\widehat{\pi}_{A/B}); \Integral_p)$,
	as induced by
	the inclusion $\mathrm{Fratt}_2(\pi_{A/B})\to \mathrm{Fratt}_2(\widehat{\pi}_{A/B})$.
	(The natural homomorphism 
	$H_3(\mathrm{Fratt}_2(\pi_{A/B}); \Integral) \otimes_\Integral \Integral_p
	\to H_3(\mathrm{Fratt}_2(\pi_{A/B});\Integral_p)$
	is actually also isomorphic,
	because $H_2(\mathrm{Fratt}_2(\pi_{A/B});\Integral)$ has no torsion.)
	Accordingly, the natural pairing
	\begin{equation}\label{pairing_continuous_K_p}
	H^3_{\mathtt{cont}}(\mathrm{Fratt}_2(\widehat{\pi}_{A/B});K_p) \times 
	H_3^{\mathtt{cont}}(\mathrm{Fratt}_2(\widehat{\pi}_{A/B});\Integral_p)\longrightarrow K_p,
	\end{equation}
	is well-defined, and is compatible with the natural pairing
	\begin{equation}\label{pairing_abstract_K_p}
	H^3(\mathrm{Fratt}_2(\pi_{A/B});K_p) \times 
	H_3(\mathrm{Fratt}_2(\pi_{A/B});\Integral)\longrightarrow K_p.
	\end{equation}
	
	
	The locally analytic cohomology 
	$\bigoplus_{\mathfrak{P}}
	H^3_{\mathtt{la}}(\mathrm{GL}(7,O_{\mathfrak{P}});K_{\mathfrak{P}})$
	in the diagram (\ref{vol_diagram_cohomology_O})
	is where our $p$--adic Borel regulator element $b_p$ lives (see (\ref{Borel_p_k})).
	Because the isomorphisms 
	$\tilde{\tau}_{\mathfrak{P}}$ in the diagram (\ref{vol_diagram_group_O})
	all come from $\mathrm{GL}(2,O_{\mathfrak{P}})$ inner automorphisms,
	we see that the isomorphism 
	$\bigoplus_{\mathfrak{P}}\tilde{\tau}_{\mathfrak{P}}^*$
	in the diagram (\ref{vol_diagram_cohomology_O}) is the identity.
	
	Denote by $M_{A/B}^\natural\to M_{A/B}$
	the regular finite cover
	corresponding to the $2$--Frattini subgroup of $\pi_{A/B}$.
	Fixing orientations of $M_A$ and $M_B$,
	we denote by $[M_{A/B}^{\natural}] \in H_3(\mathrm{Fratt}_2(\pi_{A/B});\Integral)$
	the fundamental class,
	and by $[M_{A/B}^{\natural}]_p \in H^{\mathtt{cont}}_3(\mathrm{Fratt}_2(\widehat{\pi}_{A/B});\Integral_p)$
	its image under the natural inclusion.
	By Lemmas \ref{Mu_cube} and \ref{p_regularity},
	we obtain the relation 
	$\Psi_*[M_A^{\natural}]_p=\pm [M_B^{\natural}]_p$.
	Chase the $p$--adic Borel regulator element $b_p$
	(see (\ref{Borel_p_k}))
	on the diagram (\ref{vol_diagram_cohomology_O}).
	Note that $\bigoplus_{\mathfrak{P}}\tilde{\tau}_{\mathfrak{P}}^*$ is the identity.
	Then we obtain
	$$
	\left\langle \sum_\mathfrak{P}\tilde{\rho}^*_{A,\stable,\mathfrak{P}}(b_p),\,[M_A^{\natural}]_p \right\rangle
	=
	\pm \left\langle \sum_\mathfrak{P}\tilde{\rho}^*_{B,\stable,\mathfrak{P}}(b_p),\,[M_B^{\natural}]_p \right\rangle,
	$$
	with respect to the natural pairing 
	(\ref{pairing_continuous_K_p}),
	and hence
	$$
	\left\langle \tilde{\rho}^*_{A,\stable}(b_p),\,[M_A^{\natural}]\right\rangle 
	=
	\pm \left\langle \tilde{\rho}^*_{B,\stable}(b_p),\,[M_B^{\natural}]\right\rangle,$$
	with respect to the natural pairing (\ref{pairing_abstract_K_p}).
	Note that the last equality holds in $K_p$,
	consistent with the coefficient ring of the cohomology classes in the pairing,
	while the fundamental classes are always 
	homology classes with $\Integral$ coefficient as usual.

	Remember that we use the same notation $b_p$ for both 
	the $p$-adic Borel regulator element (\ref{Borel_p_k}) 
	and its interpretation as a map (\ref{Borel_p_k_map}).
	Rewriting with the $p$--adic Borel regulator map (\ref{Borel_p_k_map}),
	we obtain the relation
	$$b_p\left(\tilde{\rho}_{A*}[M_A^{\natural}]\right)=\pm b_p\left(\tilde{\rho}_{B*}[M_B^{\natural}]\right)$$
	in $K_p$,
	treating $\tilde{\rho}_{A/B*}[M_{A/B}^{\natural}]$
	as in $H_3(\mathrm{SL}(7,O_{\{p\}});\Integral)$.
	By our assumption on $p$ at the beginning of the proof (Hypothesis \ref{Borel_hypothesis_p}),
	we obtain the relation 
	\begin{equation}\label{equal_rho_fundamental_class}
	\tilde{\rho}_{A*}[M_A^\natural]=\pm \tilde{\rho}_{B*}[M_B^\natural]
	\end{equation}
	in 
	$H_3(\mathrm{SL}(7,O_{\{p\}});\Integral)\otimes_\Integral\Rational_p
	\cong
	H_3(\mathrm{SL}(7,K);\Integral)\otimes_\Integral\Rational_p$,
	(see Remarks \ref{Borel_theorem_infinity_remark} and \ref{Borel_hypothesis_p_remark}
	for the last natural isomorphism).
	
	The relation (\ref{equal_rho_fundamental_class}) is the key to our whole argument.
	It also explains why Hypothesis \ref{Borel_hypothesis_p} matters to us.
	Applying the classical Borel regulator map associated to $K$,
	we obtain the relation
	$$
	\left\langle \tilde{\rho}_{A,\stable}^*(b_\infty),\,[M_A^{\natural}] \right\rangle
	=
	\pm \left\langle \tilde{\rho}_{B,\stable}^*(b_\infty),\,[M_B^{\natural}] \right\rangle,
	$$
	with respect to the natural pairing
	$H^3(\mathrm{Fratt}_2(\pi_{A/B});K_\infty)\times
	H_3(\mathrm{Fratt}_2(\pi_{A/B});\Integral)\to K_\infty$,
	(see (\ref{Borel_infinity_k}) and  (\ref{Borel_infinity_k_map})).
		
	For any embedding $\sigma\colon K\to\Complex$ representing a complex place $v$ of $K$,
	we infer 
	\begin{equation}\label{vol_removing_tilde}
	\mathrm{Vol}(M_A^{\natural},\rho_{A,\sigma})=
	\mathrm{Vol}(M_A^{\natural},\tilde{\rho}_{A,\sigma})=
	\pm\mathrm{Vol}(M_B^{\natural},\tilde{\rho}_{B,\sigma})=
	\pm\mathrm{Vol}(M_B^{\natural},\rho_{B,\sigma}).
	\end{equation}
	In (\ref{vol_removing_tilde}), we have dropped 
	the subscript in $\mathrm{Vol}_{\mathrm{SL}(2,\Complex)}$
	for simplicity;
	the middle step follows from the above relation in $K_\infty$
	and the volume formula (\ref{volume_Borel_complex}),
	using the ring homomorphism $K_\infty\to K_v\to \Complex$
	(namely, the factor projection followed by the completion of $\sigma$);	
	the first and the last steps are because
	volume of representations is $\mathrm{SL}(2,\Complex)$--conjugation invariant
	\cite[Proposition 2.3]{DLSW_rep_vol},
	and hence $\mathrm{GL}(2,\Complex)$--conjugation invariant.
	Note that (\ref{vol_removing_tilde})
	also holds when $\sigma$ factors through a real embedding,
	since the volume in that case is always zero.
	
	Therefore, for any field embedding $\sigma\colon K\to \Complex$
	(coming from an embedding of $\Rational^\algcl$ as in Lemma \ref{profinite_invariance_vol_pA}),
	we obtain
	\begin{equation}\label{vol_AB}
	\mathrm{Vol}(M_A,\rho_{A,\sigma})=
	\frac{\mathrm{Vol}\left(M_A^\natural,\rho_{A,\sigma}\right)}{[M_A^\natural:M_A]}=
	\frac{\pm\mathrm{Vol}\left(M_B^\natural,\rho_{B,\sigma}\right)}{[M_B^\natural:M_B]}=
	\pm\mathrm{Vol}(M_B,\rho_{B,\sigma}),
	\end{equation}
	as asserted.
	In (\ref{vol_AB}),
	the middle step in (\ref{vol_AB}) follows from (\ref{vol_removing_tilde})
	and the covering degree equality $[M_A^\natural:M_A]=[M_B^\natural:M_B]$;
	the first and the last steps in (\ref{vol_AB}) 
	follow from the proportionality of volume of representations
	for pullbacks to finite covers (see \cite[Proposition 3.1]{DLSW_rep_vol}).
 	
	This completes the proof of Lemma \ref{profinite_invariance_vol_pA}.
	As a final remark,
	the sign ambiguity in (\ref{vol_AB}) 
	arises not only because we have fixed
	orientations of the $3$--manifolds arbitrarily, 
	but also because 
	the sign might vary as the prime $p$ varies.
\end{proof}

\section{Main theorems}\label{Sec-main_theorems}
In this section, we prove the main results as mentioned in the introduction.
Theorem \ref{main_profinite_correspondence_PSL} is the precise form of Theorem \ref{main_rep_EZ},
and
Theorem \ref{main_profinite_invariance_vol} is the precise form of Theorem \ref{main_vol_EZ}.

Let $\Rational^\algcl$ be an algebraic closure of $\Rational$.
For any group $\pi$,
a representation $\rho\colon\pi\to\mathrm{PSL}(2,\Rational^\algcl)$
is said to be \emph{Zariski dense} if its image is Zariski dense
in the $\Rational^\algcl$--algebraic group 
$\mathrm{PSL}(2,\Rational^\algcl)\cong\mathrm{SO}(3,\Rational^\algcl)$.
In this case,
the \emph{invariant trace field} and the \emph{invariant quaternion algebra}
 of $\rho$ refer to 
the invariant trace field in $\Rational^\algcl$
and the invariant quaternion algebra in $\mathrm{Mat}(2,\Rational^\algcl)$ 
of the preimage of $\rho(\pi)$ in $\mathrm{SL}(2,\Rational^\algcl)$,
which is a central extension of $\rho(\pi)$ by $\{\pm I\}$. 
(See Remark \ref{invar_qa_remark}.)

\begin{theorem}\label{main_profinite_correspondence_PSL}
	Let $\pi_A$ and $\pi_B$ 
	be fundamental groups of orientable closed hyperbolic $3$--manifolds.
	Suppose that
	$\Psi\colon \widehat{\pi}_A\to\widehat{\pi}_B$
	is a group isomorphism between the profinite completions.
	Let $\Rational^\algcl$ be an algebraic closure of $\Rational$.
	Denote by $\mathcal{Y}_A$ 
	the set of conjugacy classes of 
	the Zariski dense 
	$\mathrm{PSL}(2,\Rational^{\algcl})$--representations	
	of $\pi_A$,
	and similarly $\mathcal{Y}_B$ with $\pi_B$.
	
	Then, there is a bijective map
	$\eta_\Psi\colon \mathcal{Y}_B\to \mathcal{Y}_A$,
	depending only on $\pi_A$, $\pi_B$, and $\Psi$.
	Moreover, 
	if
	$\rho_A\colon\pi_A\to \mathrm{PSL}(2,\Rational^{\algcl})$
	and
	$\rho_B\colon\pi_B\to \mathrm{PSL}(2,\Rational^{\algcl})$
	are Zariski dense representations 
	of $\eta_\Psi$--corresponding conjugacy classes,
	then they have identical invariant trace fields 
	$k(\rho_A)=k(\rho_B)$ in $\Rational^{\algcl}$,
	and have conjugate invariant quaternion algebras
	$\mathscr{B}(\rho_A)$ and $\mathscr{B}(\rho_B)$
	in $\mathrm{Mat}(2,\Rational^\algcl)$.
\end{theorem}

\begin{proof}
	We establish the asserted map $\eta_\Psi$ as follows.	
	For any Zariski dense representation $\rho_B\colon\pi_B\to \mathrm{PSL}(2,\Rational^\algcl)$,
	the induced representation $\rho_{B,\sigma}\colon \pi_B\to \mathrm{PSL}(2,\Complex_p)$
	is bounded 
	for all but finitely many rational primes $p$ and 
	for every embedding 
	$\sigma\colon\mathrm{PSL}(2,\Rational^\algcl)\to \mathrm{PSL}(2,\Complex_p)$.
	Here $\Complex_p$ denotes the $p$--adic complex field,
	which is the completion of an algebraic closure of $\Rational_p$.
	We claim that the $\Psi$--pullback of $\rho_{B,\sigma}$ restricted to $\pi_A$
	is conjugate to the the induced representation $\rho_{A,\sigma}$
	of some Zariski dense representation
	$\rho_A\colon \pi_A\to\mathrm{PSL}(2,\Rational^\algcl)$,
	whose conjugacy class does not depend on 
	particular choices of $p$ or $\sigma$.
	With this claim,
	$\eta_\Psi$ is constructed as mapping the conjugacy class $[\rho_B]\in\mathcal{Y}_B$ to 
	$[\rho_A]\in\mathcal{Y}_A$.
	
	The claim can be derived quickly from Lemma \ref{profinite_correspondence_rho_pA},
	but we need two basic observations with $\mathrm{PSL}(2,\Rational^\algcl)$--representations
	of aspherical closed $3$--manifold groups 
	(or more generally, finitely generated $2$--$\mathrm{PD}_3$ groups).
	
	First, we observe that every representation 
	$\rho\colon\pi\to \mathrm{PSL}(2,\Rational^\algcl)$
	of an aspherical closed $3$--manifold group $\pi$
	admits a lift $\rho^\sharp\colon \mathrm{Fratt}_2(\pi)\to \mathrm{SL}(2,\Rational^\algcl)$
	restricted to the $2$--Frattini subgroup.
	In fact, the obstruction to lifting a $\mathrm{PSL}(2,\Rational^\algcl)$--representation $\rho$
	to $\mathrm{SL}(2,\Rational^\algcl)$ can be expressed as a second Stiefel--Whitney class
	$w_2(\rho)\in H^2(\pi;\Integral/2\Integral)$, 
	(see \cite[Remark 4.1]{Heusener--Porti_PSL}; compare \cite[Proposition 5.1]{DLSW_rep_vol}).
	Because of the natural homomorphism
	$H_1(\mathrm{Fratt}_2(\pi);\Integral/2\Integral)\to H_1(\pi;\Integral/2\Integral)$
	is obviously zero,
	the natural homomorphism
	$H_2(\pi;\Integral/2\Integral)\to H_2(\mathrm{Fratt}_2(\pi);\Integral/2\Integral)$
	is also zero, by the $\Integral/2\Integral$--Poincar\'e duality.
	So, there are no obstruction to lifting 
	$\rho$ restricted to $\mathrm{Fratt}_2(\pi)$.
	 
	Secondly, we observe that every Zariski dense representation
	$\rho\colon\pi\to \mathrm{PSL}(2,\Rational^\algcl)$
	of an aspherical closed $3$--manifold group $\pi$
	is uniquely determined by its restriction to $\mathrm{Fratt}_2(\pi')$,
	where $\pi'$ is some finite-index subgroup.
	This follows by a similar argument as used in proving 
	Lemma \ref{virtual_Frattini_character}.
	In fact, suppose that $\tilde{\rho}\colon \pi\to \mathrm{PSL}(2,\Rational^\algcl)$
	is a representation that equals $\rho$ restricted to $\mathrm{Fratt}_2(\pi')$.
	We take a lift $\rho^\sharp\colon\mathrm{Fratt}_2(\pi')\to\mathrm{SL}(2,\Rational^\algcl)$
	by the above observation,
	and take $g_1,g_2,g_3\in\mathrm{Fratt}_2(\pi)$ such that 
	$X_i=2\cdot\rho^\sharp(g_i)-\chi_{\rho^\sharp}(g_i)\cdot I$,
	$i=1,2,3$,
	span $\mathfrak{sl}(2,\Rational^\algcl)$ over $\Rational^\algcl$. 
	We may assume that $\pi'$ is normal in $\pi$, 
	possibly after passing to a further subgroup.
	Then we check 
	$\mathrm{Ad}_{\rho(u)}(X_i)=\mathrm{Ad}_{\tilde{\rho}(u)}(X_i)$,
	for $i=1,2,3$ and for all $u\in\pi$,
	the same way as in the proof of Lemma \ref{virtual_Frattini_character}.
	Since the adjoint representation  of $\mathrm{PSL}(2,\Rational^\algcl)$
	is an isomorphism
	onto $\mathrm{Aut}(\mathfrak{sl}(2,\Rational^\algcl))\cong\mathrm{SO}(3,\Rational^\algcl)$,
	we obtain $\rho(u)=\tilde{\rho}(u)$ for all $u\in\pi$.
	
	With the above observations,
	we take a finite-index subgroup $\pi_B'$ of $\mathrm{Fratt}_2(\pi_B)$ 
	assoicated to a fibered finite cover $M_B'\to M_B$ \cite[Theorem 9.2]{Agol_VHC}.
	Obtain the $\Psi$--corresponding finite index subgroup $\pi_A'$ of $\mathrm{Fratt}_2(\pi_A)$,
	so the associated finite cover $M'_A \to M_A$ is also fibered \cite{JZ}.
	By the first observation,
	we can lift $\rho_{B}$ restricted to $\pi_{B}'$
	as $\rho^\sharp_{B}\colon \pi_{B}'\to\mathrm{SL}(2,\Rational^\algcl)$.
	In particular, $\rho^\sharp_{B}$ is Zariski dense,
	and is unique up to a scalar representation factor $\pi_{B}'\to\{\pm1\}$.
	Then the $\Psi$--pullback of 
	$\rho^\sharp_{B,\sigma}\colon\pi_B'\to\mathrm{SL}(2,\Rational^\algcl)\to \mathrm{SL}(2,\Complex_p)$
	gives rise to a Zariski dense representation $\pi_A'\to\mathrm{SL}(2,\Complex_p)$
	with character valued in $\Rational^\algcl$ (Lemma \ref{profinite_correspondence_rho_pA}).
	It follows that the $\Psi$--pullback of
	$\rho_B\colon\pi_B\to\mathrm{PSL}(2,\Rational^\algcl)$
	is a Zariski dense representation 
	$\pi_A\to\mathrm{PSL}(2,\Complex_p)$ 
	with $\mathrm{Ad}$--character valued in $\Rational^\algcl$.
	Therefore, it can be realized with a representation
	$\rho_A\colon \pi_A\to\mathrm{PSL}(2,\Rational^\algcl)$,
	which is unique up to conjugacy.
	Moreover, any lift of $\rho_A$ restricted to $\pi_A'$	
	gives rise to a representation
	$\rho^\sharp_A\colon\pi_A'\to\mathrm{SL}(2,\Rational^\algcl)$
	which $\Psi$--corresponds to 
	$\rho^\sharp_B\colon\pi_B'\to\mathrm{SL}(2,\Rational^\algcl)$
	in the sense of Lemma \ref{profinite_correspondence_rho_pA}.
	Since 
	the $\mathrm{Hom}(\pi_A',\{\pm1\})$--equivalence class of $\rho^{\sharp}_A$ 
	is uniquely determined by the $\mathrm{Hom}(\pi_A',\{\pm1\})$--equivalence class of $\rho^\sharp_B$
	and $\Psi$  (Lemma \ref{profinite_correspondence_rho_pA}),
	we infer from the second observation that 
	the conjugacy class of $\rho_A$ is uniquely determined
	by the conjugacy class of $\rho_B$ and $\Psi$.
	This proves our claim in the construction of $\eta_\Psi$.
	
	The rest assertions in Theorem \ref{main_profinite_correspondence_PSL}
	regarding invariant trace fields and invariant quaternion algebras
	follow immediately from the properties in Lemma \ref{profinite_correspondence_rho_pA}
	and the virtual invariance of those objects (see Corollary \ref{virtual_invar_qa}).
\end{proof}

\begin{theorem}\label{main_profinite_invariance_vol}
	Assume Hypothesis \ref{Borel_hypothesis_p}.	
	If a pair of uniform lattices in $\mathrm{PSL}(2,\Complex)$
	have isomorphic profinite completions,
	then they have identical covolume,
	identical invariant trace fields in $\Complex$,
	and conjugate invariant quaternion algebras in $\mathrm{Mat}(2,\Complex)$.
	Moreover,
	they are either simultaneously arithmetic, 
	or simultaneously non-arithmetic.
\end{theorem}

\begin{proof}
	Possibly after passing to finite-index sublattices,
	we may assume the lattices torsion-free \cite[Chapter 7, \S 7.6, Corollary 4]{Ratcliffe_book}.
	Hence, it is equivalent to argue with the orientable closed hyperbolic $3$--manifolds
	obtained as the quotients of $\Hyp^3$ by the lattices.
		
	Denote by $M_A$ and $M_B$ the orientable closed hyperbolic $3$--manifolds.
	Suppose that $\Psi\colon\widehat{\pi}_A\to\widehat{\pi}_B$
	is an isomorphism between the profinite completions
	of their fundamental groups $\pi_A$ and $\pi_B$.
	Denote by $\Rational^\algcl$ the algebraic closure of $\Rational$ in $\Complex$.
	Take a $\Psi$--corresponding pair of finite-index subgroups $\pi'_A$ and $\pi'_B$
	associated to fibered finite covers $M'_A\to M_A$ and $M'_B\to M_B$, respectively.
	By Lemma \ref{profinite_invariance_vol_pA} (based on Hypothesis \ref{Borel_hypothesis_p})
	and Example \ref{hyp_manifold_rep_vol},
	we infer that $M'_A$ and $M'_B$ have identical hyperbolic volume,
	and moreover, any discrete faithful representation 
	$\rho'_B\colon \pi'_B\to\mathrm{SL}(2,\Rational^\algcl)\to\mathrm{SL}(2,\Complex)$
	corresponds to a discrete faithful representation 
	$\rho'_A\colon \pi'_A\to\mathrm{SL}(2,\Rational^\algcl)\to\mathrm{SL}(2,\Complex)$,
	up to $\mathrm{Hom}(\pi'_{A/B},\{\pm1\})$--equivalence,
	in the sense of Lemma \ref{profinite_correspondence_rho_pA}.
	It follows that $M_A$ and $M_B$ have identical hyperbolic volume,
	identical invariant trace fields in $\Complex$,
	and conjugate invariant quaternion algebras in $\mathrm{Mat}(2,\Complex)$
	(Lemma \ref{profinite_correspondence_rho_pA} and Corollary \ref{virtual_invar_qa}).
	
	If $M_B$ is arithmetic, $\rho'_B$ restricted to $\mathrm{Fratt}_2(\pi'_B)$
	is bounded at all the primes in the invariant trace field $k(M_B)=k(M_A)$,
	so $\rho'_A$ restricted to $\mathrm{Fratt}_2(\pi'_A)$ 
	also has the same property (Lemma \ref{profinite_correspondence_rho_pA}).
	This shows that the discrete faithful representation $\rho'_A$ 
	has integral trace for all $g\in\pi'_A$.
	Since the invariant trace field $k(M_A)=k(M_B)$ has exactly one complex place,
	and the invariant quaternion algebra $\mathscr{B}(M_A)\cong\mathscr{B}(M_B)$
	ramifies at all the real places of $k(M_A)=k(M_B)$,
	we infer that $M_A$ is also arithmetic \cite[Chapter 8, Theorem 8.3.2]{Maclachlan--Reid_book}.
	By symmetry, $M_A$ is non-arithmetic if $M_B$ is non-arithmetic.	
	To summarize,
	$M_A$ and $M_B$ are either simultaneously arithmetic, 
	or simultaneously non-arithmetic.
\end{proof}

\bibliographystyle{amsalpha}


\end{document}